\pdfoutput=1
\documentclass[12pt]{amsart}
\usepackage[utf8]{inputenc}
\usepackage[english]{babel}
\usepackage{amsmath,amssymb, mathtools, MnSymbol}
\usepackage{tikz, tikz-cd}
\usetikzlibrary{decorations.pathreplacing, calc,intersections,fit, matrix, positioning}
\usepackage{tkz-euclide}
\usetkzobj{all}
\usepackage{wrapfig}
\usepackage{csquotes}
\usepackage{multicol}
\usepackage{cutwin}

\usepackage[backend=biber,style=alphabetic, maxnames=5]{biblatex}
\renewbibmacro{in:}{\ifentrytype{article}{}{\printtext{\bibstring{in}\intitlepunct}}}
\usepackage{enumitem}
\setlist[enumerate]{nosep, label=(\arabic*)}

\usepackage{hyperref}
\hypersetup{
    pdftitle={Categories of Two-Colored Pair Partitions, Part II},
    pdfauthor={Alexander Mang and Moritz Weber},
    pdfsubject={Categories of Two-Colored Pair Partitions},
    pdfkeywords={quantum groups, unitary easy quantum groups, unitary group,  half-liberation, tensor category, two-colored partitions, partitions of sets, categories of partitions, Brauer algebra}
}
\usepackage{xpatch}
\xpatchcmd{\proof}{\itshape}{\normalfont\proofnamefont}{}{}

\newcommand{\proofnamefont}{\scshape}

\theoremstyle{plain}

\newtheorem*{ucorollary}{Corollary}
\newtheorem{main}{Main Theorem}
\newtheorem{theorem}{Theorem}[section]
\newtheorem{proposition}[theorem]{Proposition}
\newtheorem{lemma}[theorem]{Lemma}
\newtheorem{corollary}[theorem]{Corollary}

\theoremstyle{definition}
\newtheorem{definition}[theorem]{Definition}
\newtheorem{remark}[theorem]{Remark}

\numberwithin{equation}{section}

\DeclareMathOperator{\inte}{int}

\newcommand{\mc}{\mathcal}
\newcommand{\mr}{\mathrm}
\newcommand{\eqpd}{\coloneqq}

\newcommand{\nnint}{{\mathbb{N}_0}}
\newcommand{\pint}{\mathbb{N}}
\newcommand{\integers}{\mathbb{Z}}
\newcommand{\Cp}{\mathcal{P}^{\circ\bullet}}

\newcommand{\Cppnb}{\mathcal{P}^{\circ\bullet}_{2,\mathrm{nb}}}
\newcommand{\resbr}{\mathcal{B}_{\mathrm{res}}}

\newcommand{\trmw}{\mathrm{WIn}}
\newcommand{\trms}{\mathrm{SIn}}

\newcommand{\colors}{\{\circ,\bullet\}}
\newcommand{\brp}{\mathrm{Br}}
\newcommand{\minbr}{\mathcal{B}_{\mathrm{min}}}
\newcommand{\idpt}[1]{\mathrm{Id}({#1})}
\colorlet{lightgray}{black!40}
\colorlet{gray}{black!60}
\colorlet{darkgray}{black!80}

\newcommand{\PartBracketBBWW}{%
  \begin{tikzpicture}[scale=0.25,baseline=0.04cm]
    \def\xdist{0.666}
    \def\ydist{1.25}
    \def\ypar{0.5}
    \node [scale=0.33, circle, draw=black, fill=black] (a1) at ({0*\xdist},{0*\ydist}) {};
    \node [scale=0.33, circle, draw=black, fill=black] (a2) at ({1*\xdist},{0*\ydist}) {};
    \node [scale=0.33, circle, draw=black, fill=white] (a3) at ({2*\xdist},{0*\ydist}) {};
    \node [scale=0.33, circle, draw=black, fill=white] (a4) at ({3*\xdist},{0*\ydist}) {};
    \node [scale=0.33, circle, draw=black, fill=black] (b1) at ({0*\xdist},{1*\ydist}) {};
    \node [scale=0.33, circle, draw=black, fill=black] (b2) at ({1*\xdist},{1*\ydist}) {};
    \node [scale=0.33, circle, draw=black, fill=white] (b3) at ({2*\xdist},{1*\ydist}) {};
    \node [scale=0.33, circle, draw=black, fill=white] (b4) at ({3*\xdist},{1*\ydist}) {};
    \draw (a1) -- ++ (0,{\ypar}) -| (a4);
    \draw (b1) -- ++ (0,{-\ypar}) -| (b4);
    \draw (a2) to (b2);
    \draw (a3) to (b3);    
  \end{tikzpicture}
}

\newcommand{\PartBracketWWBB}{%
  \begin{tikzpicture}[scale=0.25,baseline=0.04cm]
    \def\xdist{0.666}
    \def\ydist{1.25}
    \def\ypar{0.5}
    \node [scale=0.33, circle, draw=black, fill=white] (a1) at ({0*\xdist},{0*\ydist}) {};
    \node [scale=0.33, circle, draw=black, fill=white] (a2) at ({1*\xdist},{0*\ydist}) {};
    \node [scale=0.33, circle, draw=black, fill=black] (a3) at ({2*\xdist},{0*\ydist}) {};
    \node [scale=0.33, circle, draw=black, fill=black] (a4) at ({3*\xdist},{0*\ydist}) {};
    \node [scale=0.33, circle, draw=black, fill=white] (b1) at ({0*\xdist},{1*\ydist}) {};
    \node [scale=0.33, circle, draw=black, fill=white] (b2) at ({1*\xdist},{1*\ydist}) {};
    \node [scale=0.33, circle, draw=black, fill=black] (b3) at ({2*\xdist},{1*\ydist}) {};
    \node [scale=0.33, circle, draw=black, fill=black] (b4) at ({3*\xdist},{1*\ydist}) {};
    \draw (a1) -- ++ (0,{\ypar}) -| (a4);
    \draw (b1) -- ++ (0,{-\ypar}) -| (b4);
    \draw (a2) to (b2);
    \draw (a3) to (b3);    
  \end{tikzpicture}
}

\newcommand{\PartBracketBWwW}{%
  \begin{tikzpicture}[scale=0.25,baseline=0.04cm]
    \def\xdist{0.666}
    \def\ydist{1.25}
    \def\ypar{0.5}
    \node [scale=0.33, circle, draw=black, fill=black] (a1) at ({0*\xdist},{0*\ydist}) {};
    \node [scale=0.33, circle, draw=black, fill=white] (a2) at ({1*\xdist},{0*\ydist}) {};
    \node [scale=0.33, circle, draw=black, fill=white] (a3) at ({4.2*\xdist},{0*\ydist}) {};
    \node [scale=0.33, circle, draw=black, fill=black] (b1) at ({0*\xdist},{1*\ydist}) {};
    \node [scale=0.33, circle, draw=black, fill=white] (b2) at ({1*\xdist},{1*\ydist}) {};
    \node [scale=0.33, circle, draw=black, fill=white] (b3) at ({4.2*\xdist},{1*\ydist}) {};
    \node [scale=0.55] (l1) at ({2.5*\xdist},{1*\ydist}) {$\otimes w$ };
    \draw (a1) -- ++ (0,{\ypar}) -| (a3);
    \draw (b1) -- ++ (0,{-\ypar}) -| (b3);
    \draw (a2) to (b2);
  \end{tikzpicture}
}

\newcommand{\PartPermWWwW}{%
  \begin{tikzpicture}[scale=0.25,baseline=0.04cm]
    \def\xdist{0.666}
    \def\ydist{1.25}
    \def\ypar{0.5}
    \node [scale=0.33, circle, draw=black, fill=white] (a1) at ({0*\xdist},{0*\ydist}) {};
    \node [scale=0.33, circle, draw=black, fill=white] (a2) at ({1.55*\xdist},{-0.075*\ydist}) {};
    \node [scale=0.33, circle, draw=black, fill=white] (a3) at ({6.2*\xdist},{0*\ydist}) {};
    \node [scale=0.33, circle, draw=black, fill=white] (b1) at ({0*\xdist},{0.825*\ydist}) {};
    \node [scale=0.33, circle, draw=black, fill=white] (b2) at ({1.55*\xdist},{0.9*\ydist}) {};
    \node [scale=0.33, circle, draw=black, fill=white] (b3) at ({6.2*\xdist},{0.825*\ydist}) {};
    \node [scale=0.55] (l1) at ({3.6*\xdist},{1.0*\ydist}) {$\otimes w\hspace{-3pt}-\hspace{-4pt}1$ };
    \draw (a1) -- (b3);
    \draw (b1) -- (a3);
    \draw (a2) to (b2);
  \end{tikzpicture}
}

\newcommand{\PartBracketBWBW}{%
  \begin{tikzpicture}[scale=0.25,baseline=0.04cm]
    \def\xdist{0.666}
    \def\ydist{1.25}
    \def\ypar{0.5}
    \node [scale=0.33, circle, draw=black, fill=black] (a1) at ({0*\xdist},{0*\ydist}) {};
    \node [scale=0.33, circle, draw=black, fill=white] (a2) at ({1*\xdist},{0*\ydist}) {};
    \node [scale=0.33, circle, draw=black, fill=black] (a3) at ({2*\xdist},{0*\ydist}) {};
    \node [scale=0.33, circle, draw=black, fill=white] (a4) at ({3*\xdist},{0*\ydist}) {};
    \node [scale=0.33, circle, draw=black, fill=black] (b1) at ({0*\xdist},{1*\ydist}) {};
    \node [scale=0.33, circle, draw=black, fill=white] (b2) at ({1*\xdist},{1*\ydist}) {};
    \node [scale=0.33, circle, draw=black, fill=black] (b3) at ({2*\xdist},{1*\ydist}) {};
    \node [scale=0.33, circle, draw=black, fill=white] (b4) at ({3*\xdist},{1*\ydist}) {};
    \draw (a1) -- ++ (0,{\ypar}) -| (a4);
    \draw (b1) -- ++ (0,{-\ypar}) -| (b4);
    \draw (a2) to (b2);
    \draw (a3) to (b3);    
  \end{tikzpicture}
}

\newcommand{\PartBracketWBWB}{%
  \begin{tikzpicture}[scale=0.25,baseline=0.04cm]
    \def\xdist{0.666}
    \def\ydist{1.25}
    \def\ypar{0.5}
    \node [scale=0.33, circle, draw=black, fill=white] (a1) at ({0*\xdist},{0*\ydist}) {};
    \node [scale=0.33, circle, draw=black, fill=black] (a2) at ({1*\xdist},{0*\ydist}) {};
    \node [scale=0.33, circle, draw=black, fill=white] (a3) at ({2*\xdist},{0*\ydist}) {};
    \node [scale=0.33, circle, draw=black, fill=black] (a4) at ({3*\xdist},{0*\ydist}) {};
    \node [scale=0.33, circle, draw=black, fill=white] (b1) at ({0*\xdist},{1*\ydist}) {};
    \node [scale=0.33, circle, draw=black, fill=black] (b2) at ({1*\xdist},{1*\ydist}) {};
    \node [scale=0.33, circle, draw=black, fill=white] (b3) at ({2*\xdist},{1*\ydist}) {};
    \node [scale=0.33, circle, draw=black, fill=black] (b4) at ({3*\xdist},{1*\ydist}) {};
    \draw (a1) -- ++ (0,{\ypar}) -| (a4);
    \draw (b1) -- ++ (0,{-\ypar}) -| (b4);
    \draw (a2) to (b2);
    \draw (a3) to (b3);    
  \end{tikzpicture}
}

\newcommand{\PartPermWBvWvW}{%
  \begin{tikzpicture}[scale=0.25,baseline=0.04cm]
    \def\xdist{0.666}
    \def\ydist{1.25}
    \def\ypar{0.5}
    \node [scale=0.33, circle, draw=black, fill=white] (a1) at ({0*\xdist},{0*\ydist}) {};
    \node [scale=0.33, circle, draw=black, fill=black] (a2) at ({1.05*\xdist},{-0.075*\ydist}) {};
    \node [scale=0.33, circle, draw=black, fill=white] (a3) at ({3.9*\xdist},{-0.075*\ydist}) {};
    \node [scale=0.33, circle, draw=black, fill=white] (a4) at ({6.65*\xdist},{0*\ydist}) {};
    \node [scale=0.33, circle, draw=black, fill=white] (b1) at ({0*\xdist},{0.825*\ydist}) {};
    \node [scale=0.33, circle, draw=black, fill=black] (b2) at ({1.05*\xdist},{0.9*\ydist}) {};
    \node [scale=0.33, circle, draw=black, fill=white] (b3) at ({3.9*\xdist},{0.9*\ydist}) {};
    \node [scale=0.33, circle, draw=black, fill=white] (b4) at ({6.65*\xdist},{0.825*\ydist}) {};
    \node [scale=0.55] (l1) at ({2.3*\xdist},{1*\ydist}) {$\otimes v$ };
    \node [scale=0.55] (l2) at ({5.15*\xdist},{1*\ydist}) {$\otimes v$ };        
    \draw (a1) -- (b4);
    \draw (b1) -- (a4);
    \draw (a2) to (b2);
    \draw (a3) to (b3);    
  \end{tikzpicture}
}

\newcommand{\PartHalfLibWBW}{%
  \begin{tikzpicture}[scale=0.25,baseline=0.04cm]
    \def\xdist{0.75}
    \def\ydist{1.25}
    \def\ypar{0.5}
    \node [scale=0.33, circle, draw=black, fill=white] (a1) at ({0*\xdist},{0*\ydist}) {};
    \node [scale=0.33, circle, draw=black, fill=black] (a2) at ({1*\xdist},{0*\ydist}) {};
    \node [scale=0.33, circle, draw=black, fill=white] (a3) at ({2*\xdist},{0*\ydist}) {};
    \node [scale=0.33, circle, draw=black, fill=white] (b1) at ({0*\xdist},{1*\ydist}) {};
    \node [scale=0.33, circle, draw=black, fill=black] (b2) at ({1*\xdist},{1*\ydist}) {};
    \node [scale=0.33, circle, draw=black, fill=white] (b3) at ({2*\xdist},{1*\ydist}) {};
    \draw (a1) to (b3);    
    \draw (a2) to (b2);
    \draw (a3) to (b1);    
  \end{tikzpicture}
}

\newcommand{\PartHalfLibBWB}{%
  \begin{tikzpicture}[scale=0.25,baseline=0.04cm]
    \def\xdist{0.75}
    \def\ydist{1.25}
    \def\ypar{0.5}
    \node [scale=0.33, circle, draw=black, fill=black] (a1) at ({0*\xdist},{0*\ydist}) {};
    \node [scale=0.33, circle, draw=black, fill=white] (a2) at ({1*\xdist},{0*\ydist}) {};
    \node [scale=0.33, circle, draw=black, fill=black] (a3) at ({2*\xdist},{0*\ydist}) {};
    \node [scale=0.33, circle, draw=black, fill=black] (b1) at ({0*\xdist},{1*\ydist}) {};
    \node [scale=0.33, circle, draw=black, fill=white] (b2) at ({1*\xdist},{1*\ydist}) {};
    \node [scale=0.33, circle, draw=black, fill=black] (b3) at ({2*\xdist},{1*\ydist}) {};
    \draw (a1) to (b3);    
    \draw (a2) to (b2);
    \draw (a3) to (b1);    
  \end{tikzpicture}
  }

  \newcommand{\PartHalfLibWWW}{%
  \begin{tikzpicture}[scale=0.25,baseline=0.04cm]
    \def\xdist{0.75}
    \def\ydist{1.25}
    \def\ypar{0.5}
    \node [scale=0.33, circle, draw=black, fill=white] (a1) at ({0*\xdist},{0*\ydist}) {};
    \node [scale=0.33, circle, draw=black, fill=white] (a2) at ({1*\xdist},{0*\ydist}) {};
    \node [scale=0.33, circle, draw=black, fill=white] (a3) at ({2*\xdist},{0*\ydist}) {};
    \node [scale=0.33, circle, draw=black, fill=white] (b1) at ({0*\xdist},{1*\ydist}) {};
    \node [scale=0.33, circle, draw=black, fill=white] (b2) at ({1*\xdist},{1*\ydist}) {};
    \node [scale=0.33, circle, draw=black, fill=white] (b3) at ({2*\xdist},{1*\ydist}) {};
    \draw (a1) to (b3);    
    \draw (a2) to (b2);
    \draw (a3) to (b1);    
  \end{tikzpicture}
}

\newcommand{\PartCrossWW}{%
  \begin{tikzpicture}[scale=0.25,baseline=0.015cm]
    \def\xdist{1}
    \def\ydist{1}
    \node [scale=0.33, circle, draw=black, fill=white] (a1) at ({0*\xdist},{0*\ydist}) {};
    \node [scale=0.33, circle, draw=black, fill=white] (a2) at ({1*\xdist},{0*\ydist}) {};
    \node [scale=0.33, circle, draw=black, fill=white] (b1) at ({0*\xdist},{1*\ydist}) {};
    \node [scale=0.33, circle, draw=black, fill=white] (b2) at ({1*\xdist},{1*\ydist}) {};
    \draw (a1) to (b2);    
    \draw (a2) to (b1);
  \end{tikzpicture}
  }

  \newcommand{\PartIdenB}{%
  \begin{tikzpicture}[scale=0.25,baseline=0.015cm]
    \def\xdist{1}
    \def\ydist{1}
    \node [scale=0.33, circle, draw=black, fill=black] (a1) at ({0*\xdist},{0*\ydist}) {};
    \node [scale=0.33, circle, draw=black, fill=black] (b1) at ({0*\xdist},{1*\ydist}) {};
    \draw (a1) to (b1);    
  \end{tikzpicture}
  }

  \newcommand{\PartIdenW}{%
  \begin{tikzpicture}[scale=0.25,baseline=0.015cm]
    \def\xdist{1}
    \def\ydist{1}
    \node [scale=0.33, circle, draw=black, fill=white] (a1) at ({0*\xdist},{0*\ydist}) {};
    \node [scale=0.33, circle, draw=black, fill=white] (b1) at ({0*\xdist},{1*\ydist}) {};
    \draw (a1) to (b1);    
  \end{tikzpicture}
  }
  
\newcommand{\PartIdenLoBW}{%
  \begin{tikzpicture}[scale=0.25,baseline=-0.015cm]
    \def\xdist{1}
    \def\ydist{1}
    \def\ypar{1}
    \node [scale=0.33, circle, draw=black, fill=black] (a1) at ({0*\xdist},{0*\ydist}) {};
    \node [scale=0.33, circle, draw=black, fill=white] (a2) at ({1*\xdist},{0*\ydist}) {};
    \draw (a1) -- ++ (0,{\ypar}) -| (a2);    
  \end{tikzpicture}
  }
\newcommand{\PartIdenLoWB}{%
  \begin{tikzpicture}[scale=0.25,baseline=-0.015cm]
    \def\xdist{1}
    \def\ydist{1}
    \def\ypar{1}
    \node [scale=0.33, circle, draw=black, fill=white] (a1) at ({0*\xdist},{0*\ydist}) {};
    \node [scale=0.33, circle, draw=black, fill=black] (a2) at ({1*\xdist},{0*\ydist}) {};
    \draw (a1) -- ++ (0,{\ypar}) -| (a2);    
  \end{tikzpicture}
  }
  \newcommand{\PartIdenLoBB}{%
  \begin{tikzpicture}[scale=0.25,baseline=-0.015cm]
    \def\xdist{1}
    \def\ydist{1}
    \def\ypar{1}
    \node [scale=0.33, circle, draw=black, fill=black] (a1) at ({0*\xdist},{0*\ydist}) {};
    \node [scale=0.33, circle, draw=black, fill=black] (a2) at ({1*\xdist},{0*\ydist}) {};
    \draw (a1) -- ++ (0,{\ypar}) -| (a2);    
  \end{tikzpicture}
}
  \newcommand{\PartIdenLoWW}{%
  \begin{tikzpicture}[scale=0.25,baseline=-0.015cm]
    \def\xdist{1}
    \def\ydist{1}
    \def\ypar{1}
    \node [scale=0.33, circle, draw=black, fill=white] (a1) at ({0*\xdist},{0*\ydist}) {};
    \node [scale=0.33, circle, draw=black, fill=white] (a2) at ({1*\xdist},{0*\ydist}) {};
    \draw (a1) -- ++ (0,{\ypar}) -| (a2);    
  \end{tikzpicture}
  }
  
\begin{document}
\title[Categories of Two-Colored Pair Partitions II]{Categories of Two-Colored Pair Partitions\\Part~II: Categories Indexed By Semigroups} 
\author{Alexander Mang}
\author{Moritz Weber}
\address{Saarland University, Fachbereich Mathematik, 
	66041 Saarbrücken, Germany}
\email{s9almang@stud.uni-saarland.de, weber@math.uni-sb.de}
\thanks{The second author was supported by the ERC Advanced Grant NCDFP, held
by Roland Spei\-cher, by the SFB-TRR 195, and by the DFG project \emph{Quantenautomorphismen von Graphen}. This work was part of the first author's Master's thesis.}

\date{\today}

\begin{abstract}
  Within the framework of unitary easy quantum groups, we study an analogue of Brauer's Schur-Weyl approach to the representation theory of the orthogonal group. We consider concrete combinatorial categories whose morphisms are formed by partitions of finite sets into disjoint subsets of cardinality two; the points of these sets are colored black or white. These categories correspond to \enquote{half-liberated easy} interpolations between the unitary group and Wang's quantum counterpart. We complete the classification of all such categories demonstrating that the subcategories of a certain natural halfway point are equivalent to additive subsemigroups of the natural numbers; the categories above this halfway point have been classified in a preceding article. We achieve this using combinatorial means exclusively. Our work reveals that the half-liberation procedure is quite different from what was previously known from the orthogonal case.

\end{abstract}
\maketitle

\section*{Introduction}
\label{section:introduction}

Given partitions of an upper and a lower finite sequence of two-colored points into disjoint sets, one can create new partitions of this kind by vertical and horizontal concatenation as well as exchanging the roles of upper and lower row. In this article we proceed with classifying the sets of partitions invariant under these operations. They resemble the structures introduced by Brauer \cite{Br37} in order to study the representation theory of the orthogonal group; see Section~\ref{subsection:categories}. The classification program of such sets of partitions was begun in \cite{TaWe15a} and since continued in \cite{MaWe18a} and \cite{Gr18}.  See also \cite{Fr17}. Such categories are of importance in Banica and Speicher's path (\cite{BaSp09}, \cite{We16} and \cite{We17a}) towards compact quantum groups  in Woronowicz's sence (\cite{Wo87}, \cite{Wo88}, \cite{Wo91} and \cite{Wo98}). However, we use combinatorial means exclusively. The quantum-algebraic implications of the combinatorial result are discussed in Section~\ref{section:concluding_remarks}.\par
We deal with subcategories of a specified category $\mathcal P^{\circ\bullet}_{2,\mathrm{nb}}$ of pair partitions which in addition conform with a certain rule on their coloration. In \cite{MaWe18a}, we showed that a subcategory $\mathcal S_0$ of  $\mathcal P^{\circ\bullet}_{2,\mathrm{nb}}$ exists such that for every category $\mathcal C\subseteq \mathcal P^{\circ\bullet}_{2,\mathrm{nb}}$ holds $\mathcal S_0\subseteq \mathcal C$ or  $\mathcal C\subseteq \mathcal S_0$. All the categories satisfying the latter condition we determined in \cite{MaWe18a}. In the present article we address the subcategories of $\mathcal S_0$ and classify them. For a short introduction to two-colored partitions, see Section~\ref{section:preliminaries}.
\vspace{-0.25em}
\section{Main Results}
\vspace{-0.15em}
\label{section:main_result}
We define and characterize a class of categories of two-colored partitions equivalent to the additive subsemigroups of $\nnint$.
\vspace{-0.15em}
\begin{main}
 \label{theorem:main_1}
  \begin{enumerate}[label=(\alph*), labelwidth=!]
  \item  \label{item:main_1-1}
    For each subsemigroup $D$ of $(\mathbb N_0,+)$, in short: $D\in \mc D$, a category of two-colored partitions is given by the set $\mathcal I_D$ of all two-colored pair partitions with the following properties satisfied when the partition is rotated to one line:
    \begin{enumerate}
    \item \label{item:main_1-1-1}
      Each block contains one point each of every color.
    \item \label{item:main_1-1-2}
      Between the two legs of any block lie as many black points as white ones.
    \item \label{item:main_1-1-3}
      Two points of opposite color may not belong to crossing blocks if the following condition is met: The difference in the numbers of black and white points between them amounts to an element of $D$.
    \end{enumerate}
  \item \label{item:main_1-2a} The categories $(\mc I_D)_{D\in \mc D}$ are pairwise distinct.
\item \label{item:main_1-2} For all $D,D'\in \mc D$ holds
  \begin{align*}
    D\subseteq D'\implies \mc I_D\supseteq \mc I_{D'}.
  \end{align*}
  In particular, \(\mc I_\nnint\subseteq \mc I_D\subseteq \mc I_\emptyset\) holds for every $D\in\mc D$.    
    \item  \label{item:main_1-3}
      For every finite subset $w$ of $\mathbb N$ (with $0\notin \pint$), define:
      \vspace{-0.25em}
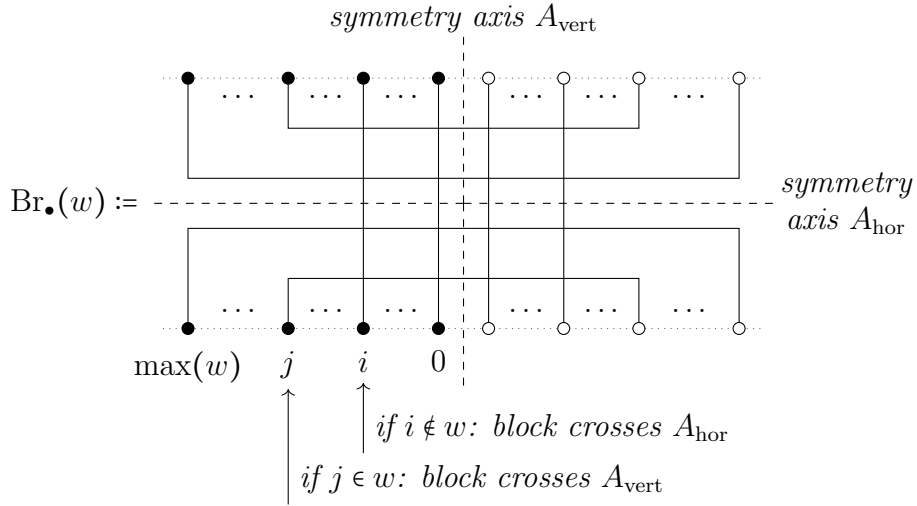
  \begin{center}
    $\mathrm{Br}_\bullet(w)\coloneq$\hspace{-10pt}
    \begin{tikzpicture}[scale=0.666,baseline=2.35cm*0.666]
      \draw[dotted, gray] (-0.5,0) -- (11.5,0);
      \draw[dotted, gray] (-0.5,5) -- (11.5,5);
      \node[circle, scale=0.4, draw=black, fill=black] (x1) at (0,0) {};
      \node[circle, scale=0.4, draw=black, fill=white] (x2) at (11,0) {};
      \node[circle, scale=0.4, draw=black, fill=black] (x3) at (0,5) {};
      \node[circle, scale=0.4, draw=black, fill=white] (x4) at (11,5) {};
      \node[circle, scale=0.4, draw=black, fill=black] (y1) at (5,0) {};
      \node[circle, scale=0.4, draw=black, fill=white] (y2) at (6,0) {};
      \node[circle, scale=0.4, draw=black, fill=black] (y3) at (5,5) {};
      \node[circle, scale=0.4, draw=black, fill=white] (y4) at (6,5) {};
      \node[circle, scale=0.4, draw=black, fill=black] (z1) at (2,0) {};
      \node[circle, scale=0.4, draw=black, fill=white] (z2) at (9,0) {};
      \node[circle, scale=0.4, draw=black, fill=black] (z3) at (2,5) {};
      \node[circle, scale=0.4, draw=black, fill=white] (z4) at (9,5) {};
      \node[circle, scale=0.4, draw=black, fill=black] (w1) at (3.5,0) {};
      \node[circle, scale=0.4, draw=black, fill=white] (w2) at (7.5,0) {};
      \node[circle, scale=0.4, draw=black, fill=black] (w3) at (3.5,5) {};
      \node[circle, scale=0.4, draw=black, fill=white] (w4) at (7.5,5) {};                  
      \draw (x1) -- ++(0,2) -| (x2);
      \draw (x3) -- ++(0,-2) -| (x4);
      \draw (y1) to (y3);
      \draw (y2) to (y4);
      \draw (z1) -- ++(0,1) -| (z2);
      \draw (z3) -- ++(0,-1) -| (z4);
      \draw (w1) to (w3);
      \draw (w2) to (w4);
      \draw[dashed] (5.5,2.5) -- (5.5,-1.2);
      \draw[dashed] (5.5,2.5) -- (5.5,5.9);
      \draw[dashed] (5.5,2.5) -- (-0.8,2.5);
      \draw[dashed] (5.5,2.5) -- (11.8,2.5);      
      \node at (5.5,6.2) {symmetry axis $A_{\mathrm{vert}}$};
      \node[align=center, right] at (11.6,2.5) {symmetry\\ axis $A_{\mathrm{hor}}$};      
      \node [below = 2pt of x1] {$\max(w)$};
      \node [below = 2pt of z1] (u1) {$j$};
      \node [below = 2pt of w1] (u2) {$i$};
      \node [below = 2pt of y1] {$0$};
      \path (x1)  -- node [pos=0.5, above=2pt] {$\ldots$} (z1);
      \path (z1)  -- node [pos=0.5, above=2pt] {$\ldots$} (w1);
      \path (w1)  -- node [pos=0.5, above=2pt] {$\ldots$} (y1);
      \path (x3)  -- node [pos=0.5, below=2pt] {$\ldots$} (z3);
      \path (z3)  -- node [pos=0.5, below=2pt] {$\ldots$} (w3);
      \path (w3)  -- node [pos=0.5, below=2pt] {$\ldots$} (y3);
      \path (x2)  -- node [pos=0.5, above=2pt] {$\ldots$} (z2);
      \path (z2)  -- node [pos=0.5, above=2pt] {$\ldots$} (w2);
      \path (w2)  -- node [pos=0.5, above=2pt] {$\ldots$} (y2);
      \path (x4)  -- node [pos=0.5, below=2pt] {$\ldots$} (z4);
      \path (z4)  -- node [pos=0.5, below=2pt] {$\ldots$} (w4);
      \path (w4)  -- node [pos=0.5, below=2pt] {$\ldots$} (y4);
      \node  (l1) [below = 1.2cm of u1, right] {if $j\in w$: block crosses $A_{\mathrm{vert}}$};
      \draw[->] (l1.south west) -- (u1.south);
      \node  (l2) [below = 0.6cm of u2, right] {if $i\notin w$: block crosses $A_{\mathrm{hor}}$};
       \draw[->] (l2.south west) -- (u2.south);
    \end{tikzpicture}
  \end{center}
  For every $D\in \mc D$, the category $\mathcal I_D$ is generated
  \begin{itemize}
  \item [$\filledsquare$] by the partition $\mathrm{Br}_\bullet(\mathbb N\backslash D)$ if $\pint\backslash D$ is finite,
  \item [$\filledsquare$] by the partitions $\{\mathrm{Br}_\bullet(\{1,\ldots,v\}\backslash D)\mid v\in \pint\}$ if $\pint\backslash D$ is infinite 
    \item [$\filledsquare$] and, in both cases, in addition to that, by $\PartHalfLibWBW$ if $0\notin D$.
  \end{itemize}
     \end{enumerate}
  \end{main}
At the same time, we show that these categories are in fact all categories of two-colored partitions that exist below the unique maximal element of the family.
  \begin{main}
    \label{theorem:main_2}
For every category $\mc C$ with\[\langle \emptyset\rangle\subseteq \mc C\subseteq \mc I_\emptyset,\] there exists $D\in \mc D$ such that $\mc C=\mc I_D$. In particular, $\langle\emptyset\rangle=\mc I_\nnint$.
\end{main}
These two theorems warrant a corollary when combined with the results of \cite{MaWe18a}.\par
  For every $w\in \nnint$, denote by $\mc S_w$ the set given by all two-colored pair partitions with the following properties satisfied once the partition is rotated to one line:
      \begin{enumerate}
    \item \label{item:main_1-1-1}
      Each block contains one point each of every color.
    \item \label{item:main_1-1-2}
      Between the two legs of any block the difference in the numbers of black and white points is a multiple of $w$.
    \end{enumerate}
    In \cite{MaWe18a} the following facts were shown:
    The sets $(\mc S_w)_{w\in \nnint}$ are pairwise distinct categories of two-colored partitions with
    \begin{align*}
      \mc S_0=\bigcap_{w'\in\pint}\mc S_{w'}\subseteq \mc S_w\subseteq \mc S_1=\langle \PartCrossWW\rangle
    \end{align*}
    for every $w\in \nnint$ and with
    \begin{align*}
      w\integers\subseteq w'\integers \implies \mc S_w\subseteq \mc S_{w'}
    \end{align*}
    for all $w,w'\in \nnint$.
    For every $w\in \pint$, the category $\mc S_w$ is generated by
    
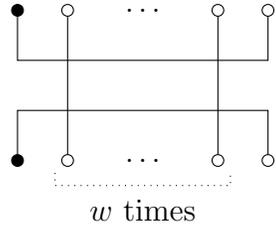
\begin{gather*}
      \begin{tikzpicture}[scale=0.666, baseline=0]
\node [scale=0.4, fill=black, draw=black, circle] (x1) at (0,0) {};
\node [scale=0.4, fill=black, draw=black, circle] (x3) at (0,3) {};
\node [scale=0.4, fill=white, draw=black, circle] (y1) at (1,0) {};		
\node [scale=0.4, fill=white, draw=black, circle] (y2) at (1,3) {};				
\node [scale=0.4, fill=white, draw=black, circle] (z1) at (4,0) {};		
\node [scale=0.4, fill=white, draw=black, circle] (z2) at (4,3) {};					
\node [scale=0.4, fill=white, draw=black, circle] (x2) at (5,0) {};
\node [scale=0.4, fill=white, draw=black, circle] (x4) at (5,3) {};				
\draw (x1) -- ++(0,1) -| (x2);
\draw (x3) -- ++(0,-1) -| (x4);		
\draw (y1) -- (y2);
\draw (z1) -- (z2);	
\node at (2.5,0) {$\ldots$};
\node at (2.5,3) {$\ldots$};
\draw [dotted] (0.75,-0.25) -- ++(0,-0.25) -- (4.25,-0.5) -- ++ (0,0.25);
\node at (2.5,-1) {$w$ times};
\end{tikzpicture}.
\end{gather*}	
The category $\mc S_0$ is cumulatively generated by the partitions
\begin{gather*}
 \PartHalfLibWBW\quad\text{and}\quad\mathrm{Br}_\bullet(\{v\}) \text{ for all } v\in \pint. 
\end{gather*}
\par
We combine the results from \cite{MaWe18a} with the above theorems, yielding a full classification.

\begin{ucorollary}
    \label{corollary:main_3}
    For every category $\mc C$ with
    \begin{align*}
      \langle \emptyset\rangle\subseteq \mc C\subseteq \langle \PartCrossWW\rangle
    \end{align*}
    there exist either $D\in\mc D$ such that $\mc C=\mc I_D$ or $w\in \nnint$ such that $\mc C=\mc S_w$.
 The categories gathered in the two families $(\mc I_D)_{D\in\mc D}$ and $(\mc S_w)_{w\in \pint}$ are all pairwise distinct and for every $D\in \mc D$ and all $w\in \nnint$ holds
    \begin{align*}
      \langle\emptyset\rangle=\mc I_\nnint\subseteq \mc I_D\subseteq \mc I_\emptyset=\mc S_0\subseteq \mc S_w\subseteq \mc S_1=\langle\PartCrossWW\rangle.
    \end{align*}
\end{ucorollary}
This classifies all unitary half-liberations, i.e.\ all easy quantum groups $G$ with $U_n\subseteq G\subseteq U_n^+$. See Section~\ref{section:concluding_remarks} for a discussion of this quantum group context and for other implications of these results.
\par
\vfill
\pagebreak
\section{Reminder on Two-Colored Partitions and their Categories}
\label{section:preliminaries}
For a more detailed introduction to two-colored partitions and their categories, confer \cite{TaWe15a}, and, more specifically for this article, for a treatment of two-colored partitions with neutral blocks, including more examples and illustrations, see \cite{MaWe18a}.
\subsection{Two-Colored Partitions}

\begin{wrapfigure}[6]{L}{4.5cm}
  \centering
  \begin{tikzpicture}[scale=0.666]
    \draw [dotted] (-0.5,0) -- (5.5,0);
    \draw [dotted] (-0.5,3) -- (5.5,3);
    \node [circle, scale=0.4, draw=black, fill=white] (x1) at (0,0) {};
    \node [circle, scale=0.4, draw=black, fill=black] (x2) at (1,0) {};
    \node [circle, scale=0.4, draw=black, fill=white] (x3) at (2,0) {};
    \node [circle, scale=0.4, draw=black, fill=black] (x4) at (3,0) {};
    \node [circle, scale=0.4, draw=black, fill=black] (x5) at (4,0) {};
    \node [circle, scale=0.4, draw=black, fill=white] (x6) at (5,0) {};
    \node [circle, scale=0.4, draw=black, fill=white] (y1) at (0,3) {};
    \node [circle, scale=0.4, draw=black, fill=white] (y2) at (1,3) {};
    \node [circle, scale=0.4, draw=black, fill=white] (y3) at (2,3) {};
    \node [circle, scale=0.4, draw=black, fill=black] (y4) at (3,3) {};
    \node [circle, scale=0.4, draw=black, fill=white] (y5) at (4,3) {};
    \node [circle, scale=0.4, draw=black, fill=black] (y6) at (5,3) {};
    \draw (x1) to (y1);
    \draw (x4) to (y4);
    \draw (x5) to (y6);
    \draw (x6) to (y5);
    \draw (x2) -- ++(0,1) -| (x3);
    \draw (y2) -- ++(0,-1) -| (y3);    
  \end{tikzpicture}
\end{wrapfigure}

By a \emph{(two-colored) partition}  we mean a combinatorial object specified by the following data: two finite sets, the \emph{upper} and \emph{lower row}, a total order on each of them (from \emph{left}, less, to \emph{right}, greater), an exhaustive decomposition into mutually disjoint subsets, the \emph{blocks}, of the disjoint union of the upper and lower row (the \emph{points}) and, lastly, a two-valued (\emph{black}, $\bullet$, or \emph{white}, $\circ$) map on the points, assigning to every point its \emph{color}. \par
If a block contains both upper and lower points, we call it a \emph{through block} and a \emph{non-through block} otherwise. We say that $\circ$ and $\bullet$ are \emph{inverse} to each other. Partitions are represented graphically by two parallel lines of black and white dots connected by a collection of strings.  The set of all partitions is denoted by $\mathcal P^{\circ\bullet}$. 
A partition each of whose blocks has two elements is called a \emph{pair partition}, an element of $\mathcal P^{\circ\bullet}_2$. We restrict ourselves to pair partitions in this article.

\subsection{Operations}
\label{subsection:operations}
From given partitions $p$ and $p'$, a new partition, the \emph{tensor product} $p\otimes p'$, is created by appending each row of $p'$ at the right end of the respective row of $p$.
\par
By exchanging the roles of the upper and the lower row of $p\in \mathcal P^{\circ\bullet}$, we obtain the \emph{involution} $p^*$ of $p$.
\par
A pairing $(p,p')$ of partitions $p,p'\in \mathcal P^{\circ\bullet}$ is \emph{composable} if the lower row of $p'$ and the upper row of $p$ agree in size and coloration. Under these conditions vertical concatenation is possible and yields the \emph{composition} $pp'$ of $(p,p')$: The lower row of $p$ also becomes the lower row of $pp'$, whereas the upper row is carried over from $p'$; Existing non-through blocks of $p$ on the lower row and, likewise, of $p'$ on the upper row are retained; The other blocks of $pp'$ are induced by the partition $s$ which is the least upper bound of, on the one hand, $p'$ restricted to its lower and, on the other hand, $p$ restricted to its upper row; Namely, for every block $B$ of $s$, the points from the upper row of $p'$ and the lower row of $p$ whose former blocks had a non-empty intersection with $B$ form a block in $pp'$.
\par
The \emph{color inversion} $\overline p$ of $p\in \mathcal P^{\circ\bullet}$ replaces $\circ$ with $\bullet$ and vice versa for all points.
\par
Reversing the total orders of both rows of $p\in \mathcal P^{\circ\bullet}$ produces the \emph{reflection} $\hat p$ of $p$. The color inversion of $\hat p$ is called the \emph{verticolor reflection} $\tilde p$ of $p$.
\par
Four basic kinds of \emph{rotation} can be defined: To obtain $p^\rcurvearrowdown$, we remove the leftmost point $\alpha$ on the upper row of $p\in \mathcal P^{\circ\bullet}$ and add a point $\beta$ of the opposite color of $\alpha$ to the lower row left to its leftmost point. The point $\beta$ replaces $\alpha$ as far as the blocks are concerned. Transferring the rightmost point of the upper row to the right end of the lower row gives the rotation $p^\lcurvearrowdown$. Analogously, $p^\lcurvearrowup$ and $p^\rcurvearrowup$ result from moving points up from the lower row, instead.
\par
Defining $p^\circlearrowright:=(p^\lcurvearrowup)^\lcurvearrowdown$ and  $p^\circlearrowleft:=(p^\rcurvearrowdown)^\rcurvearrowup$ yields clockwise and counter-clockwise \emph{cyclic rotations}.
\par
Given $p\in \mathcal P^{\circ\bullet}$ and a set $S$ of points in $p$, the \emph{erasing} $E(p,S)$ of $S$ from $p$ is obtained by removing $S$ and combining all blocks of $p$ which contained an element of $S$ into one new block.
\subsection{Categories}
\label{subsection:categories}
A \emph{category (of partitions)} is a subset of $\mathcal P^{\circ\bullet}$ which is closed under tensor products, compositions and involutions and contains the partitions $\PartIdenW$, $\PartIdenB$, $\PartIdenLoBW$ and $\PartIdenLoWB$. While categories are then also invariant under verticolor reflection as well as basic and cyclic rotations, this need not be the case for reflection and color inversion. For any set $\mc G\subseteq \Cp$, we denote the smallest category containing $\mc G$ by $\langle \mc G\rangle$ and call it \emph{the category generated by $\mc G$}. Categories of (uncolored) partitions were first introduced in \cite{BaSp09}. Note that the composition of uncolored pair partitions with the same amount of upper and lower points yields the multiplication in the Brauer algebra \cite{Br37} or the Temperley-Lieb algebra \cite{TeLi71} up to a scalar factor.

\subsection{Orientation}
On the points of $p\in \mathcal P^{\circ\bullet}$ with lower row $L$ and upper row $U$, a cyclic order, the \emph{orientation}, is defined by the condition that it concur with the total order $\leq _L$ on $L$, but with the exact reverse of the total order $\leq _U$ on $U$, that the minimum of $\leq_U$ be succeeded by the minimum of $\leq _L$ and that the maximum of $\leq_U$ be preceded by the maximum of $\leq_L$. Intervals with respect to this cyclic order are denoted by, e.g., $\left[\alpha,\beta\right]_p$,  $\left]\alpha,\beta\right[_p$, etc.\ for points $\alpha$ and $\beta$ in $p$. See \cite[Sect.~3.1]{MaWe18a}.

\subsection{Sectors}
Given a proper subset $S$ of the points of $p$ that can be written as an interval with respect to the cyclic order, we call the set containing exactly the first and last point of $S$ the \emph{boundary} $\partial S$ of $S$. In contrast, the set $\inte(S):=S\backslash \partial S$ is referred to as the \emph{interior} of $S$.
If $\partial S$ is a block of $p$, the set $S$ is called a \emph{sector} in $p$. The sectors $S'$ in $p$ with $S'\subseteq \inte(S)$ are the \emph{subsectors} of $S$. See \cite[Sect.~3.4]{MaWe18a}.

\subsection{Color Sum}
\label{subsection:color-sum}
Based on the native coloration, we define the \emph{normalized color} of any given point of a partition to congrue with the native color in the case of a lower point, but to be the inverse color of any upper point.
\par
The \emph{color sum} $\sigma_p$ of $p\in \mathcal P^{\circ\bullet}$ is the signed measure with density $1$ and $-1$ given to the normalized colors $\circ$ and $\bullet$ respectively.
The null sets of $\sigma_p$ we call \emph{neutral}.
\par
Every category of partitions is closed under the erasing of neutral intervals. A neutral interval of length $2$ is called a \emph{turn}. See \cite[Sect.~3.3]{MaWe18a}.

\subsection{Connectedness}
Two blocks $B$ and $B'$ in $p\in \mathcal P^{\circ\bullet}$ are said to \emph{cross} if there are four pairwise distinct points $\alpha,\beta\in B$ and $\gamma,\delta\in B'$ occurring in the order $(\alpha,\gamma,\beta,\delta)$ with respect to the orientation. If no two blocks cross in $p$, then we say that $p$ is \emph{non-crossing}, in short: $p\in \mathcal {NC}^{\circ\bullet}$ (see \cite{TaWe15a} for all subcategories of $\mc {NC}^{\circ\bullet}$).
\par
We call the blocks $B$ and $B'$ \emph{connected} if $B=B'$, if $B$ and $B'$ cross or if there are pairwise different blocks $B_1,\ldots,B_m$ in $p$ such that $B$ crosses $B_1$, such that $B_i$ crosses $B_{i+1}$ for every $i\in \mathbb N$ with $i<m$, and such that $B_m$ crosses $B'$.
\par
The classes of this equivalence relation are the \emph{connected components} of $p$. And we say that $p$ is \emph{connected} if it has only a single connected component. Erasing the complement of any connected component $S$ of $p$ yields the \emph{factor partition} of $S$.
See \cite[Sect.~3.2]{MaWe18a}.
\subsection{Pair Partitions with Neutral Blocks}
\label{subsection:pair-partitions-with-neutral-blocks}
We denote by $\Cppnb$ the set of all pair partitions all of whose blocks are neutral sets. Furthermore, denote by $\mc S_0$ the set of all $p\in \Cppnb$ such that $\sigma_p(S)=0$ for all sectors $S$ in $p$. See \cite[Sect.~4 and Main Thm.~1]{MaWe18a} for more on $\mc S_0$.
\section{\texorpdfstring{Definition of $\mathcal I_D$ and Set Relationships\\{[Main Theorem~\ref*{theorem:main_1}~{\normalfont\ref*{item:main_1-2}]}}}{Definition and Set Relationships}}
\label{section:definition}
To define the sets $\mc I_D$ which are the subject matter of this article, we introduce the notion of color distance. 

 \begin{definition}
   Let $p\in \Cppnb$ be arbitrary and let $\alpha$ and $\beta$ be points in $p$. We call
   \begin{align*}
     \delta_p(\alpha,\beta)\eqpd
     \begin{cases}
       \sigma_p(]\alpha,\beta[_p),&\text{if } \alpha \text{ and } \beta \text{ have different normalized colors},\\
       \sigma_p(]\alpha,\beta]_p),&\text{if }\alpha \text{ and } \beta \text{ have the same normalized color},
     \end{cases}
   \end{align*}
   the \emph{signed color distance from $\alpha$ to $\beta$ in $p$} and
   \begin{align*}
     d_p(\alpha,\beta)\eqpd |\delta_p(\alpha,\beta)|.
   \end{align*}
   the \emph{(absolute) color distance from $\alpha$ to $\beta$ in $p$}.
 \end{definition}
 While only the absolute color distance is required to define the sets $\mc I_D$, it is the signed color distance which enables the proofs.
The following lemma shows that, given $p\in \Cppnb$, the name \enquote{distance} is appropriately chosen for $\delta_p$ and $d_p$. Note that the set of all points of $p$, being the disjoint union of neutral blocks, is neutral as well.
\begin{lemma}
  \label{lemma:properties_of_signed_color_distance}
   Let  $\alpha$, $\beta$ and $\gamma$ be points in $p\in \Cppnb$.
   \begin{enumerate}[wide,label=(\alph*)]
   \item\label{item:properties_of_signed_color_distance-item_1} It holds $\delta_p(\alpha,\alpha)= 0$.
   \item\label{item:properties_of_signed_color_distance-item_2} It holds $\delta_p(\alpha,\beta)= -\delta_p(\beta,\alpha)$.
   \item\label{item:properties_of_signed_color_distance-item_3} It holds $\delta_p(\alpha,\gamma)= \delta_p(\alpha,\beta)+\delta_p(\beta,\gamma)$.
     \item\label{item:properties_of_signed_color_distance-item_4} The map $d_p$ is a pseudo-metric on the set of points of $p$.
   \end{enumerate}
 \end{lemma}
 \begin{proof}
   \begin{enumerate}[wide,label=(\alph*)]
   \item The definition of $\delta_p$ yields $\delta_p(\alpha,\alpha)=\sigma_p(]\alpha,\alpha]_p)=0$.
   \item We can rewrite the definition of $\delta_p$ as
\[\delta_p(\alpha,\beta)=\sigma_p(]\alpha,\beta]_p)+\frac{1}{2}\left(\sigma_p(\{\alpha\})-\sigma_p(\{\beta\})\right).\]
Using $\sigma_p(]\alpha,\beta]_p)=- \sigma_p(]\beta,\alpha]_p)$ now proves the claim.
\item We compute, employing the formula for $\delta_p$ from the proof of Claim~\ref{item:properties_of_signed_color_distance-item_2},
  \begin{align*}
    \delta_p(\alpha,\beta)+\delta_p(\beta,\gamma)&=\sigma_p(]\alpha,\beta]_p)+\sigma_p(]\beta,\gamma]_p)\\ & \phantom{{}={}}+\frac{1}{2}\left(\sigma_p(\{\alpha\})-\sigma_p(\{\beta\})\right)+\frac{1}{2}\left(\sigma_p(\{\beta\})-\sigma_p(\{\gamma\})\right).
  \end{align*}
  Thus, from $\sigma_p(]\alpha,\beta]_p)+\sigma_p(]\beta,\gamma]_p)= \sigma_p(]\alpha,\gamma]_p)$ follows the claim.
  \item Claim~\ref{item:properties_of_signed_color_distance-item_4} is implied by the previous three.\qedhere
   \end{enumerate}
 \end{proof}
 \begin{remark}
   Without the assumption $p\in \Cppnb$,  Lemma~\hyperref[lemma:properties_of_signed_color_distance]{\ref*{lemma:properties_of_signed_color_distance}~\ref*{item:properties_of_signed_color_distance-item_1}--\ref*{item:properties_of_signed_color_distance-item_3}} remains true for arbitrary $p\in \Cp$ if we replace equality by congruence modulo $\Sigma(p)\eqpd \sigma_p(P_p)$, where $P_p$ denotes the set of all points of $p$.
 \end{remark}
For special partitions, color distance respects the block structure as the following lemma shows.
  \begin{lemma}
    \label{lemma:color_distance_well_definedness}
  Let $\{\alpha,\beta\}$ and $\{\alpha',\beta'\}$ be blocks in $p\in\Cppnb$. If $p\in \mathcal{S}_0$, then \[\delta_p(\alpha,\alpha')=\delta_p(\alpha,\beta')=\delta_p(\beta,\alpha')=\delta_p(\beta,\beta').\]
\end{lemma}
\begin{proof}
  Follows from Lemma~\hyperref[item:properties_of_signed_color_distance-item_2]{\ref*{lemma:properties_of_signed_color_distance}~\ref*{item:properties_of_signed_color_distance-item_2}} and~\hyperref[item:properties_of_signed_color_distance-item_3]{\ref*{item:properties_of_signed_color_distance-item_3}} because $p\in \mc S_0$ means $\delta_p(\alpha,\beta)=\delta_p(\alpha',\beta')=0$, see Section~\ref{subsection:pair-partitions-with-neutral-blocks}.
\end{proof}
Hence, for $p\in \mc S_0$, we can actually regard $\delta_p$ and $d_p$ as defining color distances not only for points but also for blocks. 

\begin{definition}
  \label{definition:categories_I_N}
  Let $B$ and $B'$ be two blocks in $p\in \mathcal{S}_0$. We call \[\delta_p(B,B'):=\delta_p(\alpha,\alpha') \quad\text{and}\quad d_p(B,B')\eqpd |\delta_p(B,B')|,\] where $\alpha\in B$, $\alpha'\in B'$, the \emph{signed} respectively \emph{(absolute) color distance} from $B$ to $B'$.
\end{definition}
The properties of the signed and absolute color distance of points from Lemma~\ref{lemma:properties_of_signed_color_distance} carry over to the signed and absolute color distance of blocks.
\par
Now, we are in a position to define the sets $\mathcal I_D$ from the \hyperref[theorem:main_1]{Main Theorems}.
\begin{definition}
  \label{definition:categories_I_N}
  For every subsemigroup $D$ of $(\mathbb N_0,+)$, denote by $\mathcal I_D$ the set of all partitions $p\in \mathcal S_0$ such that, for all blocks $B$ and $B'$ in $p$, whenever $d_p(B,B')\in D$, the blocks $B$ and $B'$ do not cross each other in $p$.
\end{definition}
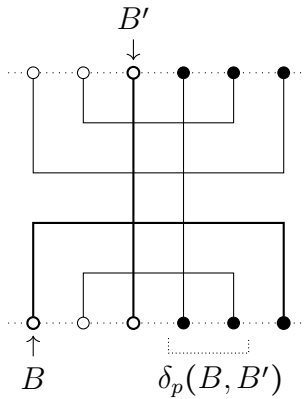
\begin{minipage}[c]{0.3\textwidth}
\flushleft
\begin{tikzpicture}[scale=0.666, baseline =6.5cm*0.666]
    \useasboundingbox (-0.75, -1.75) rectangle (5.75,6.5);
    \draw[dotted] (-0.5,0) -- (5.5,0);
    \draw[dotted] (-0.5,5) -- (5.5,5);
    \node [circle, scale=0.4, draw=black, fill=white, thick] (x0) at (0,0) {};
    \node [circle, scale=0.4, draw=black, fill=white] (x1) at (1,0) {};
    \node [circle, scale=0.4, draw=black, fill=white, thick] (x2) at (2,0) {};
    \node [circle, scale=0.4, draw=black, fill=black] (x3) at (3,0) {};
    \node [circle, scale=0.4, draw=black, fill=black] (x4) at (4,0) {};
    \node [circle, scale=0.4, draw=black, fill=black, thick] (x5) at (5,0) {};
    \node [circle, scale=0.4, draw=black, fill=white] (y0) at (0,5) {};
    \node [circle, scale=0.4, draw=black, fill=white] (y1) at (1,5) {};
    \node [circle, scale=0.4, draw=black, fill=white, thick] (y2) at (2,5) {};
    \node [circle, scale=0.4, draw=black, fill=black] (y3) at (3,5) {};
    \node [circle, scale=0.4, draw=black, fill=black] (y4) at (4,5) {};
    \node [circle, scale=0.4, draw=black, fill=black] (y5) at (5,5) {};
    \draw[thick] (x0) -- ++ (0,2) -| (x5);
    \draw (y0) -- ++ (0,-2) -| (y5);
    \draw (x1) -- ++ (0,1) -| (x4);
    \draw (y1) -- ++ (0,-1) -| (y4);    
    \draw[thick] (x2) -- (y2);
    \draw (x3) -- (y3);
    \node[below = 0.36cm of x0] (l1) {$B$};
    \node[above = 0.36cm of y2] (l2) {$B'$};
    \draw[->] (l1.north) -- ($(x0.south)+(0,-0.125cm)$);
    \draw[->] (l2.south) -- ($(y2.north)+(0,0.125cm)$);
    \draw[densely dotted] ($(x3)+(-0.3,-0.3)$) -- ++ (0,-0.3)   -| ($(x4)+(0.3,-0.3)$) node [below,pos=0.325] {$\delta_p(B,B')$}; 
  \end{tikzpicture}\par
  \end{minipage}
  \begin{minipage}[c]{0.7\textwidth}
    \setlength\parindent{15pt}\fussy
    \indent For example, the two crossing blocks $B$ and $B'$ in the partition on the left hand side
    have color distance $2$. All color distances occurring between blocks in the partition are $0$, $1$ or $2$. There are also crossings between blocks of distance $1$, but no blocks with color distance $0$ cross each other, making the partition an element of $\mc I_{\mathbb N_0\backslash\{1,2\}}$.
    \par
    In comparison, all three blocks in the partition $\PartHalfLibWBW\in \mc I_{\pint}$ have color distance $0$, which is why this partition is not an element of $\mc I_{\mathbb N_0\backslash\{1,2\}}$.
  \end{minipage}
  \par
Part~\ref{item:main_1-2} of Main Theorem~\ref{theorem:main_1} follows immediately from Definition~\ref{definition:categories_I_N}.  
\begin{proposition}
  \label{proposition:set_relations}
  \begin{enumerate}[label=(\alph*),labelwidth=!]
  \item\label{item:set_relations-item_1} It holds $\mc I_\nnint=\mc{NC}^{\circ\bullet}\cap \Cppnb=\langle\emptyset\rangle$.
  \item\label{item:set_relations-item_2} It holds $\mc I_\emptyset=\mc S_0$.
  \item\label{item:set_relations-item_3} For all subsemigroups $D,D'$ of $(\nnint,+)$ holds
    \begin{align*}
      D\subseteq D'\implies \mc I_{D}\supseteq \mc I_{D'}.
    \end{align*}
  \end{enumerate}
\end{proposition}
\begin{proof}
  \begin{enumerate}[label=(\alph*),labelwidth=!, wide]
  \item If blocks $B$ and $B'$ in $p\in \mc S_0$ may only cross if $d(B,B')\notin \nnint$,  then $p$ must be non-crossing.  Conversely, recognize that $\mc{NC}^{\circ\bullet}\cap \Cppnb\subseteq \mc S_0$ because $\inte(S)$ must be a subpartition for every sector $S$ in a non-crossing $p\in \Cppnb$, implying $\sigma_p(S)=0$. It was shown in \cite[Proposition~3.3~a)]{TaWe15a} that $\mc{NC}^{\circ\bullet}\cap \Cppnb=\langle\emptyset\rangle$.
    \item Likewise, the condition that blocks $B$ and $B'$ in $p\in \mc S_0$ are forbidden from crossing unless $d(B,B')\in\nnint$ is no restriction at all. So, $\mc I_\emptyset=\mc S_0$.
    \item If blocks $B$ and $B'$ may not cross in $p\in \mc S_0$ if $d(B,B')\in D'$ and if $D\subseteq D'$, then, especially, they cannot cross if $d(B,B')\in D$.\qedhere
  \end{enumerate}
\end{proof}
\section{\texorpdfstring{Category Property of $\mathcal I_D$\\{[Main Theorem~\ref*{theorem:main_1}~{\normalfont\ref*{item:main_1-1}]}}}{Category Property}}
  In the following, it is convenient to gather together all the color distances which occur between crossing blocks in a given partition.
  \begin{definition}
    \label{definition:A}
    For all $p\in \mc S_0$ define
    \begin{align*}
      A(p)\eqpd \{d_{p}(B,B')\mid B,B'\text{ crossing blocks in }p\}.
    \end{align*}
  \end{definition}
  This notation can be used to express membership in one of the sets $\mc I_D$ from Definition~\ref{definition:categories_I_N} more compactly:
  \begin{remark}
    \label{remark:I_N-with-A}
    For all subsemigroups $D$ of $(\nnint,+)$ holds
    \begin{align*}
      \mc I_D= \{p\in \mc S_0\mid A(p)\subseteq \nnint\backslash D\}.
    \end{align*}
  \end{remark}
The next lemma shows how the map $A$ behaves under category operations.
  \begin{lemma}
    \label{lemma:A-under-operations}
    Let $p,p'\in \mc S_0$ be arbitrary.
    \begin{enumerate}[label=(\alph*)]
    \item It holds $A(p^*)=A(p)$.
    \item It holds $A(p\otimes p')= A(p)\cup A(p')$.
      \item If $(p,p')$ is composable, then $A(pp')\subseteq A(p)\cup A(p')$.
    \end{enumerate}
  \end{lemma}

\begin{proof}

  \begin{enumerate}[wide,label=(\alph*)]
  \item 
 Exchanging the roles of the upper and the lower row of $p$ does not affect color distances: Both the sign of the color sum measure and the cyclic order effectively reverse and the two effects cancel each other. Hence, $A(p^*)=A(p)$.\par
\item  On the one hand, no crossings exist between the two subpartitions of $p\otimes p'$ corresponding to $p$ and to $p'$ respectively. So all crossings in $p\otimes p'$ stem from crossings either in $p$ or in $p'$. On the other hand, the color distances between crossings from $p$ and $p'$ are unaltered when passing to $p\otimes p'$ because the subpartitions of $p\otimes p'$ resulting from $p$ and $p'$ are neutral as a whole due to $p,p'\in \mc S_0$.\par
\item  Let $A$ and $B$ be two blocks crossing each other in $pp'$. By passing to $(pp')^*=(p')^*p^*$ if necessary, we can assume that both $A$ and $B$ intersect the lower row, i.e.\ each have a point in common with the set of all lower points. We treat the case of both $A$ and $B$ being through blocks. The other cases are similar.\par
  In that situation, there exist $m,n\in \mathbb N$ and sequences $A_1,\ldots,A_m,B_1,\ldots,B_n$ of blocks in $p$ and $A_1',\ldots,A_m',B_1',\ldots,B_n'$ of blocks in $p'$ such that the following conditions are met: Block $A$ intersects both $A_1$ and $A_m'$; Block $B$ intersects both $B_1$ and $B_n'$; If we identify the lower row of $p'$ with the upper row of $p$, then, for all $i,j\in \mathbb N$ with $i< m$ and $j< n$, in each of the following four pairs of blocks the two blocks intersect each other: $(A_i,A_i')$, $(A_i',A_{i+1})$, $(B_j,B_j')$ and $(B_j',B_{j+1})$.
  \par
  
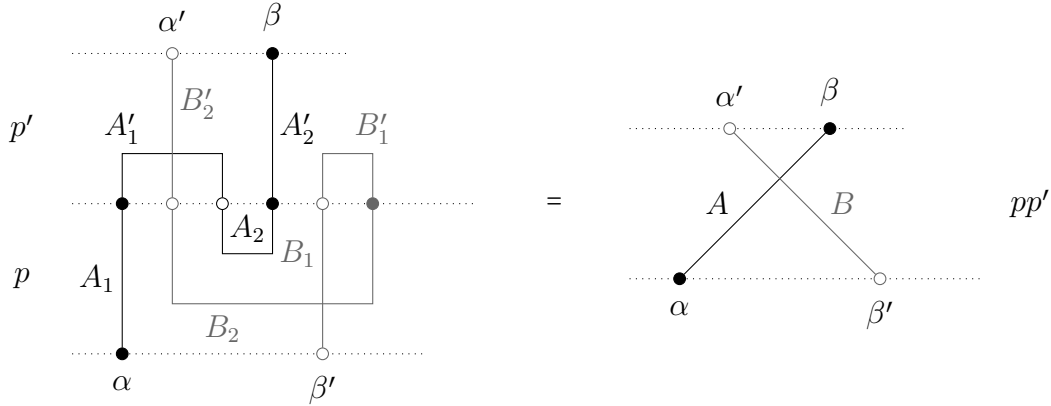
\vspace{-0.25em}
  \begin{center}
    \begin{tikzpicture}[scale=0.666, baseline = 4cm*0.666]
      \draw [dotted] (-0.5,1) -- (6.5,1);
      \draw [dotted] (-0.5,4) -- (7.5,4);
      \draw [dotted] (-0.5,7) -- (5,7);
      \node [circle, scale=0.4, draw=black, fill=black] (x1) at (0.5,1) {};
      \node [circle, scale=0.4, draw=black, fill=black] (x2) at (3.5,7) {};
      \node [circle, scale=0.4, draw=black, fill=black] (a1) at (0.5,4) {};
      \node [circle, scale=0.4, draw=black, fill=white] (a2) at (2.5,4) {};
      \node [circle, scale=0.4, draw=black, fill=black] (a3) at (3.5,4) {};      
      \node [circle, scale=0.4, draw=gray, fill=white] (y1) at (4.5,1) {};
      \node [circle, scale=0.4, draw=gray, fill=white] (y2) at (1.5,7) {};
      \node [circle, scale=0.4, draw=gray, fill=white] (b1) at (4.5,4) {};
      \node [circle, scale=0.4, draw=gray, fill=gray] (b2) at (5.5,4) {};
      \node [circle, scale=0.4, draw=gray, fill=white] (b3) at (1.5,4) {};                  
      \draw (x1) -- (a1);
      \draw (a1) -- ++(0,1)-| (a2);
      \draw (a2) -- ++(0,-1)-| (a3);
      \draw (a3) -- (x2);
      \draw[gray] (y1) -- (b1);
      \draw[gray] (b1) -- ++(0,1)-| (b2);
      \draw[gray] (b2) -- ++(0,-2)-| (b3);
      \draw[gray] (b3) -- (y2);
      \node at (-1.5,2.5) {$p$};
      \node at (-1.5,5.5) {$p'$};
      \node at (0,2.5) {$A_1$};
      \node at (0.5,5.5) {$A_1'$};
      \node at (3,3.5) {$A_2$};
      \node at (4,5.5) {$A_2'$};
      \node[gray] at (4,3) {$B_1$};
      \node[gray] at (5.5,5.5) {$B_1'$};
      \node[gray] at (2.5,1.5) {$B_2$};
      \node[gray] at (2,6) {$B_2'$};
      \node [below =2pt of x1] {$\alpha$};
      \node [below =2pt of y1] {$\beta'$};
      \node [above =2pt of x2] {$\beta$};
      \node [above =2pt of y2] {$\alpha'$};
    \end{tikzpicture}\quad\quad
        \begin{tikzpicture}[scale=0.666,baseline=1.5cm*0.666]
      \draw [dotted] (-0.5,0) -- (6.5,0);
      \draw [dotted] (-0.5,3) -- (5,3);
      \node [circle, scale=0.4, draw=black, fill=black] (x1) at (0.5,0) {};
      \node [circle, scale=0.4, draw=black, fill=black] (x2) at (3.5,3) {};
      \node [circle, scale=0.4, draw=gray, fill=white] (y1) at (4.5,0) {};
      \node [circle, scale=0.4, draw=gray, fill=white] (y2) at (1.5,3) {};
      \draw (x1) -- (x2);
      \draw[gray] (y1) -- (y2);
      \node at (-2,1.5) {$=$};      
      \node at (7.5,1.5) {$pp'$};
      \node at (1.25,1.5) {$A$};
      \node[gray] at (3.75,1.5) {$B$};
      \node [below =2pt of x1] {$\alpha$};
      \node [below =2pt of y1] {$\beta'$};
      \node [above =2pt of x2] {$\beta$};
      \node [above =2pt of y2] {$\alpha'$};
    \end{tikzpicture}
  \end{center}
    \vspace{-0.25em}
  \par
 The fact that all the sectors of all the blocks $A_1,A_1',\ldots,A_m,A_m',B_1,B_1',\ldots,B_n,B_n'$ are neutral can be used to prove by induction with the help of Lemma~\hyperref[item:properties_of_signed_color_distance-item_3]{\ref*{lemma:properties_of_signed_color_distance}~\ref*{item:properties_of_signed_color_distance-item_3}} first
  \begin{align*}
    0=\delta_{p}(A_{i_1},A_{i_2})=\delta_{p'}(A'_{i_1},A'_{i_2})=\delta_{p}(B_{j_1},B_{j_2})=\delta_{p'}(B'_{j_1},B'_{j_2})
  \end{align*}
for all $i_1,i_2,j_1,j_2\in \pint$ with $i_1,i_2\leq m$ and  $j_1,j_2\leq n$,
and thus
  \begin{align*}
    \delta_{pp'}(A,B)=\delta_{p}(A_i,B_j)=\delta_{p'}(A'_i,B_j')
  \end{align*}
  for all $i,j\in \mathbb N$ with $i\leq m$ and $j\leq n$. Because $A$ and $B$ cross in $pp'$, there must exist $i,j\in \mathbb N$ such that $A_i$ and $B_j$ cross in $p$ or $A_i'$ and $B_j'$ cross in $p'$. That means $d_{pp'}(A,B)\in A(p)\cup A(p')$, which concludes the proof.\qedhere
    \end{enumerate}
  \end{proof}
  We employ for all sets $X$ and $Y$, all maps $f:X\to Y$ and all subsets $S\subseteq X$  the notation $f(S)\eqpd \{f(x)\mid x\in S\}$. Moreover, for all systems $X$ of sets use  $\bigcup X\eqpd\bigcup_{Y\in X}Y$.
    \begin{remark}
    \label{lemma:invariance}
    For all sets $\mc S\subseteq \mc S_0$ holds $\bigcup A(\langle \mc S\rangle)= \bigcup A(\mc S)$.
  \end{remark}
  \begin{proof}
    The set $\mc C\eqpd \{p\in \mc S_0\mid A(p)\subseteq \bigcup A(\mc S)\}$ satisfies $\mc S\subseteq \mc C$. Because $\mc C$ is a category by Lemma~\ref{lemma:A-under-operations}, we conclude $\langle \mc S\rangle \subseteq \mc C$. That implies $\bigcup A(\langle \mc S\rangle )\subseteq \bigcup A(\mc C)\subseteq \bigcup A(\mc S)$.
  \end{proof}
  
Lemma~\ref{lemma:A-under-operations} is the key to proving Part~\ref{item:main_1-1} of Main Theorem~\ref{theorem:main_1}.
\begin{proposition}
  \label{theorem:category_property_of_I_N}
  For every subsemigroup $D$ of $(\mathbb N_0,+)$, the set $\mathcal I_D$ is a category of partitions.
\end{proposition}
\begin{proof}
  For $p,p'\in \mc I_D$ holds $A(p)\cup A(p')\subseteq \nnint\backslash D$ by Remark~\ref{remark:I_N-with-A}. Lemma~\ref{lemma:A-under-operations} hence proves $A(p^*)=A(p)\subseteq \nnint \backslash D$, $A(p\otimes p')= A(p)\cup A(p')\subseteq \nnint \backslash D$ and, if $(p,p')$ is composable, $A(pp')\subseteq  A(p)\cup A(p')\subseteq \nnint \backslash D$.

\end{proof}
Note that Lemma~\ref{lemma:invariance} gives no clue as to which sets $\bigcup A(\mc C)$ actually occur for categories $\mc C\subseteq \mc S_0$. The fact that only subsemigroups of $(\nnint,+)$ are possible requires an entirely different argument, to be given in the subsequent sections.

\section{Reminder on Brackets}
\label{section:brackets}
We recall the  definitions and results from \cite[Sect.~6]{MaWe18a} about \emph{bracket} partitions  required in the subsequent sections of this article. With the help of brackets we will be able to give explicit generators of the categories $\mc I_D$ and classify all subcategories of $\mc  S_0$.
\subsection{Brackets}
All categories $\mc C$ with $\mc C\subseteq \Cppnb$ will in fact be described solely in terms of the classes of the following equivalence relation.
\begin{definition}{\normalfont\cite[Def.~6.2]{MaWe18a}}
	Given $p,p'\in\mathcal P^{\circ\bullet}_{2,\mathrm{nb}}$ and sectors $S$ in $p$ and $S'$ in $p'$, we number the points in $\inte (S)$ and $\inte (S')$ with respect to the cyclic order. We say that $(p,S)$ and $(p',S')$ are \emph{equivalent} if the following four conditions are met:
	\begin{enumerate}
		\item The sectors $S$ and $S'$ are of equal size.
		\item The same normalized colors occur in the same order in $S$ and $S'$.
		\item For all $i$, the $i$-th point of $S$ belongs to a block crossing $\partial S$ in $p$ if and only if the $i$-th point of $S'$ belongs to a block crossing $\partial S'$ in $p'$. 
		\item For all $i,j$, the $i$-th and $j$-th points of $S$ form a block in $p$ if and only if the $i$-th and $j$-th points of $S'$ form a block in $p'$.
                \end{enumerate}
                In other words: $p$ restricted to $S$ coincides with $p'$ restricted to $S'$ (up to rotation). 
\end{definition}
Particular representatives of the equivalence classes are \emph{bracket} partitions.
\par
\vspace{0.25cm}
\noindent
       \begin{minipage}[c]{0.666\textwidth}
         \begin{definition}{\normalfont\cite[Def.~6.1]{MaWe18a}}
We call $p\in\mathcal P^{\circ\bullet}_{2,\mathrm{nb}}$ a \emph{bracket} if $p$ is projective, i.e.\ $p=p^*$ and $p^2=p$, and if the lower row of $p$ is a sector in $p$. 
\end{definition}
\begin{definition}{\normalfont\cite[Def.~6.3]{MaWe18a}}
	Let $S$ be a sector in $p\in\mathcal P^{\circ\bullet}_{2,\mathrm{nb}}$. We refer to the (uniquely determined) bracket $q$ with lower row $M$ which satisfies that $(p,S)$ and $(q,M)$ are equivalent as  \emph{the bracket $B(p,S)$ associated with $(p,S)$}.
      \end{definition}
      \end{minipage}
      \hfill
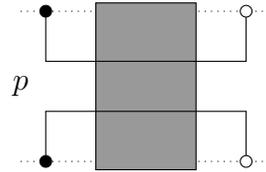
       \begin{minipage}[c]{0.333\textwidth}
	\centering
	\begin{tikzpicture}[scale=0.666]
          \draw [dotted, shift={(-0.5,0)}] (0,0) -- (5,0);
	\draw [dotted, shift={(-0.5,3)}] (0,0) -- (5,0);		
	\node [scale=0.4, fill=black, draw=black,circle] (x1) at (0,0) {};
	\node [scale=0.4, fill=white, draw=black,circle] (x2) at (4,0) {};		
	\node [scale=0.4, fill=black, draw=black,circle] (y1) at (0,3) {};
	\node [scale=0.4, fill=white, draw=black,circle] (y2) at (4,3) {};				
	\draw [fill=lightgray] (1,-0.166) rectangle (3,3.166);
	\draw (x1) -- ++ (0,1) -| (x2);
	\draw (y1) -- ++ (0,-1) -| (y2);
	\node at (-0.5,1.5) {$p$};	
	\end{tikzpicture}
      \end{minipage}
      \par
\vspace{0.25cm}
Categories are closed under passing to associated brackets.
\begin{lemma}
	\label{lemma:associated_brackets} {\normalfont\cite[Lem.~6.4]{MaWe18a}}
For all sectors $S$ in $p\in \mathcal  P^{\circ\bullet}_{2,\mathrm{nb}}$ holds $B(p,S)\in\langle p\rangle$.
\end{lemma}

\subsection{Residual Brackets} 
It can be seen that every category is generated by its set of brackets. But that result can be significantly refined to \emph{residual brackets}. Recall the definition of the verticolor reflection $\tilde p$ of a partition $p\in \Cp$ from Section~\ref{subsection:operations} and of a turn from Section~\ref{subsection:color-sum}.

\begin{definition}
We call a partition $p\in \Cp$ \emph{verticolor-reflexive} if $\tilde p=p$.
\end{definition}
Especially, verticolor-reflexive partitions have evenly many points in both their rows, which is why the following definition makes sense.
\begin{definition}{\normalfont\cite[Def.~6.10]{MaWe18a}}
  \label{definition:dualizable}
We refer to a bracket $p$ with lower row $S$ as \emph{dualizable} if $p\in \mc S_0$, if $p$ is ver\-ti\-co\-lor-re\-flexive, if $\inte(S)$ is non-empty and if the two middle points of $\inte(S)$ form a turn and belong to through blocks.
\end{definition}
Rotating a dualizable bracket cyclically by a quarter times the number of its points produces again a bracket (and both directions of rotation give identical partitions).
\begin{gather*}
  \begin{tikzpicture}[scale=0.666,baseline=1.56cm]
    \draw[dotted] (-0.5,0) -- (9.5,0);
    \draw[dotted] (-0.5,5) -- (9.5,5);
    \node[circle,scale=0.4,draw=black,fill=black] (x1) at (0,0) {};
    \node[circle,scale=0.4,draw=black,fill=white] (x2) at (1,0) {};
    \node[circle,scale=0.4,draw=black,fill=white] (x3) at (2,0) {};
    \node[circle,scale=0.4,draw=black,fill=white] (x4) at (3,0) {};    
    \node[circle,scale=0.4,draw=black,fill=black] (w1) at (4,0) {};
    \node[circle,scale=0.4,draw=black,fill=white] (w2) at (5,0) {};
    \node[circle,scale=0.4,draw=black,fill=black] (x5) at (6,0) {};
    \node[circle,scale=0.4,draw=black,fill=black] (x6) at (7,0) {};
    \node[circle,scale=0.4,draw=black,fill=black] (x7) at (8,0) {};
    \node[circle,scale=0.4,draw=black,fill=white] (x8) at (9,0) {};    
    \node[circle,scale=0.4,draw=black,fill=black] (y1) at (0,5) {};
    \node[circle,scale=0.4,draw=black,fill=white] (y2) at (1,5) {};
    \node[circle,scale=0.4,draw=black,fill=white] (y3) at (2,5) {};
    \node[circle,scale=0.4,draw=black,fill=white] (y4) at (3,5) {};    
    \node[circle,scale=0.4,draw=black,fill=black] (z1) at (4,5) {};
    \node[circle,scale=0.4,draw=black,fill=white] (z2) at (5,5) {};
    \node[circle,scale=0.4,draw=black,fill=black] (y5) at (6,5) {};
    \node[circle,scale=0.4,draw=black,fill=black] (y6) at (7,5) {};
    \node[circle,scale=0.4,draw=black,fill=black] (y7) at (8,5) {};
    \node[circle,scale=0.4,draw=black,fill=white] (y8) at (9,5) {};    
    \draw (x1) -- ++(0,2) -| (x8);
    \draw (y1) -- ++(0,-2) -| (y8);
    \draw (x4) -- ++(0,1) -| (x5);
    \draw (y4) -- ++(0,-1) -| (y5);    
    \draw (x2) to (y2);
    \draw (x3) to (y3);
    \draw (x6) to (y6);
    \draw (x7) to (y7);        
    \draw (w1) to (z1);
    \draw (w2) to (z2);
    \node at (4.5,-1) {$p$};    
    \draw [dashed] (4.5,2.5) -- (4.5,5.5);
    \draw [dashed] (4.5,2.5) -- (4.5,-0.5);    
    \draw [dashed] (4.5,2.5) -- (-0.5,2.5);
    \draw [dashed] (4.5,2.5) -- (9.5,2.5);        
  \end{tikzpicture}
\underset{\circlearrowright\frac{n}{4}}{\overset{\circlearrowleft\frac{n}{4}}{\longrightarrow}}
  \begin{tikzpicture}[scale=0.666,baseline=1.56cm]
    \draw[dotted] (-0.5,0) -- (9.5,0);
    \draw[dotted] (-0.5,5) -- (9.5,5);
    \node[circle,scale=0.4,draw=black,fill=white] (x1) at (0,0) {};
    \node[circle,scale=0.4,draw=black,fill=black] (x2) at (1,0) {};
    \node[circle,scale=0.4,draw=black,fill=black] (x3) at (2,0) {};
    \node[circle,scale=0.4,draw=black,fill=black] (x4) at (3,0) {};    
    \node[circle,scale=0.4,draw=black,fill=white] (w1) at (4,0) {};
    \node[circle,scale=0.4,draw=black,fill=black] (w2) at (5,0) {};
    \node[circle,scale=0.4,draw=black,fill=white] (x5) at (6,0) {};
    \node[circle,scale=0.4,draw=black,fill=white] (x6) at (7,0) {};
    \node[circle,scale=0.4,draw=black,fill=white] (x7) at (8,0) {};
    \node[circle,scale=0.4,draw=black,fill=black] (x8) at (9,0) {};    
    \node[circle,scale=0.4,draw=black,fill=white] (y1) at (0,5) {};
    \node[circle,scale=0.4,draw=black,fill=black] (y2) at (1,5) {};
    \node[circle,scale=0.4,draw=black,fill=black] (y3) at (2,5) {};
    \node[circle,scale=0.4,draw=black,fill=black] (y4) at (3,5) {};    
    \node[circle,scale=0.4,draw=black,fill=white] (z1) at (4,5) {};
    \node[circle,scale=0.4,draw=black,fill=black] (z2) at (5,5) {};
    \node[circle,scale=0.4,draw=black,fill=white] (y5) at (6,5) {};
    \node[circle,scale=0.4,draw=black,fill=white] (y6) at (7,5) {};
    \node[circle,scale=0.4,draw=black,fill=white] (y7) at (8,5) {};
    \node[circle,scale=0.4,draw=black,fill=black] (y8) at (9,5) {};    
    \draw (x1) -- ++(0,2.143) -| (x8); 
    \draw (y1) -- ++(0,-2.143) -| (y8);
    \draw (x2) to (y2);
    \draw (x7) to (y7);
    \draw (x3) -- ++(0,1.428) -| (x6);
    \draw (x4) -- ++(0,0.714) -| (x5);
    \draw (y3) -- ++(0,-1.428) -| (y6);
    \draw (y4) -- ++(0,-0.714) -| (y5);        
    \draw (w1) to (z1);
    \draw (w2) to (z2);
    \node at (4.5,-1) {$p^\dagger$};        
    \draw [dashed] (4.5,2.5) -- (4.5,5.5);
    \draw [dashed] (4.5,2.5) -- (4.5,-0.5);    
    \draw [dashed] (4.5,2.5) -- (-0.5,2.5);
    \draw [dashed] (4.5,2.5) -- (9.5,2.5);            
  \end{tikzpicture}
\end{gather*}

\begin{definition}{\normalfont\cite[Def.~6.11]{MaWe18a}}
  For a dualizable bracket $p$ with $n$ points, we call the bracket $p^\dagger\eqpd p^{\circlearrowleft\frac{n}{4}}=p^{\circlearrowright\frac{n}{4}}$ the \emph{dual bracket} of $p$.
\end{definition}
With the bracket $p\in\Cppnb$ being dualizable, so is its dual $p^\dagger$ and it holds $(p^\dagger)^\dagger=p$ and $\langle p\rangle=\langle p^\dagger\rangle$. 
\begin{definition}{\normalfont\cite[Def.~6.12]{MaWe18a}}
  \label{definition:residual_bracket}
  \begin{enumerate}[label=(\alph*)]
  \item \label{item:residual_bracket-item_1} 
    Let  $p$ be a bracket with lower row $S$.
    \begin{enumerate}
    \item \label{item:residual_bracket-item_1-item_1} We call $p$ \emph{residual of the first kind} if $p$ is connected and if $\inte(S)$ contains no turns of $p$.
    \item \label{item:residual_bracket-item_1-item_2} We call $p$ \emph{residual of the second kind} if $p$ is connected and dualizable and if $\inte(S)$ contains exactly one turn of $p$.
    \item \label{item:residual_bracket-item_1-item_3} We call $p$ \emph{residual} if $p$ is residual of the first or the second kind.
    \end{enumerate}
  \item The set of all residual brackets is denoted by $\mathcal{B}_{\mathrm{res}}$.
  \end{enumerate}
\end{definition}
Let the lower row $S$ of $p\in \resbr$ start with a point of color $c$. If $p$ is residual of the first kind, there exists $w\in \pint$ such that $S$ is given by 
\begin{align*}
  c \ \underset{w}{\underbrace{c\ldots c}} \ \overline c \quad\text{or}\quad c\  \underset{w}{\underbrace{\overline c\ldots \overline c}} \ \overline c,
\end{align*}
whereas, if $p$ is residual of the second kind, we find $v\in \pint$ such that $S$ has the coloration
\begin{align*}
  c \  \underset{v}{\underbrace{c\ldots c}}\, \underset{v}{\underbrace{\overline c\ldots \overline c}} \ \overline c \quad\text{or}\quad   c \  \underset{v}{\underbrace{\overline c\ldots \overline c}}\,\underset{v}{\underbrace{ c\ldots  c}}\ \overline c.
\end{align*}
With the set $\resbr$ of residual brackets we have found a \enquote{universal generator set}:
\begin{proposition}
  \label{proposition:generation_of_categories_by_residual_brackets}
  {\normalfont\cite[Prop.~6.13]{MaWe18a}}
  For every category $\mathcal C\subseteq \mathcal P^{\circ\bullet}_{2,\mathrm{nb}}$ holds \[\mc C=\left\langle \mc C\cap \resbr\right\rangle.\]
\end{proposition}

\subsection{Bracket Arithmetics} To further reduce the set $\resbr$, to filter out residual brackets generating the same categories (see Section~\ref{subsection:residual_brackets:categories_generated_by_residual_brackets}), we need to know how to generate new residual brackets from old ones via category operations.\par
\vspace{0.75em}
\noindent
\begin{minipage}{5cm}
	\centering
	\begin{tikzpicture}[scale=0.666]
	\draw [dotted, shift={(-0.5,0)}] (0,0) -- (5,0);
	\draw [dotted, shift={(-0.5,3)}] (0,0) -- (5,0);		
	\node [scale=0.4, fill=black, draw=black,circle] (x1) at (0,0) {};
	\node [scale=0.4, fill=white, draw=black,circle] (x2) at (4,0) {};		
	\node [scale=0.4, fill=black, draw=black,circle] (y1) at (0,3) {};
	\node [scale=0.4, fill=white, draw=black,circle] (y2) at (4,3) {};				
	\draw [fill=lightgray] (1,-0.166) rectangle (3,3.166);
	\draw (x1) -- ++ (0,1) -| (x2);
	\draw (y1) -- ++ (0,-1) -| (y2);
	\node at (2,1.5) {$a$};
	\node at (2,-1) {$\mathrm{Br}\left(\bullet\mid a\mid \circ\right)$};
	\end{tikzpicture}
\end{minipage}
\hfill
\begin{minipage}[c]{0.666\textwidth}
	\begin{definition}{\normalfont\cite[Def.~6.14]{MaWe18a}}
		\begin{enumerate}[label=(\alph*)]
                  			\item If $p\in\mathcal{P}^{\circ\bullet}$ is a bracket, the projective partition which is obtained from $p$ by erasing in every row the left- and the rightmost point, is called the \emph{argument} of $p$. 
			\item Conversely, for each projective $a\in\mathcal P^{\circ\bullet}_{2,\mathrm{nb}}$ and every color $c\in\{\circ,\bullet\}$, denote by $\mathrm{Br}\left(c\mid a\mid \overline{c}\right)$ the bracket whose leftmost lower point is of color $c$ and which has the argument $a$. 
		\end{enumerate}
	\end{definition}
\end{minipage}
\vspace{0.25cm}
We can define two ways of altering the starting color of a bracket while, in some sense, preserving its argument. Write $\mathrm{Id}(\circ)\eqpd \PartIdenW$ and  $\mathrm{Id}(\bullet)\eqpd \PartIdenB$.
\begin{definition}{\normalfont\cite[Def.~6.17]{MaWe18a}}

For every $c\in \colors$ and projective $a\in \Cppnb$, we call
    \begin{align*}
      \trmw(\mathrm{Br}\left(c\mid a\mid \overline{c}\right))\eqpd \mathrm{Br}\left(\overline{c}\mid \mathrm{Br}\left(c\mid a\mid \overline{c}\right)\mid c\right)
    \end{align*}
    the \emph{weak inversion} and
        \begin{align*}
      \trms(\mathrm{Br}\left(c\mid a\mid \overline{c}\right))\eqpd  \mathrm{Br}\left({\overline{c}}\mid \mathrm{Id}(c)\otimes a\otimes \mathrm{Id}(\overline{c})\mid c\right)
        \end{align*}
        the \emph{strong inversion} of $\mathrm{Br}\left(c\mid a\mid \overline{c}\right)$.

\end{definition}
      \begin{center}
        \centering
        \begin{tikzpicture}[scale=0.666]
	\draw [dotted, shift={(-0.5,0)}] (-1,0) -- (6,0);
	\draw [dotted, shift={(-0.5,5)}] (-1,0) -- (6,0);
	\node [scale=0.4, fill=white, draw=black,circle] (z1) at (-1,0) {};
	\node [scale=0.4, fill=black, draw=black,circle] (z2) at (5,0) {};
      	\node [scale=0.4, fill=white, draw=black,circle] (w1) at (-1,5) {};
	\node [scale=0.4, fill=black, draw=black,circle] (w2) at (5,5) {};		        
	\node [scale=0.4, fill=black, draw=black,circle] (x1) at (0,0) {};
	\node [scale=0.4, fill=white, draw=black,circle] (x2) at (4,0) {};		
	\node [scale=0.4, fill=black, draw=black,circle] (y1) at (0,5) {};
	\node [scale=0.4, fill=white, draw=black,circle] (y2) at (4,5) {};				
	\draw [fill=lightgray] (1,-0.166) rectangle (3,5.166);
	\draw (x1) -- ++ (0,1) -| (x2);
	\draw (y1) -- ++ (0,-1) -| (y2);
	\draw (z1) -- ++ (0,2) -| (z2);
	\draw (w1) -- ++ (0,-2) -| (w2);        
	\node at (2,2.5) {$a$};
	\node at (2,-1) {$\trmw(\mathrm{Br}\left(\bullet\mid a\mid \circ\right))$};
      \end{tikzpicture}
      \hspace{2em}
              \begin{tikzpicture}[scale=0.666]
	\draw [dotted, shift={(-0.5,0)}] (-1,0) -- (6,0);
	\draw [dotted, shift={(-0.5,5)}] (-1,0) -- (6,0);
	\node [scale=0.4, fill=white, draw=black,circle] (z1) at (-1,0) {};
	\node [scale=0.4, fill=black, draw=black,circle] (z2) at (5,0) {};
      	\node [scale=0.4, fill=white, draw=black,circle] (w1) at (-1,5) {};
	\node [scale=0.4, fill=black, draw=black,circle] (w2) at (5,5) {};		        
	\node [scale=0.4, fill=black, draw=black,circle] (x1) at (0,0) {};
	\node [scale=0.4, fill=white, draw=black,circle] (x2) at (4,0) {};		
	\node [scale=0.4, fill=black, draw=black,circle] (y1) at (0,5) {};
	\node [scale=0.4, fill=white, draw=black,circle] (y2) at (4,5) {};				
	\draw [fill=lightgray] (1,-0.166) rectangle (3,5.166);
	\draw (x1) --  (y1);
	\draw (x2) --  (y2);
	\draw (z1) -- ++ (0,2) -| (z2);
	\draw (w1) -- ++ (0,-2) -| (w2);        
	\node at (2,2.5) {$a$};
	\node at (2,-1) {$\trms(\mathrm{Br}\left(\bullet \mid a \mid  \circ\right))$};
	\end{tikzpicture}
      \end{center}
      These two transformations can indeed be performed using category operations.
\begin{lemma}
  \label{lemma:bracket_operations}
   {\normalfont\cite[Lem.~6.18 d)--f)]{MaWe18a}}
	Let $p,p'$ be two brackets starting with the same color.
	\begin{enumerate}[label=(\alph*)]
		\item
                  \label{item:bracket_operations-item_6}
                  It holds
                	$\langle \PartBracketBWBW\rangle=\langle\PartHalfLibWBW \rangle=\langle \PartHalfLibBWB\rangle=\langle \PartBracketWBWB\rangle$.
		\item 
		\label{item:bracket_operations-item_3}		
                Weak inversion is a reversible category operation: $\langle p\rangle=\langle \trmw(p)\rangle$.
              \item
		\label{item:bracket_operations-item_4}
                Strong inversion is reversible as well, but it is only available in certain categories:
$\langle p, \PartHalfLibWBW \rangle =\langle \trms(p)\rangle$.
	\end{enumerate}
      \end{lemma}
For further operations on the set of all brackets see \cite[Lem.~6.18]{MaWe18a}.

\section{Minimal Brackets and Bracket Patterns} 
\label{subsection:residual_brackets:categories_generated_by_residual_brackets}

From Proposition~\ref{proposition:generation_of_categories_by_residual_brackets} we know that categories $\mc C\subseteq \Cppnb$ are generated by their sets $\mc C\cap \resbr$ of residual brackets. In this section we improve on this result by showing that subcategories $\mc C$ of $\mc S_0$ are determined by their sets of \emph{minimal brackets} (Proposition~\ref{corollary:generation_by_minimal_brackets}), proper subsets of $\mc C\cap \resbr$. In addition, we prove that the set of minimal brackets contained in a given category is closed under three generic operations. In fact, in Proposition~\ref*{proposition:devolving}~\ref{item:devolving-item_1} and \ref{item:devolving-item_3} we will see that considering which minimal brackets beget which by means of these transformations is sufficient to prove the remaining Parts~\ref{item:main_1-2a} and~\ref{item:main_1-3} of Main Theorem~\ref{theorem:main_1} and Main Theorem~\ref{theorem:main_2}.

\subsection{Minimal Brackets}
With the aim of refining Proposition~\ref{proposition:generation_of_categories_by_residual_brackets} we advance from considering residual brackets to studying \emph{minimal brackets}.
\begin{definition}
  We call a bracket $p$ with lower row $S$ \emph{minimal} if $p$ is connected and dualizable and if $S$ contains exactly one turn and starts with a $\bullet$-colored point. The set of all minimal brackets is denoted by $\minbr$.
\end{definition}
It is immediate from the definition that minimal brackets are residual of the second kind. In fact, residual brackets of the first kind are of no concern to us since we want to classify subcategories of $\mc S_0$:
\begin{lemma}
  \label{lemma:residual_brackets_in_S_zero}
  Let $p\in \Cppnb$ be a residual bracket. Then, $p\in \mc S_0$ if and only if $p$ is residual of the second kind.
\end{lemma}
\begin{proof}
 If $p$ is residual of the first kind, then the property that $\inte(S)$ contains no turns necessitates $\sigma_p(S)\neq 0$, implying $p\notin \mc S_0$. On the other hand, any residual bracket of the second kind is in $\mc S_0$ by Definition~\ref{definition:dualizable}.
\end{proof}

Given a bracket $p\in\Cppnb$ with lower row $S$, note an important difference in the definitions of $p$ being residual of the second kind and of $p$ being minimal: In both cases we ask a set to contain exactly one turn, but in the former this set is $\inte(S)$, while in the latter it is $S$. The following lemma explains the difference between the two classes of brackets in detail. Recall that we denote the color inversion of $q\in \Cp$ by $\overline q$, see Section~\ref{subsection:operations}.

\begin{lemma}
  \label{lemma:characterization_of_residual_brackets} 
The residual brackets $\resbr\cap\mathcal S_0$ of $\mc S_0$ are precisely the partitions
  \begin{gather*}
    \PartBracketBWBW,\quad\PartBracketWBWB,\quad p,\quad \overline p,\quad
 \trmw(p),\quad \trmw(\overline p),\quad\trms(p),\quad\text{and}\quad
\trms(\overline p)
  \end{gather*}
  for minimal brackets $p\in \minbr$.
\end{lemma}
\begin{proof}
Let $q\in \resbr\cap\mathcal S_0$ have lower row $S$. Lemma~\ref{lemma:residual_brackets_in_S_zero} shows that $q$ is residual of the second kind. Hence $\inte(S)$ contains exactly one turn. That means that at most three turns can exist in all of $S$, two of which must intersect $\partial S$. If $S$ has just one turn, then either $q$ or $\overline q$ is minimal already. So, suppose that $S$ has three turns. If $S$ has four points then, $q=\PartBracketBWBW$ or $q=\PartBracketWBWB$. Hence, let $S$ have at least six points. As $q$ is verticolor-reflexive, either $\inte(S)$ is a subsector of $S$ or $\partial (\inte(S))$ intersects through blocks exclusively. In the first case we can write $q=\trmw(p)$ for some bracket $p$, in the second $q=\trms(p)$. The partition $p$ inherits being residual of the second kind from $q$. If the lower row of $p$  starts with $\bullet$-colored point, $p$ is minimal; otherwise $\overline p$ is.
\end{proof}
In the light of the preceding result and  Lemma~\hyperref[item:bracket_operations-item_6]{\ref*{lemma:bracket_operations}~\ref*{item:bracket_operations-item_6}} we make the following distinction.  
\begin{definition}
  \label{definition:monoid_case_non_monoid_case}
  Let $\mathcal C\subseteq \mathcal S_0$ be a category of partitions. 
  \begin{enumerate}[label=(\alph*),labelwidth=!]
  \item We say that $\mathcal C$ is in the \emph{monoid case} if $\PartHalfLibWBW\notin \mathcal C$.
  \item We say that $\mathcal C$ is in the \emph{non-monoid case} if $\PartHalfLibWBW\in \mathcal C$.
  \end{enumerate}
\end{definition}
Then, we can draw the ensuing conclusion from the above characterization of the set $\resbr\cap\mc S_0$, refining Proposition~\ref{proposition:generation_of_categories_by_residual_brackets}.
\begin{proposition}
  \label{corollary:generation_by_minimal_brackets_and_their_color_inversions}
  Let $\mc C\subseteq \mc S_0$ be a category.
  \begin{enumerate}[label=(\alph*)]
  \item If $\mc C$ is in the monoid case, then $\mc C=\langle \mc C\cap (\minbr \cup \overline{\minbr})\rangle$.
    \item If $\mc C$ is in the non-monoid case, then $\mc C=\langle \mc C\cap (\minbr \cup \overline{\minbr}),\, \PartHalfLibWBW \,\rangle$.
  \end{enumerate}
\end{proposition}
\begin{proof}
  By Proposition~\ref{proposition:generation_of_categories_by_residual_brackets} holds $\mc C=\langle \mc C\cap\resbr\rangle$. Now, Lemma~\ref{lemma:bracket_operations} together with Lemma~\ref{lemma:characterization_of_residual_brackets} yields the result.
\end{proof}

One goal of the following subsection is to show that we can omit $\overline {\minbr}$ in the above result, see Proposition~\ref{corollary:generation_by_minimal_brackets}.

\subsection{Bracket Patterns and Their Induced Partitions} To speak about individual minimal brackets, we introduce the language of \emph{bracket patterns}.
In the following, note our convention $0\notin \pint$ and $0\in \nnint$. 
\begin{definition}
  \label{definition:bracket_patterns}
  \begin{enumerate}[label=(\alph*),labelwidth=!]
  \item A \emph{bracket pattern} is a non-empty finite subset $w$ of $\mathbb N$. 
    \item If $w$ is a bracket pattern, we call $\|w\|:=\max( w)$ the \emph{frame of $w$}.
 \item
   For every bracket pattern $w$ and every color $c\in \{\circ,\bullet\}$, we define \emph{ the $c$-bracket $\mathrm{Br}_c(w)$ of $w$} as the unique residual bracket of the second kind with $2(\|w\|+1)$ points in its lower row $S$, exactly one turn in $S$ and with the property that, if we label the points of the left half of $S$ right to left from $0$ to $\|w\|$, the block at position $k$ crosses the horizontal symmetry axis if $k\notin w$ and crosses the vertical symmetry axis if $k\in w$.
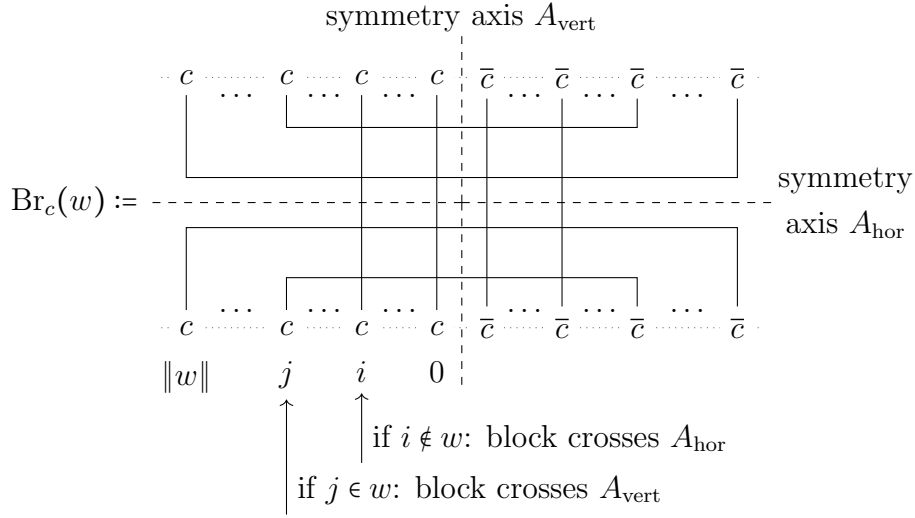
     \begin{align*}
    \mathrm{Br}_c(w)\coloneq
    \begin{tikzpicture}[scale=0.666,baseline=2.35cm*0.666]
      \draw[dotted, gray] (-0.5,0) -- (11.5,0);
      \draw[dotted, gray] (-0.5,5) -- (11.5,5);
      \node[fill=white] (x1) at (0,0) {$c$};
      \node[fill=white] (x2) at (11,0) {$\overline c$};
      \node[fill=white] (x3) at (0,5) {$c$};
      \node[fill=white] (x4) at (11,5) {$\overline c$};
      \node[fill=white] (y1) at (5,0) {$c$};
      \node[fill=white] (y2) at (6,0) {$\overline c$};
      \node[fill=white] (y3) at (5,5) {$c$};
      \node[fill=white] (y4) at (6,5) {$\overline c$};
      \node[fill=white] (z1) at (2,0) {$c$};
      \node[fill=white] (z2) at (9,0) {$\overline c$};
      \node[fill=white] (z3) at (2,5) {$c$};
      \node[fill=white] (z4) at (9,5) {$\overline c$};
      \node[fill=white] (w1) at (3.5,0) {$c$};
      \node[fill=white] (w2) at (7.5,0) {$\overline c$};
      \node[fill=white] (w3) at (3.5,5) {$c$};
      \node[fill=white] (w4) at (7.5,5) {$\overline c$};
      \draw (x1) -- ++(0,2) -| (x2);
      \draw (x3) -- ++(0,-2) -| (x4);
      \draw (y1) to (y3);
      \draw (y2) to (y4);
      \draw (z1) -- ++(0,1) -| (z2);
      \draw (z3) -- ++(0,-1) -| (z4);
      \draw (w1) to (w3);
      \draw (w2) to (w4);
      \draw[dashed] (5.5,2.5) -- (5.5,-1.2);
      \draw[dashed] (5.5,2.5) -- (5.5,5.9);
      \draw[dashed] (5.5,2.5) -- (-0.8,2.5);
      \draw[dashed] (5.5,2.5) -- (11.8,2.5);      
      \node at (5.5,6.2) {symmetry axis $A_{\mathrm{vert}}$};
      \node[align=center, right] at (11.6,2.5) {symmetry\\ axis $A_{\mathrm{hor}}$};      
      \node [below = 2pt of x1] {$\|w\|$};
      \node [below = 2pt of z1] (u1) {$j$};
      \node [below = 2pt of w1] (u2) {$i$};
      \node [below = 2pt of y1] {$0$};
      \path (x1)  -- node [pos=0.5, above=2pt] {$\ldots$} (z1);
      \path (z1)  -- node [pos=0.5, above=2pt] {$\ldots$} (w1);
      \path (w1)  -- node [pos=0.5, above=2pt] {$\ldots$} (y1);
      \path (x3)  -- node [pos=0.5, below=2pt] {$\ldots$} (z3);
      \path (z3)  -- node [pos=0.5, below=2pt] {$\ldots$} (w3);
      \path (w3)  -- node [pos=0.5, below=2pt] {$\ldots$} (y3);
      \path (x2)  -- node [pos=0.5, above=2pt] {$\ldots$} (z2);
      \path (z2)  -- node [pos=0.5, above=2pt] {$\ldots$} (w2);
      \path (w2)  -- node [pos=0.5, above=2pt] {$\ldots$} (y2);
      \path (x4)  -- node [pos=0.5, below=2pt] {$\ldots$} (z4);
      \path (z4)  -- node [pos=0.5, below=2pt] {$\ldots$} (w4);
      \path (w4)  -- node [pos=0.5, below=2pt] {$\ldots$} (y4);
      \node  (l1) [below = 1.2cm of u1, right] {if $j\in w$: block crosses $A_{\mathrm{vert}}$};
      \draw[->] (l1.south west) -- (u1.south);
      \node  (l2) [below = 0.6cm of u2, right] {if $i\notin w$: block crosses $A_{\mathrm{hor}}$};
       \draw[->] (l2.south west) -- (u2.south);
    \end{tikzpicture}
     \end{align*}
  \end{enumerate}
\end{definition}
In this notation we can characterize the minimal brackets as follows.
\begin{lemma}
  \label{lemma:characterization_of_minimal_brackets}
  It holds $\minbr=\{\mathrm{Br}_\bullet(w)\mid w\text{ bracket pattern}\}$.
\end{lemma}

   We define three operations on bracket patterns.
   \begin{definition}
     \label{definition:bracket-operations-first}
     Let $w$ and $w'$ be bracket patterns.
  \begin{enumerate}[label=(\alph*),labelwidth=!]
  \item Let the \emph{superposition  of $w$ and $w'$} be given simply by union of sets, \[w\cup w':= \{i\mid i\in w\text{ or }i\in w'\}.\]
  \item For all $j\in w$, denote by \[\cap_jw:=\{i\mid i \in w, i\leq j\}\] the \emph{$j$-projection of $w$}.
      \item Lastly, define the \emph{dual of $w$} by \[w^\dagger:=\{\|w\|-i\mid i\in \mathbb N_0, i<\|w\|,i\notin w \}.\]
  \end{enumerate}
\end{definition}
To relate these operations on bracket patterns to category operations on the associated minimal brackets we need the following technical result.

\begin{lemma}
  \label{lemma:bracket-pattern-technicalities}
  Let $w$ be an arbitrary bracket pattern.
  \begin{enumerate}[label=(\alph*)]
  \item\label{lemma:bracket-pattern-technicalities-1} It holds $\|w\|=\|w^\dagger\|$.
  \item\label{lemma:bracket-pattern-technicalities-2} It holds $(w^\dagger)^\dagger=w$.
  \item\label{lemma:bracket-pattern-technicalities-3} It holds $(w\cup (w^\dagger))^\dagger=w\cap (w^\dagger)$.
  \end{enumerate}
\end{lemma}
\begin{proof}
  \begin{enumerate}[wide,label=(\alph*)]
  \item Since $0\notin w$, the definition of $w^\dagger$ implies $\|w\|=\|w^\dagger\|$.
  \item  Let $i\in \pint$ with $i<\|w\|$ be arbitrary. Part~\ref{lemma:bracket-pattern-technicalities-1} shows that it suffices to prove that $i\in (w^\dagger)^\dagger$ if and only if $i\in w$. Because $0<i<\|w\|$, by definition of $(w^\dagger)^\dagger$ the statement $i\in (w^\dagger)^\dagger$ is equivalent to there existing $j\in \pint$ with $j<\|w\|$ such that $j\notin w^\dagger$ and $i=\|w\|-j$. In other words, $i\in (w^\dagger)^\dagger$ if and only if $\|w\|-i\notin w^\dagger$.
And, by definition of $w^\dagger$, an index $j\in \pint$ with $j<\|w\|$ satisfies $j\notin w^\dagger$ if and only if for all $k\in \pint$ with $k<\|w\|$ holds $k\in w$ whenever  $j=\|w\|-k$. Applying this to $j\eqpd \|w\|-i$ shows that $i\in (w^\dagger)^\dagger$ if and only if $i\in w$.
  \item Once more, Part~\ref{lemma:bracket-pattern-technicalities-1} allows us to confine ourselves to proving $i\in (w \cup (w^\dagger))^\dagger$ if and only if $i\in w\cap (w^\dagger)$ for all $i\in \pint$ with $i<\|w\|$. As seen in the proof of Part~\ref{lemma:bracket-pattern-technicalities-2}, for such $i$ the statement $i\in (w \cup (w^\dagger))^\dagger$ is equivalent to $\|w\|-i\notin w\cup (w^\dagger)$. This in turn is the same as saying both $\|w\|-i\notin w$ and $\|w\|-i\notin w^\dagger$. Using Part~\ref{lemma:bracket-pattern-technicalities-2}, we can reformulate the last statement equivalently as $\|w\|-i\notin (w^\dagger)^\dagger$ and $\|w\|-i\notin w^\dagger$. By the definitions of $(w^\dagger)^\dagger$ and $w^\dagger$, this is then true if and only if $i\in w^\dagger$ and $i\in w$, i.e.\ $i\in w\cap (w^\dagger)$.
    \qedhere
  \end{enumerate}
\end{proof}
Note that, as opposed to Claims~\ref{item:minimal_bracket_operations-item_1} and \ref{item:minimal_bracket_operations-item_3}, the colors on the left and the right hand side of the identities in Claim~\ref{item:minimal_bracket_operations-item_2} of the following lemma do not agree. Claim~\ref{item:minimal_bracket_operations-item_4} remedies that.
\begin{lemma} 
    \label{lemma:minimal_bracket_operations}
  Let $w$ and $w'$ be bracket patterns and $c\in\colors$.
  \begin{enumerate}[label=(\alph*),labelwidth=!]
  \item \label{item:minimal_bracket_operations-item_1}   It holds $\brp_c(w\cup w')\in \langle \brp_c(w),\brp_c(w') \rangle$.
  \item \label{item:minimal_bracket_operations-item_3} For all $j\in w$ holds $ \brp_c(\cap_jw)\in \langle \brp_{c}(w)\rangle$.
  \item \label{item:minimal_bracket_operations-item_2} It holds $\brp_{\overline c}(w^\dagger)=\brp_{c}(w)^\dagger$. Hence, $\langle \brp_{\overline c}(w^\dagger)\rangle=\langle \brp_c(w)\rangle$.
  \item \label{item:minimal_bracket_operations-item_4} It holds  $\langle \brp_\bullet(w)\rangle=\langle \brp_\circ(w)\rangle$. 
  \end{enumerate}
\end{lemma}
\begin{proof}
  \begin{enumerate}[wide, label=(\alph*)]
      \item We can assume  $\|w'\|\leq\|w\|$. Then, the pairing \[(\mathrm{Br}_c(w),\idpt{c}^{\otimes (\|w\|-\|w'\|)}\otimes \mathrm{Br}_c(w')\otimes \idpt{\overline c}^{\otimes (\|w\|-\|w'\|)})\] is composable and the composition equals $\mathrm{Br}_c(w\cup w')$.
      \item The set of subsectors of the lower row $S$ of $\mathrm{Br}_c(w)$ is totally ordered by $\subseteq$. Let $S_j$ denote the $j$-th smallest subsector of $S$. Then the identity $B(\mathrm{Br}_c(w),S_j)=\mathrm{Br}_c(\cap_jw)$ in conjunction with Lemma~\ref{lemma:associated_brackets} proves the claim.
              \item  This is clear from the definitions.
\item
  We prove the claim by induction over the frame $\|w\|$ of $w$. For $\|w\|=1$, the assertion is true because  \(\brp_\bullet(\{1\})=\PartBracketBBWW\) and \(\brp_\circ(\{1\})=\PartBracketWWBB\) are duals of each other.
  So, let $\|w\|\geq 2$ and suppose that the claim holds for all bracket patterns whose frame is  at most $\|w\|-1$. Let $c\in\colors$ be arbitrary. We show $\brp_{\overline c}(w)\in \mc C\eqpd \langle \brp_{c}(w)\rangle$.
  \par
  If $w=w^\dagger$, then Part~\ref{item:minimal_bracket_operations-item_2} shows $\mr{Br}_{\overline c}(w)=\mr{Br}_{c}(w)^\dagger\in \mc C$. Hence, we can assume $w\backslash (w^\dagger)\neq \emptyset$ or $w^\dagger\backslash w\neq \emptyset$. We first treat the most involved case  of both $w\backslash (w^\dagger)\neq \emptyset$ and $w^\dagger\backslash w\neq \emptyset$.
  Define in this situation the three bracket patterns
  \begin{align*}
    w_1\eqpd \cap_{\|w\backslash (w^\dagger)\|} w \quad\text{and}\quad w_2\eqpd \cap_ {\|w^\dagger\backslash w\|}(w^\dagger)
  \end{align*}
  and
  \begin{align*}
    w_3\eqpd (w_1  \cup (w^\dagger))^\dagger \cup w_2.
  \end{align*}
  In two steps we prove first $\brp_{\overline c}(w_3^\dagger)\in \mc C$ and then $w=w_3^\dagger$.
  \par
  \textbf{Step 1:} 
  Part~\ref{item:minimal_bracket_operations-item_3} implies $\brp_c(w_1)\in \mc C$.  From $\|w\|\in w\cap (w^\dagger)$ we infer $\|w_1\|=\|w\backslash (w^\dagger)\|<\|w\|$. Hence,  by the induction hypothesis it follows $\brp_{\overline c}(w_1)\in \mc C$. By Part~\ref{item:minimal_bracket_operations-item_2} we know $\brp_{\overline c}(w^\dagger)\in \mc C$ and can thus conclude $\brp_{\overline c}(w_1\cup (w^\dagger))\in \mc C$ thanks to Part~\ref{item:minimal_bracket_operations-item_1}. In turn, that implies $\brp_{ c}((w_1\cup (w^\dagger))^\dagger)\in \mc C$ by Part~\ref{item:minimal_bracket_operations-item_2}. \par
  Moreover, Part~\ref{item:minimal_bracket_operations-item_3} shows $\brp_{\overline c}(w_2)\in \mc C$. Again, since $\|w_2\|=\|w^\dagger \backslash w\|<\|w^\dagger\|=\|w\|$, the induction hypothesis yields  $\brp_{c}(w_2)\in \mc C$.
  \par
  Combining the previous two deductions, Part~\ref{item:minimal_bracket_operations-item_1} ensures $\brp_c(w_3)=\brp_c((w_1\cup (w^\dagger))^\dagger\cup w_2)\in \mc C$. Lastly, Part~\ref{item:minimal_bracket_operations-item_2} implies $\brp_{\overline c}(w_3^\dagger)\in \mc C$.
  \par
  \textbf{Step 2:} First, we establish $w_1\cup (w^\dagger)=w\cup (w^\dagger)$. The inclusion $w_1\cup (w^\dagger)\subseteq w\cup (w^\dagger)$ is clear since $w_1\subseteq w$ by definition. So, let $i\in w\cup (w^\dagger)$ be arbitrary. We can assume $i\notin w^\dagger$, i.e.\ $i\in w\backslash (w^\dagger)$. Then, $i\in w$ and $i\leq \|w\backslash (w^\dagger)\|$ prove $i\in w_1$ by definition. Thus, $w_1\cup (w^\dagger)=w\cup (w^\dagger)$ as claimed.
  \par
  So,  we have proven $w_3=(w\cup (w^\dagger))^\dagger \cup w_2$.  By Lemma~\hyperref[lemma:bracket-pattern-technicalities-3]{\ref*{lemma:bracket-pattern-technicalities}~\ref*{lemma:bracket-pattern-technicalities-3}} we conclude $w_3=(w\cap (w^\dagger))\cup w_2$. Hence, by Lemma~\hyperref[lemma:bracket-pattern-technicalities-2]{\ref*{lemma:bracket-pattern-technicalities}~\ref*{lemma:bracket-pattern-technicalities-2}}, it only remains to verify $w^\dagger=(w\cap (w^\dagger)) \cup w_2$. As $w_2\subseteq w^\dagger$, the inclusion $(w\cap (w^\dagger)) \cup w_2\subseteq w^\dagger$ is true. Conversely, let $i\in w^\dagger$ be arbitrary. We can assume $i\notin w\cap (w^\dagger)$, i.e.\ $i\in w^\dagger\backslash w$. Now, $i\in w^\dagger$ and $i\leq \|w^\dagger\backslash w\|$ implies $i\in w_2$. \par
  That concludes the proof for the case $w\backslash (w^\dagger)= \emptyset$ and $w^\dagger\backslash w\neq \emptyset$. If $w\backslash (w^\dagger)= \emptyset$, replace $w_1$ with $w$, and, if $w\backslash (w^\dagger) \neq \emptyset$, replace $w_2$ with $w^\dagger$. Then, a similar and simpler argument proves the claim. \qedhere
  \end{enumerate}
\end{proof}
Finally, with the help of the preceding lemma, we obtain our sought improvement of Proposition~\ref{proposition:generation_of_categories_by_residual_brackets}. 
\begin{proposition}
  \label{corollary:generation_by_minimal_brackets}
  Let $\mc C\subseteq \mc S_0$ be a category.
  \begin{enumerate}[label=(\alph*)]
  \item If $\mc C$ is in the monoid case, then $\mc C=\langle \mc C\cap \minbr\rangle$.
    \item If $\mc C$ is in the non-monoid case, then $\mc C=\langle \mc C\cap \minbr,\, \PartHalfLibWBW \,\rangle$.
  \end{enumerate}
\end{proposition}
\begin{proof}
  Combine Proposition~\ref{corollary:generation_by_minimal_brackets_and_their_color_inversions}, Lemma~\ref{lemma:characterization_of_minimal_brackets} and Lemma~\hyperref[item:minimal_bracket_operations-item_4]{\ref*{lemma:minimal_bracket_operations}~\ref*{item:minimal_bracket_operations-item_4}}.
\end{proof} 
Lastly, we note how the set of color distances occurring between the blocks of a minimal bracket can be expressed using the language of bracket patterns.
\begin{definition}
  \label{definition:completion-first-time}
For every bracket pattern $w$, define the \emph{completion of $w$} by \[A(w):=\left\{ j-i\mid j\in w, i\in \mathbb N_0, i\notin w, i<j  \right\}.\]
\end{definition}
No harm will come from overloading the symbol $A$ from Definition~\ref{definition:A}, as the following lemma shows.
\begin{lemma}
  \label{lemma:minimal_bracket_A}
  For all bracket patterns $w$ holds \[A(\mr{Br}_\bullet(w))=A(w).\]
\end{lemma}
\begin{proof}
  Follows immediately from Definitions~\ref{definition:A}, \ref{definition:bracket_patterns} and~\ref{definition:completion-first-time}.
\end{proof}

\subsection{Bracket Patterns of a Category}
\label{subsection:residual_brackets:bracket_patterns_of_a_category}
While 
Lemma~\ref{lemma:minimal_bracket_operations} was crucial for proving Proposition~\ref{corollary:generation_by_minimal_brackets} by showing that minimal brackets generate the same categories as their color inversions and vice versa, it also motivates the following notion.
\begin{definition}
  \label{definition:bracket-pattern-category-first}
  Let $\mathfrak W$ be a set of bracket patterns.
  \begin{enumerate}
  \item We call $\mathfrak W$ a \emph{bracket pattern category} if it is closed under superposition, dualisation and projections.
\item
  By $\llangle \mathfrak W\rrangle$ we denote the bracket pattern category generated by the set $\mathfrak W$.
    \end{enumerate}
\end{definition}

\begin{definition}
  \label{definition:bracket_patterns_of_category}
  For every category $\mathcal C\subseteq \mathcal S_0$, we call the set
  \begin{align*}
    \mathfrak B_\mathcal C&:=\{ w\mid w\text{ bracket pattern}, \mathrm{Br}_\bullet(w)\in \mathcal C\}
  \end{align*}
   the \emph{bracket patterns of $\mathcal C$}.
 \end{definition}

In the light of Lemma~\ref{lemma:characterization_of_minimal_brackets}, the above definition implies $\mc C\cap \minbr=\{\mr{Br}_\bullet(w)\mid w\in \mathfrak B_\mathcal C\}$ for all categories $\mc C\subseteq \mc S_0$. Thus, we draw the following conclusion.

\begin{proposition}
  \label{proposition:devolving}
  Let $\mathcal C\subseteq \mathcal S_0$ be a category.
  \begin{enumerate}[label=(\alph*),labelwidth=!]
  \item \label{item:devolving-item_1}
    The bracket patterns $\mathfrak B_\mathcal C$ of $\mathcal C$ form a bracket pattern category.
\item \label{item:devolving-item_3}
  Let $\mathfrak G_\mathcal C$ be a set of bracket patterns satisfying $\mathfrak B_\mathcal C=\llangle \mathfrak G_\mathcal C\rrangle$.
  \begin{enumerate}[label=(\arabic*),labelwidth=!]
  \item If $\mathcal C$ is in the monoid case, then \(      \mathcal C=\left\langle \mathrm{Br}_\bullet(w)\mid w\in\mathfrak G_\mathcal C\right\rangle\).
  \item If $\mathcal C$ is in the non-monoid case, then \(\mathcal C=\left\langle \mathrm{Br}_\bullet(w),\PartHalfLibWBW\mid w\in\mathfrak G_\mathcal C\right\rangle\).
  \end{enumerate}
  \end{enumerate}
\end{proposition}
\begin{proof}
  \begin{enumerate}[wide, labelwidth=!, label=(\alph*)]
  \item  This is the combined result of all parts of Lemma~\ref{lemma:minimal_bracket_operations}.

  \item First, recognize that for all sets $\mathfrak S$ of bracket patterns holds
      \begin{align*}
        \left\{\mathrm{Br}_\bullet(w)\mid w\in \llangle \mathfrak S\rrangle \right\}\subseteq \left\langle \mathrm{Br}_\bullet(w) \mid w\in \mathfrak S\right\rangle,
      \end{align*}
      again thanks to Lemma~\ref{lemma:minimal_bracket_operations}.
         Using this for the first inclusion below, we find
    \begin{align*}
      \left\{ \mathrm{Br}_\bullet(w)\mid w\in \mathfrak B_{\mathcal C}\right\}&=\left\{ \mathrm{Br}_\bullet(w)\mid w\in \llangle \mathfrak G_\mathcal C\rrangle\right\}\\
                                                                                &\subseteq \left\langle  \mathrm{Br}_\bullet(w) \mid w\in \mathfrak G_\mathcal C\right \rangle\\
                                                                                &\subseteq \left\langle   \mathrm{Br}_\bullet(w) \mid w\in \mathfrak B_{\mathcal C}\right \rangle\\
      &\subseteq \mathcal C.
    \end{align*}
    Applying Proposition~\ref{corollary:generation_by_minimal_brackets}, we infer, if $\mathcal C$ is in the monoid case, 
    \begin{align*}
      \mathcal C=\left\langle \mathrm{Br}_\bullet(w)\mid w\in \mathfrak B_{\mathcal C}\right\rangle\subseteq \left\langle  \mathrm{Br}_\bullet(w) \mid w\in \mathfrak G_\mathcal C\right \rangle\subseteq \mathcal C.
    \end{align*}
    and, if $\mathcal C$ is in the non-monoid case,
        \begin{align*}
      \mathcal C=\left\langle \mathrm{Br}_\bullet(w),\PartHalfLibWBW \mid w\in \mathfrak B_{\mathcal C}\right\rangle\subseteq \left\langle  \mathrm{Br}_\bullet(w),\PartHalfLibWBW \mid w\in \mathfrak G_\mathcal C\right \rangle\subseteq \mathcal C,
        \end{align*}
        which is what we needed to show.\qedhere
  \end{enumerate}

\end{proof}
The preceding result demonstrates that there exists an injection from the set of subcategories of $\mc S_0$ on the one hand to two copies of the set of all bracket pattern categories on the other hand (one for the monoid case and one for the non-monoid case).  In addition, Proposition~\ref{proposition:devolving} reveals that we can find generators of a given category $\mc C\subseteq\mc S_0$ by determining a generator of the corresponding bracket pattern category $\mathfrak B_\mc C$. In Section~\ref{section:generators_and_classification}, we will prove the injection to be  surjective as well by combining Propositions~\ref{theorem:category_property_of_I_N} and~\ref{proposition:devolving} with the results of the following section. That will return the full classification of subcategories of $\mc S_0$.

\section{Classification of Bracket Pattern Categories}
\label{section:bracket_pattern_categories}
In Proposition~\ref{proposition:devolving} we saw that the problem of classifying subcategories of $\mathcal S_0$ and finding their generators can be devolved to the analogous tasks for another, simpler class of combinatorial objects, \emph{bracket pattern categories} (see Subsection~\ref{subsection:bracket_patterns_and_their_categories}). These two reduced problems are the object of this section. Consequently, it mentions no partitions at all and can be read entirely independently from the rest of the article.

\subsection{Bracket Patterns and their Categories}
\label{subsection:bracket_patterns_and_their_categories}
For the convenience of the reader we repeat the relevant definitions (Definition~\ref{definition:bracket_patterns}, \ref{definition:bracket-operations-first} and~\ref{definition:bracket-pattern-category-first}). Mind $0\notin \pint$, whereas $\nnint\eqpd \pint\cup\{0\}$.
\begin{definition}
  \begin{enumerate}[label=(\alph*),labelwidth=!]
    \item A \emph{bracket pattern} is a non-empty finite subset of $\pint$ and $\|w\|$ denotes its largest element, its \emph{frame}.
    \item   We define three operations on bracket patterns $w$, $w'$.
  \begin{enumerate}[labelwidth=!]
  \item Let the \emph{superposition  of $w$ and $w'$} be given simply by union of sets, \[w\cup w':= \{i\mid i\in w\text{ or }i\in w'\}.\]
  \item For all $j\in w$, denote by \[\cap_jw:=\{i\mid i \in w, i\leq j\}\] the \emph{$j$-projection of $w$}.
  \item Lastly, define the \emph{dual of $w$} by \[w^\dagger:=\{\|w\|-i\mid i\in \mathbb N_0, i<\|w\|,i\notin w \}.\]    
  \end{enumerate}
\item A set $\mathfrak W$ of bracket patterns is called a \emph{bracket pattern category} if it is closed under superposition, dualisation and projections in the above sense.
\item
For every set $\mathfrak W$ of bracket patterns, denote by $\llangle \mathfrak W\rrangle$ the bracket pattern category generated by $\mathfrak W$. If $\mathfrak{W}=\{w\}$ for a bracket pattern $w$, we slightly abuse notation by writing $\llangle w\rrangle$ instead of $\llangle \{w\}\rrangle$.
    \end{enumerate}
\end{definition}
Note that $\emptyset$ is a bracket pattern category but not a bracket pattern. We repeat Lemma~\ref{lemma:bracket-pattern-technicalities} stating elementary facts about these operations.
\begin{lemma}
  \label{lemma:bracket-pattern-technicalities-second}
  Let $w$ be an arbitrary bracket pattern.
  \begin{enumerate}[label=(\alph*)]
  \item\label{lemma:bracket-pattern-technicalities-second-1} It holds $\|w\|=\|w^\dagger\|$.
  \item\label{lemma:bracket-pattern-technicalities-second-2} It holds $(w^\dagger)^\dagger=w$.
  \item\label{lemma:bracket-pattern-technicalities-second-3} It holds $(w\cup (w^\dagger))^\dagger=w\cap (w^\dagger)$.
  \end{enumerate}
\end{lemma}

\subsection{Completion of Bracket Patterns}
Reiterating Definition~\ref{definition:completion-first-time}, the following mapping will be key to classifying bracket pattern categories.
\begin{definition}
For every bracket pattern $w$, define the \emph{completion of $w$} by \[A(w):=\left\{ j-i\mid j\in w, i\in \mathbb N_0, i\notin w, i<j  \right\}.\]
\end{definition}

The next lemma shows especially that $A$ is extensive. 
Eventually it will be proven that $A$ is idempotent. However, note that, in general, $w\subseteq w'$ does \emph{not} imply $A(w)\subseteq A(w')$ for bracket patterns $w$ and $w'$. For example, $\{4\}\subseteq \{1,2,4\}$, but $A(\{4\})=\{1,2,3,4\}\not\subseteq \{1,2,4\}=A(\{1,2,4\})$.
\begin{lemma}
  \label{lemma:necessary_range_elements}
  For every bracket pattern $w$ holds \[ w\subseteq A(w), \quad 1=\min(A(w)) \quad\text{ and }\quad \|A(w)\|=\|w\|.\]
  In particular, $A(w)\subseteq A(A(w))$.
\end{lemma}
\begin{proof}
  The facts that $w\subseteq A(w)$, that $1\leq \min(A(w))$ and that $\max(A(w))=\|w\|$ are clear from the definition of $A(w)$. If $w=\{i\mid i\in \mathbb N, i\leq \|w\|\}$, then $1\in A(w)$. So suppose there exists $i\in \mathbb N$ such that $i\leq \|w\|$ and $i\notin w$. We can choose $i$ to be maximal such. Then, $i<\|w\|$ and $i+1\in w$. It follows $1=(i+1)-i\in A(w)$ by definition of $A(w)$.
\end{proof}

We now show, amongst other things, that bracket pattern categories are invariant under the completion operation. This takes a few auxiliary results, some of which will also be used later on for other purposes.
\begin{lemma}
  \label{lemma:range_extensivity_of_bracket_pattern_operations}
  Let $w$ and $w'$ be bracket patterns and let $j\in w$ be arbitrary.
  \begin{enumerate}[label=(\alph*),labelwidth=!]
  \item\label{lemma:range_extensivity_of_bracket_pattern_operations-1} It holds $A(w\cup w')\subseteq A(w)\cup A(w')$.
  \item\label{lemma:range_extensivity_of_bracket_pattern_operations-3} It holds $A(\cap_jw)\subseteq A(w)$.
  \item\label{lemma:range_extensivity_of_bracket_pattern_operations-2} It holds $A(w^\dagger)=A(w)$.
  \end{enumerate}
\end{lemma}
\begin{proof}
  Claims~\ref{lemma:range_extensivity_of_bracket_pattern_operations-1} and~\ref{lemma:range_extensivity_of_bracket_pattern_operations-3} are immediate from the definition. To prove Part~\ref{lemma:range_extensivity_of_bracket_pattern_operations-2} it suffices to show $A(w^\dagger)\subseteq A(w)$ since $(w^\dagger)^\dagger=w$. So, let $j\in w^\dagger$ and $i\in\nnint$ with $i\notin w^\dagger$ and $i<j$, i.e.\ $j-i\in A(w^\dagger)$, be arbitrary. By definition of $w^\dagger$ we find $k \in \nnint$ with $k<\|w\|$, $k\notin w$ and $j=\|w\|-k$. Then, $i<j$ implies $k<\|w\|-i$. Moreover, since $\|w^\dagger\|=\|w\|$  and $i\notin w^\dagger$, we can infer $\|w\|-i\in (w^\dagger)^\dagger=w$. Hence, $j-i=(\|w\|-k)-i=(\|w\|-i)-k$ yields $j-i\in A(w)$ as $\|w\|-i\in w$, $k\notin w$ and $k<\|w\|-i$. 
\end{proof}

The next two propositions are central tools in the proof of the classification of bracket pattern categories. The first one holds for all sets of bracket patterns, the latter only for bracket pattern categories. Recall $\bigcup X\eqpd \bigcup_{Y\in X}Y$ for all systems $X$ of sets.
\begin{proposition}
  \label{proposition:range_of_generated_category}
  For all sets $\mathfrak W$ of bracket patterns holds \[\bigcup A(\llangle \mathfrak W\rrangle)=\bigcup A(\mathfrak W).\]
  Especially, $\bigcup A(\llangle w\rrangle)=A(w)$ for every bracket pattern $w$.
\end{proposition}
\begin{proof}
   Lemma~\ref{lemma:range_extensivity_of_bracket_pattern_operations} shows $\bigcup A(\llangle \mathfrak W\rrangle)\subseteq \bigcup A (\mathfrak W)$. The other inclusion holds since $\mathfrak W\subseteq \llangle \mathfrak W\rrangle$.
 \end{proof}
 
 \begin{proposition}
   \label{proposition:characterization_of_the_range}
  If $\mathfrak W$ is a bracket pattern category, then \[\bigcup A(\mathfrak W)=\{\|w\|\mid w\in \mathfrak W\}.\]
\end{proposition}
\begin{proof}
  The inclusion \(\{\|w\|\mid w\in \mathfrak W\}\subseteq \bigcup A(\mathfrak W)\) follows from Lemma~\ref{lemma:necessary_range_elements}. Conversely, let $w\in \mathfrak W$, $j\in w$ and $i\in  \mathbb N_0$ with $i<\|w\|$, $i\notin w$ and $i<j$, i.e. $j-i\in A(w)$, be arbitrary. From $j\in w$ follows $(\cap_j w)^\dagger\in \mathfrak W$. Now, $i\notin w$ and $i<j$ imply $j-i\in (\cap_j w)^\dagger$. Hence, \[j-i=\|\cap_{j-i}((\cap_j w)^\dagger)\|\] and $\cap_{j-i}((\cap_j w)^\dagger)\in \mathfrak W$ prove the claim.
\end{proof}
\subsection{Generators of Bracket Pattern Categories}
We now see, using these two results, that bracket pattern categories are indeed closed under the completion operation.
\begin{lemma}
  \label{lemma:range_construction}
  For all bracket patterns $w$ holds \[A(w)=\bigcup \llangle w\rrangle\in \llangle w\rrangle.\]
\end{lemma}
\begin{proof}
  The set \[\tilde w:=\bigcup \llangle w\rrangle\] is a bracket pattern contained in the category $\llangle w\rrangle$ because $\llangle w\rrangle$ is finite and closed under superposition.  We show $\tilde w=A(w)$.\par
  According to Propositions~\ref{proposition:range_of_generated_category} and \ref{proposition:characterization_of_the_range},
\[\{ \|w'\|\mid w'\in \llangle w\rrangle\}=\bigcup A(\llangle w\rrangle)=A(w).\]
 As $\|w'\|\in w'\subseteq \tilde w$ for every $w'\in \llangle w\rrangle$, we conclude $\tilde w\supseteq A(w)$.\par Conversely, $\tilde w\in\llangle w\rrangle$ implies $A(\tilde w)\in A(\llangle w\rrangle)$ and thus $A(\tilde w)\subseteq \bigcup A(\llangle w\rrangle)=A(w)$. By Lemma~\ref{lemma:necessary_range_elements} holds $\tilde w\subseteq A(\tilde w)$. Thus we have  shown $\tilde w\subseteq A(w)$, which completes the proof.
\end{proof}
In the following most important step of the classification we show that the completion allows us to give a full characterization of all singly generated bracket pattern categories. 
\begin{lemma}
  \label{lemma:characterization_of_generated_bracket_pattern_category}
  For every bracket pattern $w$ holds \[\llangle w\rrangle=\{ w'\mid w'\text{ bracket pattern}, A(w')\subseteq A(w)\}.\]
\end{lemma}
\begin{proof}
  It is clear that $\llangle w\rrangle\subseteq \{w'\mid w'\text{ bracket pattern}, A(w')\subseteq A(w)\}$ because of Proposition~\ref{proposition:range_of_generated_category}. The reverse inclusion we prove in two steps. Let $w'$ be a bracket pattern with $A(w')\subseteq A(w)$.
  \par
  \textbf{Case~1:}
    First, assume $\|w'\|=\|w\|$. Then, we define
    \begin{align*}
      \tilde w:= \left( A(w)^\dagger \cup \bigcup_{i\in A(w)\backslash w'}(\cap_{\|w\|-i}A(w))^\dagger \right)^\dagger
    \end{align*}
    and show
    \begin{enumerate}[label=(\roman*)]
    \item For all $ i\in A(w)\backslash w'$ holds $\|w\|-i\in A(w)$  and $\|\cap_{\|w\|-i}A(w)\|=\|w\|-i$.
    \item$\tilde w\in\llangle w\rrangle$.
    \item $\|\tilde w\|=\|w\|$.
      \item $w'=\tilde w$.
    \end{enumerate}
  Together \ref{item:characterization_of_generated_bracket_pattern_category-item_1ii} and \ref{item:characterization_of_generated_bracket_pattern_category-item_1iv} then prove $w'=\tilde w\in \llangle w\rrangle$.
  \begin{enumerate}[wide,label=(\roman*)]
  \item \label{item:characterization_of_generated_bracket_pattern_category-item_1i}
    Let $i\in A(w)\backslash w'$ be arbitrary. Then, $\|w\|-i\in A(w')$ since we have assumed $\|w\|=\|w'\|\in w'$ and $i\notin w'$. Now, the other assumption $A(w')\subseteq A(w)$ implies the first part $\|w\|-i\in A(w)$ of Claim~\ref{item:characterization_of_generated_bracket_pattern_category-item_1i}. The remaining assertion $\|\cap_{\|w\|-i}A(w)\|=\|w\|-i$ is then clear by definition of the projection $\cap_{\|w\|-i}A(w)$.
  \item \label{item:characterization_of_generated_bracket_pattern_category-item_1ii}
    By Lemma~\ref{lemma:range_construction} we know $A(w)\in \llangle w\rrangle$. Using  Claim~\ref{item:characterization_of_generated_bracket_pattern_category-item_1i} we can then conclude  $\bigcap_{\|w\|-i}A(w)\in \llangle w\rrangle$ for all $i \in A(w)\backslash w'$, and thus $\tilde w\in\llangle w\rrangle$ because $\llangle w\rrangle$ is a bracket pattern category.
  \item \label{item:characterization_of_generated_bracket_pattern_category-item_1iii}
Lemma~\hyperref[lemma:bracket-pattern-technicalities-second-1]{\ref*{lemma:bracket-pattern-technicalities-second}~\ref*{lemma:bracket-pattern-technicalities-second-1}} shows  $\|(\cap_{\|w\|-i}A(w))^\dagger\|=\|\cap_{\|w\|-i}A(w)\|$ for all $i\in A(w)\backslash w'$, which implies
    \begin{align*}
      \left\|\bigcup_{i\in A(w)\backslash w'}(\cap_{\|w\|-i}A(w))^\dagger\right\|=\max_{i\in A(w)\backslash w'}(\|w\|-i)< \|w\|
    \end{align*}
    by Claim~\ref{item:characterization_of_generated_bracket_pattern_category-item_1i} and Lemma~\ref{lemma:necessary_range_elements}.
    In contrast, Lemma~\hyperref[lemma:bracket-pattern-technicalities-second-1]{\ref*{lemma:bracket-pattern-technicalities-second}~\ref*{lemma:bracket-pattern-technicalities-second-1}} guarantees $\|A(w)^\dagger\|=\|A(w)\|=\|w\|$. So, both together imply $\|\tilde w^\dagger\|=\|w\|$ and thus, using Lemma~\hyperref[lemma:bracket-pattern-technicalities-second-1]{\ref*{lemma:bracket-pattern-technicalities-second}~\ref*{lemma:bracket-pattern-technicalities-second-1}} one last time, $\|\tilde w\|=\|w\|$, as claimed.
  \item \label{item:characterization_of_generated_bracket_pattern_category-item_1iv}
    Let  $j\in \pint$ with $j<\|w\|$ be arbitrary.  We show $j\in \tilde w$ if and only if $j\in w'$.
For all bracket patterns $w_0$ and all $k\in \pint$ with $k<\|w_0\|$ holds by definition of the dual:
    $k\notin w_0^\dagger$ if and only if  $\|w_0\|-k\in w_0$. Hence, for
all $i\in A(w)\backslash w'$ such that $i<j$ follows by Claim~\ref{item:characterization_of_generated_bracket_pattern_category-item_1i} 
    \begin{align*}
      \|w\|-j\notin (\cap_{\|w\|-i}A(w))^\dagger \iff (\|w\|-i)-(\|w\|-j)\in \cap_{\|w\|-i}A(w)\iff j-i\in A(w).
    \end{align*}
    At the same time, $\|w\|-j\in (\cap_{\|w\|-j}A(w))^\dagger$. Combining the two, we conclude for all $i\in A(w)\backslash w'$
    \begin{align*}
      \|w\|-j\notin (\cap_{\|w\|-i}A(w))^\dagger \iff i\neq j \text{ and }(i<j\implies j-i\in A(w)).
    \end{align*}
    Thus, using Claim~\ref{item:characterization_of_generated_bracket_pattern_category-item_1iii} we infer 
    \begin{align*}
      j\in \tilde w&\iff \|w\|-j\notin A(w)^\dagger\text{ and for all } i\in A(w)\backslash w': \|w\|-j\notin (\cap_{\|w\|-i}A(w))^\dagger\\
      &\iff j\in A(w)\text{ and for all }i\in A(w)\backslash w': i\neq j\text{ and } (i<j\implies j-i\in A(w)).
    \end{align*}
    We prove now that this last statement is equivalent to $j\in w'$.\par
    On the one hand, if $j\in A(w)$ and $i\neq j$ for all $i\in A(w)\backslash w'$, then $j\in w'$ since $w'\subseteq A(w')\subseteq A(w)$ by Lemma~\ref{lemma:necessary_range_elements} and assumption on $w'$. Hence, $j\in \tilde w$ implies $j\in w'$.\par
    Conversely, if $j\in w'$, then  it holds both $j\in A(w)$ and $j-i\in A(w')\subseteq A(w)$ for all $i\in A(w)\backslash w'$ with $i<j$ by definition of $A(w')$ and, again, our assumption on $w'$. So, $j\in w'$ also implies $j\in \tilde w$.
    \par That concludes the proof of Claim~\ref{item:characterization_of_generated_bracket_pattern_category-item_1iv} and thus  shows $w'\in \llangle w\rrangle$ in case $\|w'\|=\|w\|$.
 \end{enumerate}
\par
\textbf{Case~2:}
Now, suppose $\|w'\|\neq\|w\|$. Lemma~\ref{lemma:necessary_range_elements} and the assumption  $A(w')\subseteq A(w)$ guarantee $\|w'\|=\|A(w')\|\leq \|A(w)\|$. Hence and because $\|w'\|\in A(w')\subseteq A(w)$, we obtain a well-defined bracket pattern of $\llangle A(w)\rrangle$ by putting
\begin{align*}
  \tilde w\eqpd \cap_{\|w'\|}A(w).
\end{align*}
We now apply the result of Case~1 to $(w',A(\tilde w))$ here in the roles of $(w',w)$ there. This is possible for the following two reasons: Firstly, $\|w'\|=\|\tilde w\|=\|A(\tilde w)\|$ by Lemma~\ref{lemma:necessary_range_elements}. Secondly, $A(w')\subseteq A(A(\tilde w))$ holds as well:
From $A(w')\subseteq \tilde w$ because $A(w')\subseteq A(w)$, $\tilde w\subseteq A(w)$ and $\|A(w')\|=\|w'\|$ follows $A(w')\subseteq \tilde w \subseteq A(\tilde w)\subseteq A(A(\tilde w))$  where we have employed Lemma~\ref{lemma:necessary_range_elements} twice. So, indeed, Case~1 yields $w'\in \llangle A(\tilde w)\rrangle$. We conclude since $\tilde w\in \llangle A(w)\rrangle$
\begin{align*}
  w'\in \llangle A(\tilde w)\rrangle\subseteq \llangle \tilde w\rrangle\subseteq \llangle A(w)\rrangle\subseteq \llangle w\rrangle,
\end{align*}
using Lemma~\ref{lemma:range_construction} for the first and last inclusions. That was to be proven.
\end{proof}

We can generalize Lemma~\ref{lemma:characterization_of_generated_bracket_pattern_category} to characterize all bracket pattern categories. In addition, for finite bracket pattern categories which are finite we can identify a canonical generator.
\begin{proposition}
  \label{proposition:characterization_of_bracket_pattern_categories}
  Let  $\mathfrak W$ be an arbitrary bracket pattern category.
  \begin{enumerate}[label=(\alph*),labelwidth=!]
  \item\label{item:characterization_of_bracket_pattern_categories-1}
    It holds \[\mathfrak W=\left\{ w\mid w\text{ bracket pattern}, A(w)\subseteq \bigcup A(\mathfrak W)\right\}.\]
  \item\label{item:characterization_of_bracket_pattern_categories-2} If $\mathfrak W$ is finite, then $\bigcup \mathfrak W\in \mathfrak W$ and \[\mathfrak W=\llangle \bigcup\mathfrak W\rrangle.\]
              Especially, finite bracket pattern categories are singly generated.
    \end{enumerate}
\end{proposition}
\begin{proof}
  \begin{enumerate}[labelwidth=!,label=(\alph*), wide]
  \item  The inclusion $\mathfrak W\subseteq \left\{ w\mid w\text{ bracket pattern}, A(w)\subseteq \bigcup A(\mathfrak W)   \right\}$ is clear by definition. Conversely, let $w$ be a bracket pattern with $A(w)\subseteq \bigcup A(\mathfrak W)$. By Proposition~\ref{proposition:characterization_of_the_range}, $A(w)\subseteq \{\|w'\|\mid w'\in \mathfrak W\}$. Because $A(w)$ is finite, there exists a finite set $\mathfrak W_w\subseteq \mathfrak W$ of bracket patterns such that $A(w)\subseteq \{\|w'\|\mid w'\in \mathfrak W_w\}$. As $\mathfrak W_w$ is finite and $\mathfrak W$ is closed under superpositions, $\bigcup \mathfrak W_w\in \mathfrak W$. Because $\{\|w'\|\mid w'\in \mathfrak W_w\}\subseteq \bigcup \mathfrak W_w$ by definition, $A(w)\subseteq \bigcup \mathfrak W_w$. By Lemma~\ref{lemma:necessary_range_elements} then, $A(w)\subseteq \bigcup \mathfrak W_w \subseteq A(\bigcup \mathfrak W_w )$. Lemma~\ref{lemma:characterization_of_generated_bracket_pattern_category} hence shows $w\in \llangle \bigcup \mathfrak W_w\rrangle$. Now, $\bigcup \mathfrak W_w\in \mathfrak W$ shows $\llangle \bigcup \mathfrak W_w\rrangle\subseteq \mathfrak W$ and thus $w\in \mathfrak W$, which is what we needed to prove.
  \item
 Since $\mathfrak W$ is closed under finite superpositions, $\bigcup \mathfrak W\in \mathfrak W$, which also implies $\llangle \bigcup \mathfrak W\rrangle\subseteq \mathfrak W$. For the converse inclusion we use the characteriztation of $\llangle \mathfrak W\rrangle$ from Lemma~\ref{lemma:characterization_of_generated_bracket_pattern_category}. Hence, by Part~\ref{item:characterization_of_bracket_pattern_categories-1}   all we have to prove is $\bigcup A(\mathfrak W)\subseteq A(\bigcup \mathfrak W)$. But by Lemma~\ref{lemma:necessary_range_elements} we know $\{\|w\|\mid w\in \mathfrak W\}\subseteq \bigcup \mathfrak W\subseteq A(\bigcup \mathfrak W)$ and Proposition~\ref{proposition:characterization_of_the_range} completes the proof.\qedhere
    \end{enumerate}
  \end{proof}
  \begin{remark}
    The proof of Proposition~\hyperref[item:characterization_of_bracket_pattern_categories-1]{\ref*{proposition:characterization_of_bracket_pattern_categories}~\ref*{item:characterization_of_bracket_pattern_categories-1}} actually reveals $\bigcup A(\mathfrak W)=A(\bigcup \mathfrak W)$ for all finite bracket pattern categories $\mathfrak W$ since $A(\bigcup \mathfrak W)\in A(\mathfrak W)$.
  \end{remark}
As announced, we prove that the completion operation is idempotent, a fact we will need later.
\begin{lemma}
  \label{lemma:range_idempotency}
  For all bracket patterns $w$ holds $A(A(w))=A(w)$. In particular, $\llangle A(w)\rrangle=\llangle w\rrangle$.
\end{lemma}
\begin{proof}
  Combining $A(w)\in  \llangle w\rrangle$, by Lemma~\ref{lemma:range_construction}, and, thus, $A(A(w))\subseteq \bigcup A(\llangle w\rrangle)=A(w)$ by Proposition~\ref{proposition:range_of_generated_category}  shows $A(A(w))\subseteq A(w)$. Lemma~\ref{lemma:necessary_range_elements} guarantees $A(w)\subseteq A(A(w))$, from which the claim now follows. Finally, $\llangle A(w)\rrangle=\llangle w\rrangle$ holds by Lemma~\ref{lemma:characterization_of_generated_bracket_pattern_category}.
\end{proof}
\subsection{Link to Submonoids of $(\nnint, +)$}
The definition of the completion gives it essential properties which help to reveal the nature of bracket pattern categories in the following result. Recall that a \emph{monoid} is a semigroup with neutral element.

\begin{lemma}
  \label{lemma:semigroup_property}
  \begin{enumerate}[label=(\alph*),labelwidth=!]
  \item   \label{item:semigroup_property-1}
    If $w$ is a bracket pattern, $\mathbb N_0\backslash A(w)$ is a submonoid of $(\mathbb N_0,+)$.
  \item   \label{item:semigroup_property-2}
    For every submonoid $M$ of $(\mathbb N_0,+)$ such that $\mathbb N_0\backslash M$ is finite holds \[A(\mathbb N_0\backslash M)=\mathbb N_0\backslash M.\]
    \end{enumerate}
    Especially, for a bracket pattern $w$ holds $w=A(w)$ if and only if $w=\nnint\backslash M$ for some submonoid $M$ of $(\mathbb N_0,+)$.
\end{lemma}
\begin{proof}
  \begin{enumerate}[label=(\alph*),labelwidth=!,wide]
    \item
  First, note that $0\notin A(w)$ and hence $0\in \nnint\backslash A(w)$. Now, let  $x,y\in \mathbb N_0\backslash A(w)$ be arbitrary. If $\|w\|<x+y$, then $x+y\in \mathbb N_0\backslash A(w)$ holds by Lemma~\ref{lemma:necessary_range_elements}. So, suppose $1\leq x+y\leq \|w\|$ and $1\leq y$. Then, $x<x+y$. Hence, if $x+y\in A(w)$ held, it would follow $y=(x+y)-x\in A(A(w))$ by definition of $A(A(w))$, contradicting $y\notin A(w)$, see Lemma~\ref{lemma:range_idempotency}. In conclusion, $x+y\in \mathbb N_0\backslash A(w)$.
\item
  By Lemma~\ref{lemma:necessary_range_elements}, we only need to  show $A(\mathbb N_0\backslash M)\subseteq \mathbb N_0\backslash M$: Suppose $j\in \mathbb N_0\backslash M$, $i\in \mathbb N_0$, $i\notin \mathbb N_0\backslash M$ and $i<j$.  If $j-i\in M$ were true, then, $M$ being a semigroup, the assumption $i\in M$ would imply $j=(j-i)+i\in M$, contradicting the other premise $j\notin M$. Hence, $j-i\in \mathbb N_0\backslash M$, which is what we needed to show.\qedhere
  \end{enumerate}
\end{proof}
The preceding result easily generalizes to arbitrary bracket pattern categories.
\begin{proposition}
  \label{proposition:semigroup_property_for_categories}
  For every bracket pattern category $\mathfrak W$, the set $\mathbb N_0\backslash \bigcup A(\mathfrak W)$ is a submonoid of $(\mathbb N_0,+)$.
\end{proposition}
\begin{proof}
  As $\mathbb N_0\backslash \bigcup A(\mathfrak W)=\bigcap_{w\in \mathfrak W} \left(\mathbb N_0\backslash A(w)\right)$ and as intersection respects the submonoid structure, Lemma~\hyperref[item:semigroup_property-1]{\ref*{lemma:semigroup_property}~\ref*{item:semigroup_property-1}} proves the claim.
\end{proof}
We will now show that the following sets comprise all possible bracket pattern categories and identify their generators.
\begin{definition}
  \label{definition:bracket_pattern_category_of_a_semigroup}
  For every submonoid $M$ of $(\mathbb N_0,+)$, we call
  \begin{align*}
    \mathfrak W_M:= \{ w\mid w \text{ bracket pattern}, A(w)\subseteq \mathbb N_0\backslash M\}
  \end{align*}
  the \emph{bracket pattern category of $M$}.
\end{definition}
For the following lemma, note that we allow set differences $X\backslash Y$ even if $Y\not\subseteq X$.

\begin{lemma}
  \label{proposition:classification_of_bracket_pattern_categories}
  Let $\mathcal M$ denote the set of submonoids of $(\mathbb N_0,+)$.
  \begin{enumerate}[label=(\alph*), labelwidth=!] 
  \item \label{item:classification_of_bracket_pattern_categories-item_1}  For every $M\in \mc M$, the system $\mathfrak W_M$ is a bracket pattern category. 
    \item \label{item:classification_of_bracket_pattern_categories-item_3} 
    It holds $\mathfrak W_\nnint=\emptyset$. If $M\neq \mathbb N_0$, then
    \[\mathfrak W_M=\llangle \mathbb N_0\backslash M\rrangle \quad\text{if }\mathbb N_0\backslash M \text{ is finite,}\]and \[\mathfrak W_M=\llangle \mathbb \{0,\ldots,v\}\backslash M\mid v\in \mathbb N\rrangle \quad\text{if }\mathbb N_0\backslash M \text{ is infinite.}\]
  \item \label{item:classification_of_bracket_pattern_categories-item_4} For every $M\in \mathcal M$ holds \[M=\nnint\backslash \bigcup A(\mathfrak W_M)=\nnint\backslash \bigcup\mathfrak W_M. \]
  \end{enumerate}
\end{lemma}
\begin{proof}
  \begin{enumerate}[wide,label=(\alph*)]
  \item     Lemma~\ref{lemma:range_extensivity_of_bracket_pattern_operations} shows that $\mathfrak W_M$ is a bracket pattern category for every submonoid $M\in \mathcal M$.
  \item The defining condition $A(w)\subseteq \emptyset$ of $\mathfrak W_{\nnint}$ cannot be satisfied by any bracket pattern $w$. Hence, $\mathfrak W_{\nnint}=\emptyset$. Now, let $M\in \mc M\backslash\{\nnint\}$ be arbitrary.
    \par
    \textbf{Case~1:}
    First, let $\nnint\backslash M$ be finite. In that case, the definition of $\mathfrak W_M$ and Lemma~\hyperref[item:semigroup_property-2] {\ref*{lemma:semigroup_property}~\ref*{item:semigroup_property-2}} imply \[\mathfrak W_M=\{w\mid w\text{ bracket pattern}, A(w)\subseteq A(\mathbb N_0\backslash M) \}.\]
    Lemma~\ref{lemma:characterization_of_generated_bracket_pattern_category} hence allows us to conclude $\mathfrak W_{M}=\llangle \mathbb N_0\backslash M\rrangle$.\par
    \textbf{Case~2:}
If, on the other hand, $\nnint\backslash M$ is infinite, we argue as follows: Let $v\in \mathbb N$ be arbitrary. The sets $M_v:=\{0\}\cup \{i\mid i\in \mathbb N, i>v\}$ and $M\cup M_v$ are submonoids of $(\mathbb N_0,+)$. In addition, $\mathbb N_0\backslash ( M\cup M_v)=\{0,\ldots,v\}\backslash M$ is finite.
We conclude \[\mathfrak W_{M\cup M_v}=\llangle\{0,\ldots,v\}\backslash M\rrangle\]
by Case~1.
Moreover, by definition holds
\begin{align*}
  \mathfrak W_{M\cup M_v}=\mathfrak W_M\cap \mathfrak W_{M_v}.
\end{align*}
This and the fact that
\begin{align*}
  \mathfrak W_{M_v}
  =\{ w\mid w\text{ bracket pattern}, \|w\|\leq v\},
\end{align*}
which holds by definition of $\mathfrak W_{M_v}$ and Lemma~\ref{lemma:necessary_range_elements},
imply
\begin{align*}
  \mathfrak W_M&=\mathfrak W_M\cap\bigcup_{v\in \mathbb N}\mathfrak W_{M_v}
                 = \bigcup_{v\in \mathbb N}\left(\mathfrak W_M\cap\mathfrak W_{M_v}\right)= \bigcup_{v\in \mathbb N}\mathfrak W_{M\cup M_v}= \bigcup_{v\in \mathbb N}\llangle \{0,\ldots,v\}\backslash M\rrangle\\
  &\subseteq \llangle \{0,\ldots,v\}\backslash M\mid v\in \mathbb N\rrangle.
\end{align*}
Conversely, by  Lemma~\hyperref[item:semigroup_property-2] {\ref*{lemma:semigroup_property}~\ref*{item:semigroup_property-2}} and definition of $\mathfrak W_M$, the inclusion
\begin{align*}
  A(\{0,\ldots,v\}\backslash M)=A(\mathbb N_0\backslash (M\cup M_v))=\mathbb N_0\backslash (M\cup M_v)=\{0,\ldots,v\}\backslash M\subseteq \mathbb N_0\backslash M
\end{align*}
proves $\{0,\ldots,v\}\backslash M\in \mathfrak W_M$ for every  $v\in \mathbb N$. We have thus proven \[\mathfrak W_M=\llangle \{0,\ldots,v\}\backslash M\mid v\in \pint\rrangle\] as claimed.
  \item From $\mathfrak W_{\mathbb N_0}=\emptyset$ follows $\nnint=\nnint\backslash \bigcup A(\emptyset)=\nnint\backslash A(\mathfrak W_{\nnint})$. Now, let $M\in \mc M\backslash\{\nnint\}$ be arbitrary.
    Once more, we distinguish two cases.
    \par
    \textbf{Case~1:}
    Suppose $\nnint\backslash M$ is finite.     Then, it follows
  \begin{align*}
    \mathbb N_0\backslash M=A(\mathbb N_0\backslash M)=\bigcup A(\llangle \mathbb N_0\backslash M\rrangle)=\bigcup A(\mathfrak W_M).
  \end{align*}
  by Lemma~\hyperref[item:semigroup_property-2] {\ref*{lemma:semigroup_property}~\ref*{item:semigroup_property-2}}, Proposition~\ref{proposition:range_of_generated_category} and Claim~\ref{item:classification_of_bracket_pattern_categories-item_3}.\par
  \textbf{Case~2:}
  Now, let $\nnint\backslash M$ be infinite. Because $A(\{0,\ldots,v\}\backslash M)=\{0,\ldots,v\}\backslash M$ as seen in the proof of Claim~\ref{item:classification_of_bracket_pattern_categories-item_3}, we infer
  \begin{align*}
    \nnint\backslash M&=\bigcup\{\{0,\ldots,v\}\backslash M\mid v\in \pint\}\\
    &=\bigcup \{ A(\{0,\ldots,v\}\backslash M)\mid v\in \pint\}\\                        
    &=\bigcup A(\{ \{0,\ldots,v\}\backslash M\mid v\in \pint\})\\
    &=\bigcup A(\llangle  \{0,\ldots,v\}\backslash M\mid v\in \pint\rrangle)\\
    &=\bigcup A(\mathfrak W_M)
  \end{align*}
  by Proposition~\ref{proposition:range_of_generated_category} and Claim~\ref{item:classification_of_bracket_pattern_categories-item_3}.    \qedhere
    \end{enumerate}
    \end{proof}
    \begin{proposition}
      \label{proposition:characterization_of_bracket_pattern_categories_final}
      There is a one-to-one correspondence between submonoids of $(\nnint,+)$ and bracket pattern categories, given by the map $M\mapsto \mathfrak W_M$.
    \end{proposition} 
    \begin{proof}
      By Lemma~\hyperref[item:classification_of_bracket_pattern_categories-item_1]{\ref*{proposition:classification_of_bracket_pattern_categories}~\ref*{item:classification_of_bracket_pattern_categories-item_1}} the map $M\mapsto \mathfrak W_M$ is well-defined.       By Lemma~\hyperref[item:classification_of_bracket_pattern_categories-item_4]{\ref*{proposition:classification_of_bracket_pattern_categories}~\ref*{item:classification_of_bracket_pattern_categories-item_4}} it is injective. Finally, in  combination, Propositions~\hyperref[item:characterization_of_bracket_pattern_categories-1]{\ref*{proposition:characterization_of_bracket_pattern_categories}~\ref*{item:characterization_of_bracket_pattern_categories-1}} and \ref{proposition:semigroup_property_for_categories} prove that every bracket pattern  category $\mathfrak W$ is of the form $\mathfrak W_{M(\mathfrak W)}$ for the submonoid $M(\mathfrak W):=\mathbb N_0\backslash \bigcup A(\mathfrak W)$.
    \end{proof}
\begin{remark}
  Submonoids $M$ of $(\nnint,+)$ with finite $\nnint\backslash M$, also known as \emph{numerical semigroups}, are interesting as such (see \cite{DGR13} for a survey on open problems in this area). Given a finitely generated bracket pattern category $\mathfrak W$, the set $\nnint\backslash\bigcup  A(\mathfrak W)$ is the largest numerical semigroup disjoint from $\bigcup \mathfrak W$. (In fact, one can show $\bigcup  A(\mathfrak W)=\bigcup \mathfrak W$.) Its gap set is $\bigcup A(\mathfrak W)$, its genus $|A(\bigcup \mathfrak W)|$ and its Frobenius number $\|\bigcup \mathfrak W\|$. 
\end{remark}

\section{Generators \texorpdfstring{of $\mc I_D$ }{}and Classification\texorpdfstring{\\{[Main Theorems~\ref*{theorem:main_1}~{\normalfont\ref*{item:main_1-2a}, \ref*{item:main_1-3}}, \ref*{theorem:main_2}]}}{}}
\label{section:generators_and_classification}
As seen in Proposition~\ref{proposition:devolving}, the mapping $\mc C\mapsto \mathfrak B_\mc C$ (see Definition~\ref{definition:bracket_patterns_of_category}) injects the set of all subcategories $\mc C$ of $\mc S_0$ into two copies of the set of all bracket pattern categories (one for the monoid- and one for the non-monoid case). Moreover, by Proposition~\hyperref[item:devolving-item_3]{\ref*{proposition:devolving}~\ref*{item:devolving-item_3}} for every category $\mc C\subseteq \mc S_0$ each generator of the corresponding bracket pattern category $\mathfrak B_\mc C$ induces a generator of $\mc C$.  We combine now with these results from Section~\ref{subsection:residual_brackets:categories_generated_by_residual_brackets} our knowledge about bracket pattern categories and their generators won in Section~\ref{section:bracket_pattern_categories}.  By showing that the injection $\mc C\to \mathfrak B_\mc C$ is also surjective we thus obtain a full classification of all subcategories of $\mathcal S_0$. At the same time, identifying which bracket pattern categories correspond to which subcategories of $\mc S_0$ yields generators of the latter.\par
For every subsemigroup $D$ of $(\mathbb N_0,+)$, by Proposition~\ref{theorem:category_property_of_I_N} the set $\mc I_D$ is a category and thus the bracket patterns $\mathfrak B_{\mc I_D}$ of $\mc I_D$ are a well-defined bracket pattern category (Proposition~\hyperref[item:devolving-item_1]{\ref*{proposition:devolving}~\ref*{item:devolving-item_1}}).  Moreover, $D\cup \{0\}$ is a submonoid of $(\nnint,+)$, implying that also the set $\mathfrak W_{D\cup \{0\}}$  (Definition~\ref{definition:bracket_pattern_category_of_a_semigroup}) is a bracket pattern category (see Lemma~\hyperref[item:classification_of_bracket_pattern_categories-item_1]{\ref*{proposition:classification_of_bracket_pattern_categories}~\ref*{item:classification_of_bracket_pattern_categories-item_1}}).
\begin{lemma}
  \label{lemma:linking}
  Let $D$ be a subsemigroup of  $(\mathbb N_0,+)$.
  \begin{enumerate}[label=(\alph*)]
  \item\label{item:linking-item_1} The category $\mc I_D$ is in the monoid case if and only if $0\in D$.
  \item\label{item:linking-item_2} It holds $\mathfrak B_{\mc I_D}=\mathfrak W_{D\cup \{0\}}$. 
  \end{enumerate}
\end{lemma}
\begin{proof}
  \begin{enumerate}[label=(\alph*),wide]
  \item We  prove the contraposition of the claim. It holds $A(\PartHalfLibWBW)=\{0\}$, see the example after Definition~\ref{definition:categories_I_N}. By Definition~\ref{definition:monoid_case_non_monoid_case}, the category $\mc I_D$ is in the non-monoid case if and only if $\PartHalfLibWBW\in \mc I_D$. In turn, by Remark~\ref{remark:I_N-with-A} of $\mc I_D$ that is true if and only if $\{0\}=A(\PartHalfLibWBW)\subseteq \nnint\backslash D$. But the latter statement is equivalent to $0\notin D$. 
  \item For every bracket pattern $w$, Remark~\ref{remark:I_N-with-A} and Lemma~\ref{lemma:minimal_bracket_A} show the equivalence that $\mr{Br}_\bullet(w)\in \mathcal I_D$ holds if and only if $A(w)\subseteq \nnint\backslash D$. The latter statement can also be expressed as $A(w)\subseteq \nnint\backslash(D\cup \{0\})$ since $0\notin A(w)$. In other words, we have shown $\mathfrak B_{\mathcal I_D}=\mathfrak W_{D\cup\{0\}}$.\qedhere
\end{enumerate}
\end{proof}
\begin{theorem}
    Let $\mathcal D$ denote the set of subsemigroups of $(\mathbb N_0,+)$.
  \label{theorem:main_theorem}
  \begin{enumerate}[label=(\alph*),labelwidth=!]
  \item \label{item:main_theorem-1} Let $D\in \mathcal D$ be arbitrary.
  \begin{enumerate}
  \item If $0\in D$ and $|\mathbb N\backslash D|<\infty$, then  \(\mathcal I_D=\left\langle \mathrm{Br}_\bullet(\mathbb N \backslash D)\right \rangle.\)
  \item If $0\notin D$  and $|\mathbb N\backslash D|<\infty$, then  \(\mathcal I_D=\left\langle \mathrm{Br}_\bullet(\mathbb N \backslash D),\PartHalfLibWBW\right\rangle.\)     
  \item If $0\in D$ and $|\mathbb N\backslash D|=\infty$, then  \(\mathcal I_D=\left\langle \mathrm{Br}_\bullet(\{1,\ldots,v\}\backslash D)\mid v\in \mathbb N\right\rangle.\)
  \item If $0\notin D$ and $|\mathbb N\backslash D|=\infty$, then  \(\mathcal I_D=\left\langle \mathrm{Br}_\bullet(\{1,\ldots,v\}\backslash D), \PartHalfLibWBW \mid v\in \mathbb N\right\rangle.\)    
  \end{enumerate}
\item \label{item:main_theorem-3}
  The categories $(\mc I_D)_{D\in \mc D}$ are pairwise distinct.
    \end{enumerate}
\end{theorem}
\begin{proof}
  \begin{enumerate}[label=(\alph*),labelwidth=!,wide]
  \item
    Minding $\pint\backslash D=\nnint\backslash(D\cup \{0\})$,  Pro\-po\-si\-tions~\hyperref[item:devolving-item_3]{\ref*{proposition:devolving}~\ref*{item:devolving-item_3}} and ~\hyperref[item:classification_of_bracket_pattern_categories-item_3]{\ref*{proposition:classification_of_bracket_pattern_categories}~\ref*{item:classification_of_bracket_pattern_categories-item_3}} and Lemma~\ref{lemma:linking} prove Claim~\ref{item:main_theorem-1}.
  \item Suppose $D_1,D_2\in \mc D$ and $\mc I_{D_1}=\mc I_{D_2}$.    Lemma~\hyperref[item:linking-item_2]{\ref*{lemma:linking}~\ref*{item:linking-item_2}} showed $\mathfrak B_{\mc I_{D_i}}=\mathfrak W_{D_i\cup \{0\}}$ for all $i\in \{1,2\}$.
    So, from $\mc I_{D_1}=\mc I_{D_2}$ follows $\mathfrak W_{D_1\cup \{0\}}=\mathfrak W_{D_2\cup \{0\}}$. We conclude $D_1\cup \{0\}=D_2\cup \{0\}$ by Proposition~\ref{proposition:characterization_of_bracket_pattern_categories_final}.
  Once more, as seen in Lemma~\hyperref[item:linking-item_1]{\ref*{lemma:linking}~\ref*{item:linking-item_1}}, for all $i\in \{1,2\}$, the category $\mc I_{D_i}$ is in the monoid case if and only if $0\in D_i$.  In other words, $0\in D_1$ if and only if $0\in D_2$. Hence, $\mc I_{D_1}=\mc I_{D_2}$. \qedhere
\end{enumerate}
\end{proof}

\begin{theorem}
      Let $\mathcal D$ denote the set of subsemigroups of $(\mathbb N_0,+)$.
  For every category of partitions $\mathcal C\subseteq \mathcal S_0$ exists $D\in \mathcal D$ such that $\mathcal C=\mathcal I_D$.
  \end{theorem}
  \begin{proof}
    Let $\mc C\subseteq \mc S_0$ be an arbitrary category. The bracket patterns $\mathfrak B_\mc C$ of $\mc C$ in the sense of Definition~\ref{definition:bracket_patterns_of_category} are a bracket pattern category by Proposition~\hyperref[item:devolving-item_1]{\ref*{proposition:devolving}~\ref*{item:devolving-item_1}}. Employing the classification result of  Proposition~\ref{proposition:characterization_of_bracket_pattern_categories_final}, we conclude that there exists a submonoid $M$ of $(\nnint,+)$, i.e.\ $M\in \mc D$ with $0\in M$, such that $\mathfrak B_\mc C=\mathfrak W_{M}$, in the notation of Definition~\ref{definition:bracket_pattern_category_of_a_semigroup}. Put $E\eqpd \emptyset$ if $\mc C$ is in the monoid case and define $E\eqpd \{0\}$ otherwise. Then, $M\backslash E\in\mc D$. We show $\mc C=\mc I_{M\backslash E}$.
    \par
    By definition of $E$ and by Lemma~\hyperref[item:linking-item_1]{\ref*{lemma:linking}~\ref*{item:linking-item_1}} the category $\mc I_{M\backslash E}$ is in the monoid case if and only if $\mc C$ is. Moreover,
   Lemma~\hyperref[item:linking-item_2]{\ref*{lemma:linking}~\ref*{item:linking-item_2}} shows $\mathfrak B_{\mc I_{M\backslash E}}=\mathfrak W_M=\mathfrak B_\mc C$. Hence,  Proposition~\hyperref[item:devolving-item_3]{\ref*{proposition:devolving}~\ref*{item:devolving-item_3}} proves $\mc C=\mc I_{M\backslash E}$. That concludes the proof.\qedhere
\end{proof}
Subsemigroups $D$ of $(\nnint,+)$ with $|\pint\backslash D|=\infty$ are precisely the sets $\emptyset$, $\{0\}$ as well as $n\nnint$ and $n\pint$ for $n\in \pint$ with $n\geq 2$ (see \cite[Chpt.~1]{RG09}). In particular, we can confirm the existence of non-finitely-generated categories.
\begin{corollary}
  Let $D$ be a subsemigroup of $(\nnint,+)$ with $|\pint\backslash D|=\infty$. Then, for all sets $\mc G\subseteq \Cp$ with $\langle \mc G\rangle=\mc I_D$ holds $|\mc G|=\infty$.
\end{corollary}
\begin{proof}
  As seen in Lemma~\hyperref[item:linking-item_2]{\ref*{lemma:linking}~\ref*{item:linking-item_2}},  $\mathfrak{B}_{\mc I_D}=\mathfrak{W}_{D\cup \{0\}}$. That implies $\bigcup A(\{\mr{Br}_\bullet(w)\mid w\in \mathfrak{B}_{\mc I_D}  \})=    \bigcup A(\{\mr{Br}_\bullet(w)\mid w\in \mathfrak{W}_{D\cup \{0\}}  \})$. With the help of Lemma~\ref{lemma:minimal_bracket_A} we conclude $\bigcup A(\{\mr{Br}_\bullet(w)\mid w\in \mathfrak{B}_{\mc I_D}  \})=\bigcup A(\{w \mid w\in \mathfrak W_{D\cup \{0\}}\})=\bigcup A(\mathfrak W_{D\cup \{0\}})$. Now,  Lemma~\ref{proposition:classification_of_bracket_pattern_categories} assures us that $D\cup \{0\}=\nnint\backslash \bigcup A(\mathfrak W_{D\cup \{0\}})$, from which it follows that $\bigcup A(\{\mr{Br}_\bullet(w)\mid w\in \mathfrak{B}_{\mc I_D}  \})=\nnint\backslash (D\cup \{0\})=\pint\backslash D$. We have thus proven $\pint\backslash D\subseteq \bigcup A(\mc I_D)$ by Proposition~\hyperref[item:devolving-item_3]{\ref*{proposition:devolving}~\ref*{item:devolving-item_3}}. So, especially $\bigcup A(\mc I_D)$ is infinite. Let $\mc G\subseteq \Cp$ satisfy $\langle \mc G\rangle =\mc I_D$. By Remark~\ref{lemma:invariance} we infer $\pint\backslash D\subseteq \bigcup A(\mc I_D)= \bigcup A(\langle \mc G \rangle)=   \bigcup A(\mc G)$, proving that $\bigcup A(\mc G)$ is infinite. That is only possible if $\mc G$ is infinite.
\end{proof}
\section{Concluding Remarks}
\label{section:concluding_remarks}

\subsection{Half-Liberations of \texorpdfstring{$U_n$}{the Unitary Group}} Banica and Speicher showed that categories of one-colored partitions correspond to certain quantum subgroups of Wang's (\cite{Wa95}) free orthogonal quantum group $O_n^+$, namely the \emph{orthogonal} easy quantum groups (see \cite{BaSp09}, \cite{We16} and \cite{We17a}). Categories of two-colored partitions are in bijection with so-called \emph{unitary} easy quantum groups (cf.~\cite{TaWe15a}), certain compact quantum subgroups of Wang's (\cite{Wa95}) free \emph{unitary} quantum group $U_n^+$.
\par
Since $O_n^+$ is a quantum version, a kind of \enquote{liberation}, of the orthogonal group $O_n$, a natural question is to find all (orthogonal) easy quantum groups between $O_n$ and $O_n^+$, i.e.\ all \enquote{half-liberations} of $O_n$. Only one (orthogonal) easy quantum group exists here, namely the \emph{half-liberated} orthogonal quantum group $O_n^*$ (\cite{BaSp09}). Here the commutation relations $ab=ba$ holding in $O_n$ are not yet dropped entirely, as is the case eventually in $O_n^+$, but are relaxed to the half-commutation relations $acb=bca$. Equivalently, the category of (one-colored) partitions corresponding to $O_n^*$ is generated by
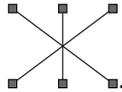
\begin{align*}
  \begin{tikzpicture}[scale=0.666]
    \node [scale=0.4, draw=black, fill=gray] (x1) at (0,0) {};
    \node [scale=0.4, draw=black, fill=gray] (y2) at (2,0) {};
    \node [scale=0.4, draw=black, fill=gray] (y1) at (0,1.5) {};
    \node [scale=0.4, draw=black, fill=gray] (x2) at (2,1.5) {};
    \node [scale=0.4, draw=black, fill=gray] (z1) at (1,0) {};
    \node [scale=0.4, draw=black, fill=gray] (z2) at (1,1.5) {};
    \draw (x1) to (x2);
    \draw (y1) to (y2);
    \draw (z1) to (z2);    
  \end{tikzpicture}.
\end{align*}
Research into generalizing this half-liberation procedure to the unitary case (see \cite{BhDADa11}, \cite{BiDu13}, \cite{BhDADaDa14},  \cite{BaBi17b} and \cite{BaBi17}), where the generators are no longer self-adjoint and thus adjoints, i.e.\ colors, come in, effectively went in the direction of coloring the points of this partition in different ways or, eventually, those of the similar-looking partition
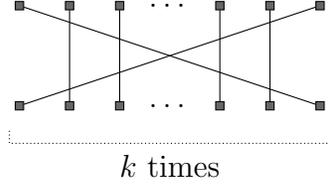
\begin{align*}
  \begin{tikzpicture}[scale=0.666, baseline=-0.111cm]
    \node [scale=0.4, draw=black, fill=gray] (x1) at (1,0) {};
    \node [scale=0.4, draw=black, fill=gray] (y2) at (7,0) {};
    \node [scale=0.4, draw=black, fill=gray] (y1) at (1,2) {};
    \node [scale=0.4, draw=black, fill=gray] (x2) at (7,2) {};
    \node [scale=0.4, draw=black, fill=gray] (z1) at (6,0) {};
    \node [scale=0.4, draw=black, fill=gray] (z2) at (6,2) {};
    \node [scale=0.4, draw=black, fill=gray] (w1) at (2,0) {};
    \node [scale=0.4, draw=black, fill=gray] (w2) at (2,2) {};
    \node [scale=0.4, draw=black, fill=gray] (a1) at (5,0) {};
    \node [scale=0.4, draw=black, fill=gray] (a2) at (5,2) {};
    \node [scale=0.4, draw=black, fill=gray] (b1) at (3,0) {};
    \node [scale=0.4, draw=black, fill=gray] (b2) at (3,2) {};    
    \draw (x1) to (x2);
    \draw (y1) to (y2);
    \draw (z1) to (z2);
    \draw (w1) to (w2);
    \draw (a1) to (a2);
    \draw (b1) to (b2);    
    \draw [densely dotted] (0.8,-0.5) --++(0,-0.25) --++ (6.4,0) -- ++ (0,0.25);
    \node at (4,-1.2) {$k$ times};
    \node at (4,0) {\ldots};
    \node at (4,2) {\ldots};    
  \end{tikzpicture}.
\end{align*}
Essentially, these partitions represent relations $ac_1\ldots c_kb=bc_1\ldots c_ka$  provided certain conditions on the factors $c_1\ldots c_k$  are satisfied.
\par
Note that finding all unitary easy quantum groups between $U_n$ and $U_n^+$, i.e.\ all unitary half-liberations, amounts to classifying all categories $\mc C$ with $\langle \emptyset\rangle\subseteq \mc C\subseteq \langle \PartCrossWW\rangle$, as is the contents of our classification result in Section~\ref{section:main_result}.
\subsection{Comparison with the Previous Research\texorpdfstring{ on Half-Liberations of $U_n$}{}}
Our results obtained in \cite{MaWe18a} and the present article reproduce in combinatorial terms and extend the previous quantum algebraic research on half-liberations of $U_n$:\par\vspace{0.5em}
\begin{enumerate}[labelwidth=!, wide]
  \item In \cite[Def.~5.5]{BhDADa11} and \cite[Definition~2.8]{BhDADaDa14}, Bhowmick, D'Andrea, Das and Dabrowski introduced the first half-liberation of $U_n$, whose algebra they denoted by $A_u^*(n)$. Later, Banica and Bichon wrote $U_n^\times$ for the corresponding quantum group \cite[Def.~3.2 (3)]{BaBi17}. Its associated category is generated by the partition
		\begin{center}\hspace{2em}
		\begin{tikzpicture}[scale=0.666,baseline=0]
		\node [circle, scale=0.4, draw=black, fill=white] (x1-1) at (0,0) {};
		\node [circle, scale=0.4, draw=black, fill=black] (x2-1) at (1,0) {};
		\node [circle, scale=0.4, draw=black, fill=white] (x3-1) at (2,0) {};
		\node [circle, scale=0.4, draw=black, fill=white] (x1-2) at (0,1.5) {};
		\node [circle, scale=0.4, draw=black, fill=black] (x2-2) at (1,1.5) {};
		\node [circle, scale=0.4, draw=black, fill=white] (x3-2) at (2,1.5) {};
		\draw (x1-1) -- (x3-2);
		\draw (x2-1) -- (x2-2);
		\draw (x3-1) -- (x1-2);
		\end{tikzpicture}.
		\end{center}
It holds \[\langle \PartHalfLibWBW  \rangle=\langle \PartHalfLibBWB\rangle=\langle \PartBracketBWBW\rangle=\langle \PartBracketWBWB\rangle\](see \cite[Lemma~6.18~d)]{MaWe18a}), parts of which can be seen as follows:
    
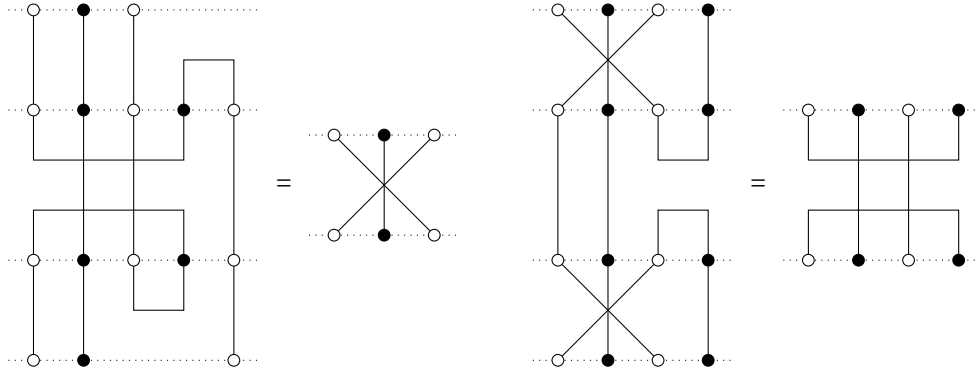
\begin{center}
    \begin{tikzpicture}[scale=0.666]
      \draw [dotted] (-0.5,0) -- (4.5,0);
      \draw [dotted] (-0.5,2) -- (4.5,2);
      \draw [dotted] (-0.5,5) -- (4.5,5);
      \draw [dotted] (-0.5,7) -- (4.5,7);
      \node[scale=0.4,draw=black,circle, fill=white] (z1) at (0,7) {};
      \node[scale=0.4,draw=black,circle, fill=black] (z2) at (1,7) {};
      \node[scale=0.4,draw=black,circle, fill=white] (z3) at (2,7) {};      
      \node[scale=0.4,draw=black,circle, fill=white] (y1) at (0,5) {};
      \node[scale=0.4,draw=black,circle,fill=black] (y2) at (1,5) {};
      \node[scale=0.4,draw=black,circle,fill=white] (y3) at (2,5) {};
      \node[scale=0.4,draw=black,circle,fill=black] (y4) at (3,5) {};
      \node[scale=0.4,draw=black,circle,fill=white] (y5) at (4,5) {};      
      \node[scale=0.4,draw=black,circle,fill=white] (x1) at (0,2) {};
      \node[scale=0.4,draw=black,circle,fill=black] (x2) at (1,2) {};
      \node[scale=0.4,draw=black,circle,fill=white] (x3) at (2,2) {};
      \node[scale=0.4,draw=black,circle,fill=black] (x4) at (3,2) {};
      \node[scale=0.4,draw=black,circle,fill=white] (x5) at (4,2) {};      
      \node[scale=0.4,draw=black,circle,fill=white] (w1) at (0,0) {};
      \node[scale=0.4,draw=black,circle,fill=black] (w2) at (1,0) {};      
      \node[scale=0.4,draw=black,circle,fill=white] (w3) at (4,0) {};
      \draw (w1) to (x1);
      \draw (w2) to (x2);
      \draw (w3) to (x5);      
      \draw (x3) --++(0,-1)-| (x4);
      \draw (x1) --++(0,1)-| (x4);
      \draw (x2) to (y2);
      \draw (x3) to (y3);      
      \draw (x5) to (y5);
      \draw (y1) to (z1);
      \draw (y2) to (z2);
      \draw (y3) to (z3);
      \draw (y4) --++(0,1)-| (y5);
      \draw (y1) --++(0,-1)-| (y4);
      \node at (5,3.5) {$=$};
      \draw [dotted] (5.5,2.5) -- (8.5,2.5);
      \draw [dotted] (5.5,4.5) -- (8.5,4.5);
      \node[scale=0.4,draw=black,circle,fill=white] (a1) at (6,2.5) {};
      \node[scale=0.4,draw=black,circle,fill=black] (a2) at (7,2.5) {};
      \node[scale=0.4,draw=black,circle,fill=white] (a3) at (8,2.5) {};      
      \node[scale=0.4,draw=black,circle,fill=white] (a4) at (6,4.5) {};
      \node[scale=0.4,draw=black,circle,fill=black] (a5) at (7,4.5) {};
      \node[scale=0.4,draw=black,circle,fill=white] (a6) at (8,4.5) {};      
      \draw (a1) to (a6);
      \draw (a2) to (a5);
      \draw (a3) to (a4);
    \end{tikzpicture}
    \quad\quad
    \begin{tikzpicture}[scale=0.666]
      \draw [dotted] (-0.5,0) -- (3.5,0);
      \draw [dotted] (-0.5,2) -- (3.5,2);
      \draw [dotted] (-0.5,5) -- (3.5,5);
      \draw [dotted] (-0.5,7) -- (3.5,7);
      \node[scale=0.4,draw=black,circle,fill=white] (z1) at (0,7) {};
      \node[scale=0.4,draw=black,circle,fill=black] (z2) at (1,7) {};      
      \node[scale=0.4,draw=black,circle,fill=white] (z3) at (2,7) {};
      \node[scale=0.4,draw=black,circle,fill=black] (z4) at (3,7) {};      
      \node[scale=0.4,draw=black,circle,fill=white] (y1) at (0,5) {};
      \node[scale=0.4,draw=black,circle,fill=black] (y2) at (1,5) {};      
      \node[scale=0.4,draw=black,circle,fill=white] (y3) at (2,5) {};
      \node[scale=0.4,draw=black,circle,fill=black] (y4) at (3,5) {};
      \node[scale=0.4,draw=black,circle,fill=white] (x1) at (0,2) {};
      \node[scale=0.4,draw=black,circle,fill=black] (x2) at (1,2) {};      
      \node[scale=0.4,draw=black,circle,fill=white] (x3) at (2,2) {};
      \node[scale=0.4,draw=black,circle,fill=black] (x4) at (3,2) {};
      \node[scale=0.4,draw=black,circle,fill=white] (w1) at (0,0) {};
      \node[scale=0.4,draw=black,circle,fill=black] (w2) at (1,0) {};      
      \node[scale=0.4,draw=black,circle,fill=white] (w3) at (2,0) {};
      \node[scale=0.4,draw=black,circle,fill=black] (w4) at (3,0) {};      
      \draw (w1) to (x3);
      \draw (w3) to (x1);
      \draw (w4) to (x4);
      \draw (x1) to (y1);
      \draw (x3) -- ++(0,1) -| (x4);
      \draw (y3) -- ++(0,-1) -| (y4);
      \draw (y1) to (z3);
      \draw (y3) to (z1);
      \draw (y4) to (z4);
      \draw (w2) to (x2);
      \draw (x2) to (y2);
      \draw (y2) to (z2);      
      \node at (4,3.5) {$=$};
      \draw [dotted] (4.5,5) -- (8.5,5);
      \draw [dotted] (4.5,2) -- (8.5,2);
      \node[scale=0.4,draw=black,circle,fill=white] (a1) at (5,2) {};
      \node[scale=0.4,draw=black,circle,fill=black] (a2) at (6,2) {};
      \node[scale=0.4,draw=black,circle,fill=white] (a3) at (7,2) {};      
      \node[scale=0.4,draw=black,circle,fill=black] (a4) at (8,2) {};      
      \node[scale=0.4,draw=black,circle,fill=white] (b1) at (5,5) {};
      \node[scale=0.4,draw=black,circle,fill=black] (b2) at (6,5) {};
      \node[scale=0.4,draw=black,circle,fill=white] (b3) at (7,5) {};
      \node[scale=0.4,draw=black,circle,fill=black] (b4) at (8,5) {};            
      \draw (a1) --++(0,1) -| (a4);
      \draw (b1) --++(0,-1) -| (b4);      
      \draw (a2) to (b2);
      \draw (a3) to (b3);      
    \end{tikzpicture}
  \end{center}
Hence, the quantum group $U_n^\times$ pertains to the category $\mc I_{\nnint\backslash\{0\}}$ in our notation (see Definition~\ref{definition:categories_I_N}). Alternatively, one could write $\mc I_{\{0\}^c}$ for this category.
\item A second half-liberation of $U_n$ was given by Bichon and Dubois-Violette. In \cite[Ex. 4.10]{BiDu13} they define an algebra $A_u^{**}(n)$ whose intertwiner spaces are generated by the partition
		\begin{gather*}
		\begin{tikzpicture}[scale=0.666,baseline=0]
		\node [circle, scale=0.4, draw=black, fill=white] (x1-1) at (0,0) {};
		\node [circle, scale=0.4, draw=black, fill=white] (x2-1) at (1,0) {};
		\node [circle, scale=0.4, draw=black, fill=white] (x3-1) at (2,0) {};
		\node [circle, scale=0.4, draw=black, fill=white] (x1-2) at (0,1.5) {};
		\node [circle, scale=0.4, draw=black, fill=white] (x2-2) at (1,1.5) {};
		\node [circle, scale=0.4, draw=black, fill=white] (x3-2) at (2,1.5) {};
		\draw (x1-1) -- (x3-2);
		\draw (x2-1) -- (x2-2);
		\draw (x3-1) -- (x1-2);
		\end{tikzpicture}.
		\end{gather*}
In  \cite{BaBi17b} and \cite{BaBi17} Banica and Bichon wrote $U_n^{**}$ for the corresponding quantum group. 
                And $U_n^{\ast\ast}$ appears again as a special case of an entire family of quantum groups introduced by Banica and Bichon in \cite[Def.~7.1]{BaBi17b}. For every $k\in \pint$ (with $0\notin\pint$), their \emph{$k$-half-liberated unitary quantum group} $U_{n,k}^*$ (also  $U_{n,k}$) has its intertwiner spaces generated by:
		
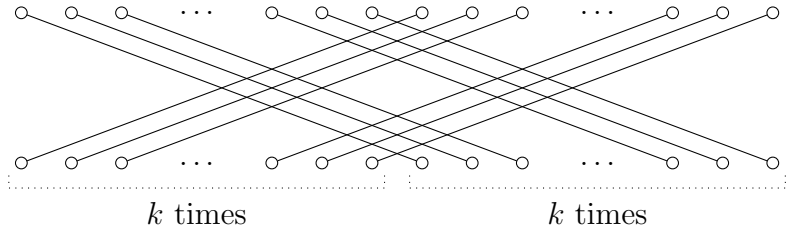
\begin{gather*}
		\begin{tikzpicture}[scale=0.666]
		\node [circle, scale=0.4, draw=black, fill=white] (x1-1) at (0,0) {};
		\node [circle, scale=0.4, draw=black, fill=white] (x2-1) at (1,0) {};
		\node [circle, scale=0.4, draw=black, fill=white] (x3-1) at (2,0) {};						
		\node [circle, scale=0.4, draw=black, fill=white] (y1-1) at (5,0) {};
		\node [circle, scale=0.4, draw=black, fill=white] (y2-1) at (6,0) {};
		\node [circle, scale=0.4, draw=black, fill=white] (y3-1) at (7,0) {};
		\node [circle, scale=0.4, draw=black, fill=white] (y4-1) at (8,0) {};
		\node [circle, scale=0.4, draw=black, fill=white] (y5-1) at (9,0) {};
		\node [circle, scale=0.4, draw=black, fill=white] (y6-1) at (10,0) {};															
		\node [circle, scale=0.4, draw=black, fill=white] (z1-1) at (13,0) {};															
		\node [circle, scale=0.4, draw=black, fill=white] (z2-1) at (14,0) {};															
		\node [circle, scale=0.4, draw=black, fill=white] (z3-1) at (15,0) {};					
		\begin{scope}[yshift=3cm]			
		\node [circle, scale=0.4, draw=black, fill=white] (x1-2) at (0,0) {};
		\node [circle, scale=0.4, draw=black, fill=white] (x2-2) at (1,0) {};
		\node [circle, scale=0.4, draw=black, fill=white] (x3-2) at (2,0) {};						
		\node [circle, scale=0.4, draw=black, fill=white] (y1-2) at (5,0) {};
		\node [circle, scale=0.4, draw=black, fill=white] (y2-2) at (6,0) {};
		\node [circle, scale=0.4, draw=black, fill=white] (y3-2) at (7,0) {};
		\node [circle, scale=0.4, draw=black, fill=white] (y4-2) at (8,0) {};
		\node [circle, scale=0.4, draw=black, fill=white] (y5-2) at (9,0) {};
		\node [circle, scale=0.4, draw=black, fill=white] (y6-2) at (10,0) {};															
		\node [circle, scale=0.4, draw=black, fill=white] (z1-2) at (13,0) {};															
		\node [circle, scale=0.4, draw=black, fill=white] (z2-2) at (14,0) {};															
		\node [circle, scale=0.4, draw=black, fill=white] (z3-2) at (15,0) {};						
		\end{scope}						
		\path (x3-1) edge[draw=none] node {$\ldots$} (y1-1);
		\path (y6-1) edge[draw=none] node {$\ldots$} (z1-1);				
		\path (x3-2) edge[draw=none] node {$\ldots$} (y1-2);
		\path (y6-2) edge[draw=none] node {$\ldots$} (z1-2);								
		\draw (x1-1) -- (y4-2);
		\draw (x2-1) -- (y5-2);
		\draw (x3-1) -- (y6-2);								
		\draw (x1-2) -- (y4-1);
		\draw (x2-2) -- (y5-1);
		\draw (x3-2) -- (y6-1);								
		\draw (z1-1) -- (y1-2);
		\draw (z2-1) -- (y2-2);
		\draw (z3-1) -- (y3-2);								
		\draw (z1-2) -- (y1-1);
		\draw (z2-2) -- (y2-1);
		\draw (z3-2) -- (y3-1);												
		\draw [dotted] (-0.25,-0.25) -- ++(0,-0.25) -- (7.25,-0.5) -- ++ (0,0.25);
		\draw [dotted, xshift=8cm] (-0.25,-0.25) -- ++(0,-0.25) -- (7.25,-0.5) -- ++ (0,0.25);
		\node at (3.5,-1) {$k$ times};
		\node at (11.5,-1) {$k$ times};				
		\end{tikzpicture}
              \end{gather*}
              It holds $U_n=U_{n,1}^*$ and $U_n^{**}=U_{n,2}^*$. In \cite[Section~9]{MaWe18a}, we showed that $U_{n,k}^*$ corresponds to our category $\mathcal S_k$ (see after Main Theorem~\ref{theorem:main_2} of the present article) for every $k\in\mathbb N$. However, the quantum group associated with the category $\mc S_0$ is not considered in \cite{BaBi17b} nor \cite{BaBi17}.
              \item In contrast, another half-liberation of $U_n$ which is studied by  Banica and Bichon is $U_{n}^*$ (\cite[Definition~8.3]{BaBi17b}), also denoted $U_{n,\infty}^*$ (\cite[Definition~4.1~ (3)]{BaBi17}). They obtain it as some quantum algebraic limit case of their series $( U_{n,k}^* )_{k\in\mathbb{N}}$ and show that its associated category is generated by the two partitions
		
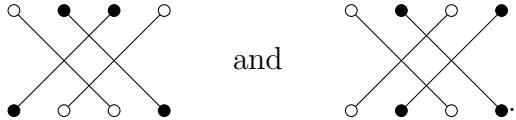
\begin{gather*}
		\begin{tikzpicture}[scale=0.666,baseline=0]
		\node [circle, scale=0.4, draw=black, fill=black] (x1-1) at (0,0) {};
		\node [circle, scale=0.4, draw=black, fill=white] (x2-1) at (1,0) {};
		\node [circle, scale=0.4, draw=black, fill=white] (x3-1) at (2,0) {};
		\node [circle, scale=0.4, draw=black, fill=black] (x4-1) at (3,0) {};
		\node [circle, scale=0.4, draw=black, fill=white] (x1-2) at (0,2) {};
		\node [circle, scale=0.4, draw=black, fill=black] (x2-2) at (1,2) {};
		\node [circle, scale=0.4, draw=black, fill=black] (x3-2) at (2,2) {};
		\node [circle, scale=0.4, draw=black, fill=white] (x4-2) at (3,2) {};
		\draw (x1-1) -- (x3-2);
		\draw (x2-1) -- (x4-2);
		\draw (x3-1) -- (x1-2);
		\draw (x4-1) -- (x2-2);									
		\end{tikzpicture}
		\quad\quad\raisebox{0.555cm}{\text{and}}\quad\quad
		\begin{tikzpicture}[scale=0.666,baseline=0]
		\node [circle, scale=0.4, draw=black, fill=white] (x1-1) at (0,0) {};
		\node [circle, scale=0.4, draw=black, fill=black] (x2-1) at (1,0) {};
		\node [circle, scale=0.4, draw=black, fill=white] (x3-1) at (2,0) {};
		\node [circle, scale=0.4, draw=black, fill=black] (x4-1) at (3,0) {};
		\node [circle, scale=0.4, draw=black, fill=white] (x1-2) at (0,2) {};
		\node [circle, scale=0.4, draw=black, fill=black] (x2-2) at (1,2) {};
		\node [circle, scale=0.4, draw=black, fill=white] (x3-2) at (2,2) {};
		\node [circle, scale=0.4, draw=black, fill=black] (x4-2) at (3,2) {};
		\draw (x1-1) -- (x3-2);
		\draw (x2-1) -- (x4-2);
		\draw (x3-1) -- (x1-2);
		\draw (x4-1) -- (x2-2);									
		\end{tikzpicture}.
              \end{gather*}
              Cyclically rotating both these partition counterclockwise once, reveals them to generate the same category as $\PartBracketBBWW$ and $\PartBracketBWBW$. Since $\langle \PartBracketBWBW\rangle=\langle \PartHalfLibWBW\rangle$ as seen above, it is our category $\mc I_{\mathbb N_0\backslash\{0,1\}}$ (also denoted $\mc I_{\{0,1\}^c}$) that is associated with the quantum group $U_{n}^*$. Especially, $U_{n}^*$ does not correspond to $\mc S_0=\mc I_\emptyset$, which, arguably, is a limit case of $(\mc S_w)_{w\in \pint}$ in a combinatorial sense.
                
            \end{enumerate}
            \par\vspace{0.5em}
Summarizing, the half-liberations of $U_n$ previously known correspond precisely to our categories $(\mathcal S_w)_{w\in \mathbb N}$, $\mathcal I_{\nnint\backslash\{0\}}$ and $\mathcal I_{\nnint\backslash\{0,1\}}$. Equivalently, the categories $\mc S_0=\mc I_\emptyset$ and $\mc I_D$ for $D\in \mc D$ with $D\neq \nnint\backslash\{0\}, \nnint\backslash\{0,1\}$ from \cite[Def.~4.1]{MaWe18a} and Definition~\ref{definition:categories_I_N} were heretofore unknown. Especially we see that the quantum groups $U_n^\times$ and $U_{n}^*$ (in the notation of \cite{BaBi17b}) represent merely special cases of our family $(\mathcal I_D)_{D\in \mc D}$. Moreover, from a combinatorial viewpoint, the half-liberations $(\mathcal S_k)_{k\in \mathbb N}\sim (U_{n,k}^*)_{k\in \pint}$ give rise as a limit case to $\mathcal I_\emptyset=\mathcal S_0$ rather than $\mc I_{\nnint\backslash\{0,1\}}\sim U_{n,\infty}^*=U_n^*$. This is not the only way our results provide a different viewpoint on the previous research into the topic of half-liberating $U_n$.
\subsection{New Perspective on the Half-Liberation Procedure}
\label{subsection:new-perspective}
Our article \cite{MaWe18a} and the present follow-up to it shed a different light on the idea of half-liberating the classical unitary group:
\par
 When considering possible half-liberations of $U_n$, a natural starting point is to simply use the half-commutation relations $abc=cba$ in the unitary case as well. This leads to the category $\mathcal{S}_2=\left\langle\PartHalfLibWWW\right\rangle$.
		For $k\geq 3$, the quantum groups $U_{n,k}^*\sim \mc S_k$ expand on these relations in the following way:
		For all $w\in\mathbb{N}$, the partition
		
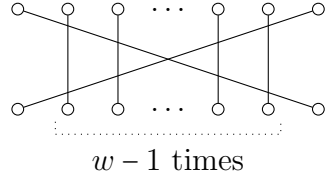
\begin{gather*}
		\begin{tikzpicture}[scale=0.666]
		\node[scale=0.4, circle, draw=black, fill=white] (x1-1) at (0,0) {};
		\node[scale=0.4, circle, draw=black, fill=white] (y1-1) at (1,0) {};
		\node[scale=0.4, circle, draw=black, fill=white] (y2-1) at (2,0) {};
		\node[scale=0.4, circle, draw=black, fill=white] (y3-1) at (4,0) {};
		\node[scale=0.4, circle, draw=black, fill=white] (y4-1) at (5,0) {};
		\node[scale=0.4, circle, draw=black, fill=white] (x2-1) at (6,0) {};			
		\node[scale=0.4, circle, draw=black, fill=white] (x1-2) at (0,2) {};
		\node[scale=0.4, circle, draw=black, fill=white] (y1-2) at (1,2) {};
		\node[scale=0.4, circle, draw=black, fill=white] (y2-2) at (2,2) {};
		\node[scale=0.4, circle, draw=black, fill=white] (y3-2) at (4,2) {};
		\node[scale=0.4, circle, draw=black, fill=white] (y4-2) at (5,2) {};
		\node[scale=0.4, circle, draw=black, fill=white] (x2-2) at (6,2) {};				
		\path (y2-1) edge[draw=none] node {$\ldots$} (y3-1);
		\path (y2-2) edge[draw=none] node {$\ldots$} (y3-2);	
		\draw (x1-1) -- (x2-2);
		\draw (x2-1) -- (x1-2);	
		\draw (y1-1) -- (y1-2);
		\draw (y2-1) -- (y2-2);
		\draw (y3-1) -- (y3-2);
		\draw (y4-1) -- (y4-2);			
		\draw [dotted] (0.75,-0.25) -- ++(0,-0.25) -- (5.25,-0.5) -- ++(0,0.25);
		\node at (3,-1) {$w-1$ times};
		\end{tikzpicture}
		\end{gather*}
		can be seen to generate the category $\mathcal{S}_w$ (\cite[Sect.~8.1]{MaWe18a}). 
                This partition corresponds to the relations $ab_1\ldots b_{w-1}c =cb_1\ldots b_{w-1}a$. Here, we see the half-commutation relations generalized to what might be called a \enquote{paradigm of transpositions}: Only two factors switch places.
The quantum groups $U_n^* \sim\mathcal{I}_{\nnint\backslash\{0,1\}}$ and $U_{n}^\times \sim \mathcal{I}_{\mathbb N_0\backslash\{0\}}$ (in the notation of \cite{BaBi17}) are still following this principle exactly. The category $\mathcal{I}_{\nnint\backslash\{0\}}$ is generated by $\PartHalfLibWBW$. And, likewise, the partition
		
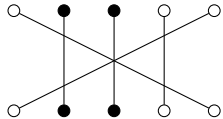
\begin{gather*}
		\begin{tikzpicture}[scale=0.666,baseline=(current  bounding  box.center)]
		\node [circle, scale=0.4, draw=black, fill=white] (x1-1) at (0,0) {};
		\node [circle, scale=0.4, draw=black, fill=black] (y1-1) at (1,0) {};
		\node [circle, scale=0.4, draw=black, fill=black] (y2-1) at (2,0) {};
		\node [circle, scale=0.4, draw=black, fill=white] (y3-1) at (3,0) {};
		\node [circle, scale=0.4, draw=black, fill=white] (x2-1) at (4,0) {};
		\node [circle, scale=0.4, draw=black, fill=white] (x1-2) at (0,2) {};
		\node [circle, scale=0.4, draw=black, fill=black] (y1-2) at (1,2) {};
		\node [circle, scale=0.4, draw=black, fill=black] (y2-2) at (2,2) {};
		\node [circle, scale=0.4, draw=black, fill=white] (y3-2) at (3,2) {};
		\node [circle, scale=0.4, draw=black, fill=white] (x2-2) at (4,2) {};
		\draw (x1-1) -- (x2-2);
		\draw (x2-1) -- (x1-2);
		\draw (y1-1) -- (y1-2);
		\draw (y2-1) -- (y2-2);
		\draw (y3-1) -- (y3-2);		
		\end{tikzpicture}
		\end{gather*}
can be shown to generate $\mathcal{I}_{\nnint\backslash\{0,1\}}$. And one can continue this line of thinking to describe the entire family $(\mathcal{I}_{D})_{D\in\mc D,0\notin D}$, where $\mc D$ denotes the set of all subsemigroups of $(\nnint,+)$. For example, one  can prove that the partition
		
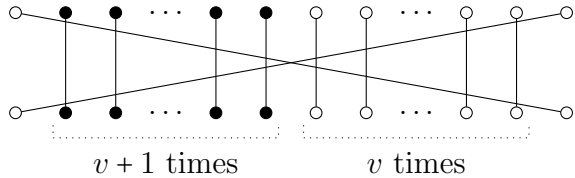
\begin{gather*}
		\begin{tikzpicture}[scale=0.666, baseline=(current  bounding  box.center)]
		\node[scale=0.4, circle, draw=black, fill=white] (x1-1) at (0,0) {};
		\node[scale=0.4, circle, draw=black, fill=black] (y1-1) at (1,0) {};
		\node[scale=0.4, circle, draw=black, fill=black] (y2-1) at (2,0) {};
		\node[scale=0.4, circle, draw=black, fill=black] (y3-1) at (4,0) {};
		\node[scale=0.4, circle, draw=black, fill=black] (y4-1) at (5,0) {};
		\node[scale=0.4, circle, draw=black, fill=white] (z1-1) at (6,0) {};
		\node[scale=0.4, circle, draw=black, fill=white] (z2-1) at (7,0) {};			
		\node[scale=0.4, circle, draw=black, fill=white] (z3-1) at (9,0) {};
		\node[scale=0.4, circle, draw=black, fill=white] (z4-1) at (10,0) {};
		\node[scale=0.4, circle, draw=black, fill=white] (x2-1) at (11,0) {};			
		\node[scale=0.4, circle, draw=black, fill=white] (x1-2) at (0,2) {};
		\node[scale=0.4, circle, draw=black, fill=black] (y1-2) at (1,2) {};
		\node[scale=0.4, circle, draw=black, fill=black] (y2-2) at (2,2) {};
		\node[scale=0.4, circle, draw=black, fill=black] (y3-2) at (4,2) {};
		\node[scale=0.4, circle, draw=black, fill=black] (y4-2) at (5,2) {};
		\node[scale=0.4, circle, draw=black, fill=white] (z1-2) at (6,2) {};
		\node[scale=0.4, circle, draw=black, fill=white] (z2-2) at (7,2) {};			
		\node[scale=0.4, circle, draw=black, fill=white] (z3-2) at (9,2) {};
		\node[scale=0.4, circle, draw=black, fill=white] (z4-2) at (10,2) {};
		\node[scale=0.4, circle, draw=black, fill=white] (x2-2) at (11,2) {};	
		\path (y2-1) edge[draw=none] node {$\ldots$} (y3-1);
		\path (z2-1) edge[draw=none] node {$\ldots$} (z3-1);	
		\path (y2-2) edge[draw=none] node {$\ldots$} (y3-2);
		\path (z2-2) edge[draw=none] node {$\ldots$} (z3-2);				
		\draw (x1-1) -- (x2-2);
		\draw (x2-1) -- (x1-2);	
		\draw (y1-1) -- (y1-2);
		\draw (y2-1) -- (y2-2);
		\draw (y3-1) -- (y3-2);
		\draw (y4-1) -- (y4-2);			
		\draw (z1-1) -- (z1-2);
		\draw (z2-1) -- (z2-2);
		\draw (z3-1) -- (z3-2);
		\draw (z4-1) -- (z4-2);						
		\draw [dotted] (0.75,-0.25) -- ++(0,-0.25) -- (5.25,-0.5) -- ++(0,0.25);
		\draw [dotted, xshift=5cm] (0.75,-0.25) -- ++(0,-0.25) -- (5.25,-0.5) -- ++(0,0.25);			
		\node at (3,-1) {$v+1$ times};
		\node at (8,-1) {$v$ times};			
		\end{tikzpicture}
		\end{gather*}
		generates $\mathcal{I}_{\nnint\backslash\{0,\ldots,v\}}$ for a given arbitrary $v\in\mathbb{N}_0$. 
That means all half-liberations known previously to the present article, including the ones from \cite{MaWe18a}, are subsumed by this transpositional paradigm. \par
But the process of half-liberating $U_n$ on the way to $U_n^+$ is not exhausted by this scheme of \emph{transposing} single pairs of factors. This is first evidenced in this article by the family $(\mc I_{D})_{D\in \mc D,0\in D}$, the categories in the \enquote{monoid case} (see Definiton~\ref{definition:monoid_case_non_monoid_case}).
The category $\mathcal I_{\nnint\backslash\{1,\ldots,v\}}$, for example, is generated by the partition
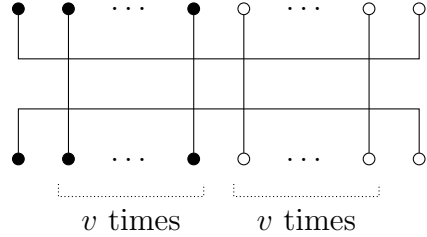
\begin{align*}
  \begin{tikzpicture}[scale=0.666, baseline=-0.111cm]
    \node [circle, scale=0.4, draw=black, fill=black] (x1) at (0,0) {};
    \node [circle, scale=0.4, draw=black, fill=black] (x2) at (1,0) {};
    \node [circle, scale=0.4, draw=black, fill=black] (x3) at (3.5,0) {};
    \node [circle, scale=0.4, draw=black, fill=white] (x4) at (4.5,0) {};    
    \node [circle, scale=0.4, draw=black, fill=white] (x5) at (7,0) {};
    \node [circle, scale=0.4, draw=black, fill=white] (x6) at (8,0) {};
    \node [circle, scale=0.4, draw=black, fill=black] (y1) at (0,3) {};
    \node [circle, scale=0.4, draw=black, fill=black] (y2) at (1,3) {};
    \node [circle, scale=0.4, draw=black, fill=black] (y3) at (3.5,3) {};
    \node [circle, scale=0.4, draw=black, fill=white] (y4) at (4.5,3) {};    
    \node [circle, scale=0.4, draw=black, fill=white] (y5) at (7,3) {};
    \node [circle, scale=0.4, draw=black, fill=white] (y6) at (8,3) {};
    \draw (x1) -- ++(0,1) -| (x6);
    \draw (y1) -- ++(0,-1) -| (y6);
    \draw (x2) -- (y2);
    \draw (x3) -- (y3);
    \draw (x4) -- (y4);
    \draw (x5) -- (y5);
    \draw [densely dotted] (0.8,-0.5) --++(0,-0.25) -- (3.7,-0.75) -- ++ (0,0.25);
    \draw [densely dotted] (4.3,-0.5) --++(0,-0.25) -- (7.2,-0.75) -- ++ (0,0.25);    
    \node at (2.25,-1.2) {$v$ times};
    \node at (5.75,-1.2) {$v$ times};
    \node at (2.25,0) {\ldots};
    \node at (2.25,3) {\ldots};
    \node at (5.75,0) {\ldots};
    \node at (5.75,3) {\ldots};        
  \end{tikzpicture},
\end{align*}
which cannot be written in the form of a transposition. The relations of the associated easy quantum group do \emph{not} take the shape $ac_1^*\ldots c_v^*d_1\ldots d_vb=bc_1^*\ldots c_v^*d_1\ldots d_v a$ as in the case of $\mc I_{\nnint\backslash\{0,\ldots,v\}}$; see also Section~\ref{subsection:C-star-relations}.\par 
In light of this fact, one is tempted to further generalize the paradigm of transpositions to one of \emph{permutations}:
The generator of $\mc I_{\nnint\backslash\{1,\ldots,v\}}$ can at least be expressed in the form of the -- non-transpositional -- permutation
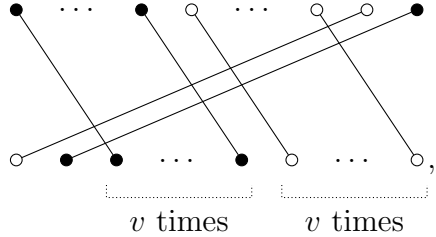
\begin{align*}
  \begin{tikzpicture}[scale=0.666, baseline=-0.111cm]
    \node [circle, scale=0.4, draw=black, fill=white] (x1) at (0,0) {};
    \node [circle, scale=0.4, draw=black, fill=black] (x2) at (1,0) {};
    \node [circle, scale=0.4, draw=black, fill=black] (x3) at (2,0) {};
    \node [circle, scale=0.4, draw=black, fill=black] (x4) at (4.5,0) {};    
    \node [circle, scale=0.4, draw=black, fill=white] (x5) at (5.5,0) {};
    \node [circle, scale=0.4, draw=black, fill=white] (x6) at (8,0) {};
    \node [circle, scale=0.4, draw=black, fill=black] (y1) at (0,3) {};
    \node [circle, scale=0.4, draw=black, fill=black] (y2) at (2.5,3) {};
    \node [circle, scale=0.4, draw=black, fill=white] (y3) at (3.5,3) {};
    \node [circle, scale=0.4, draw=black, fill=white] (y4) at (6,3) {};    
    \node [circle, scale=0.4, draw=black, fill=white] (y5) at (7,3) {};
    \node [circle, scale=0.4, draw=black, fill=black] (y6) at (8,3) {};
    \draw (x1) -- (y5);
    \draw (x2) -- (y6);
    \draw (x3) -- (y1);
    \draw (x4) -- (y2);
    \draw (x5) -- (y3);
    \draw (x6) -- (y4);
    \draw [densely dotted] (1.8,-0.5) --++(0,-0.25) -- (4.7,-0.75) -- ++ (0,0.25);
    \draw [densely dotted] (5.3,-0.5) --++(0,-0.25) -- (8.2,-0.75) -- ++ (0,0.25);    
    \node at (3.25,-1.2) {$v$ times};
    \node at (6.75,-1.2) {$v$ times};
    \node at (3.25,0) {\ldots};
    \node at (1.25,3) {\ldots};
    \node at (6.75,0) {\ldots};
    \node at (4.75,3) {\ldots};     
  \end{tikzpicture},
\end{align*}
meaning that relations of the sort $ab^* c_1^*\ldots c_v^*d_1\ldots d_v=c_1^*\ldots c_vd_1\ldots d_v ab^*$ hold in the corresponding easy quantum group. Still, some kind of partial commutativity, permutability characterizes them compared to $U_n^+$. While this modification works for $\mc I_{\nnint\backslash\{1,\ldots,v\}}$, for general $D\in \mc D$ with $0\in D$, we cannot write the generator of $\mc I_D$ as a permutation: Take $\mc I_{\nnint\backslash\{1,2,5\}}$ as a counterexample:
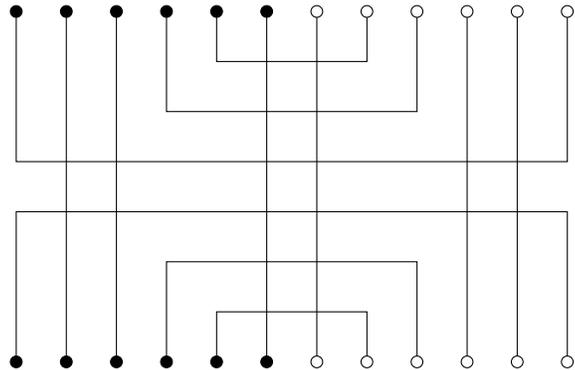
\begin{align*}
  \begin{tikzpicture}[scale=0.666]
    \node [circle, scale=0.4, draw=black, fill=black] (x1) at (0,0) {};
    \node [circle, scale=0.4, draw=black, fill=black] (x2) at (1,0) {};
    \node [circle, scale=0.4, draw=black, fill=black] (x3) at (2,0) {};
    \node [circle, scale=0.4, draw=black, fill=black] (x4) at (3,0) {};
    \node [circle, scale=0.4, draw=black, fill=black] (x5) at (4,0) {};
    \node [circle, scale=0.4, draw=black, fill=black] (x6) at (5,0) {};
    \node [circle, scale=0.4, draw=black, fill=white] (x7) at (6,0) {};
    \node [circle, scale=0.4, draw=black, fill=white] (x8) at (7,0) {};
    \node [circle, scale=0.4, draw=black, fill=white] (x9) at (8,0) {};
    \node [circle, scale=0.4, draw=black, fill=white] (x10) at (9,0) {};
    \node [circle, scale=0.4, draw=black, fill=white] (x11) at (10,0) {};
    \node [circle, scale=0.4, draw=black, fill=white] (x12) at (11,0) {};
    \node [circle, scale=0.4, draw=black, fill=black] (y1) at (0,7) {};
    \node [circle, scale=0.4, draw=black, fill=black] (y2) at (1,7) {};
    \node [circle, scale=0.4, draw=black, fill=black] (y3) at (2,7) {};
    \node [circle, scale=0.4, draw=black, fill=black] (y4) at (3,7) {};
    \node [circle, scale=0.4, draw=black, fill=black] (y5) at (4,7) {};
    \node [circle, scale=0.4, draw=black, fill=black] (y6) at (5,7) {};
    \node [circle, scale=0.4, draw=black, fill=white] (y7) at (6,7) {};
    \node [circle, scale=0.4, draw=black, fill=white] (y8) at (7,7) {};
    \node [circle, scale=0.4, draw=black, fill=white] (y9) at (8,7) {};
    \node [circle, scale=0.4, draw=black, fill=white] (y10) at (9,7) {};
    \node [circle, scale=0.4, draw=black, fill=white] (y11) at (10,7) {};
    \node [circle, scale=0.4, draw=black, fill=white] (y12) at (11,7) {};            
    \draw (x1) -- ++(0,3) -| (x12);
    \draw (y1) -- ++(0,-3) -| (y12);
    \draw (x2) -- (y2);
    \draw (x3) -- (y3);
    \draw (x6) -- (y6);
    \draw (x7) -- (y7);
    \draw (x10) -- (y10);
    \draw (x11) -- (y11);
    \draw (x4) -- ++(0,2) -| (x9);
    \draw (x5) -- ++(0,1) -| (x8);
    \draw (y4) -- ++(0,-2) -| (y9);
    \draw (y5) -- ++(0,-1) -| (y8);    
  \end{tikzpicture}.
\end{align*}
It is impossible to write this generator in a form where all strings connect the two rows. There is no permutation $\pi$ such that one could express the relations in the quantum group as $a_1\ldots a_m=a_{\pi(1)}\ldots a_{\pi(m)}$.
\par
In conclusion, except for special cases, \emph{half-liberation} does not mean \emph{half-commutation} in the unitary case, not even in the sense of non-transpositional permutation.
\par
It seems that none of the shapes
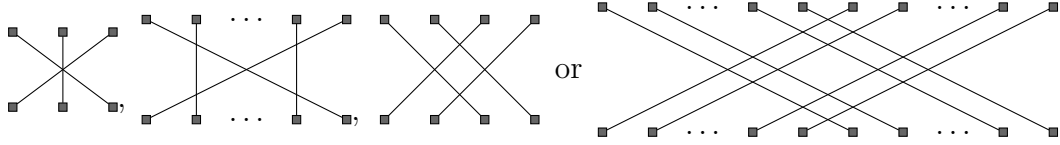
\begin{gather*}
  		\begin{tikzpicture}[scale=0.666,baseline=-0.25cm*0.666]
		\node[scale=0.4, draw=black, fill=gray] (a0) at (0,0) {};
		\node[scale=0.4, draw=black, fill=gray] (a1) at (1,0) {};
		\node[scale=0.4, draw=black, fill=gray] (a2) at (2,0) {};
		\node[scale=0.4, draw=black, fill=gray] (b0) at (0,1.5) {};
		\node[scale=0.4, draw=black, fill=gray] (b1) at (1,1.5) {};
		\node[scale=0.4, draw=black, fill=gray] (b2) at (2,1.5) {};
		\draw (a0) -- (b2);
		\draw (a1) -- (b1);				
		\draw (a2) -- (b0);				
		\end{tikzpicture}\raisebox{0.166cm}{,}\;\;
		\begin{tikzpicture}[scale=0.666, baseline=0]
		\node[scale=0.4, draw=black, fill=gray] (x1-1) at (0,0) {};
		\node[scale=0.4, draw=black, fill=gray] (y1-1) at (1,0) {};
		\node[scale=0.4, draw=black, fill=gray] (y2-1) at (3,0) {};
		\node[scale=0.4, draw=black, fill=gray] (x2-1) at (4,0) {};			
		\node[scale=0.4, draw=black, fill=gray] (x1-2) at (0,2) {};
		\node[scale=0.4, draw=black, fill=gray] (y1-2) at (1,2) {};
		\node[scale=0.4, draw=black, fill=gray] (y2-2) at (3,2) {};
		\node[scale=0.4, draw=black, fill=gray] (x2-2) at (4,2) {};	
		\path (y1-1) edge[draw=none] node {$\ldots$} (y2-1);										
		\path (y1-2) edge[draw=none] node {$\ldots$} (y2-2);														
		\draw (x1-1) -- (x2-2);
		\draw (x1-2) -- (x2-1);				
		\draw (y1-1) -- (y1-2);				
		\draw (y2-1) -- (y2-2);								
              \end{tikzpicture},\;\;
      		\begin{tikzpicture}[scale=0.666,baseline=0]
		\node[scale=0.4, draw=black, fill=gray] (a0) at (0,0) {};
		\node[scale=0.4, draw=black, fill=gray] (a1) at (1,0) {};
		\node[scale=0.4, draw=black, fill=gray] (a2) at (2,0) {};
		\node[scale=0.4, draw=black, fill=gray] (a3) at (3,0) {};                
		\node[scale=0.4, draw=black, fill=gray] (b0) at (0,2) {};
		\node[scale=0.4, draw=black, fill=gray] (b1) at (1,2) {};
		\node[scale=0.4, draw=black, fill=gray] (b2) at (2,2) {};
		\node[scale=0.4, draw=black, fill=gray] (b3) at (3,2) {};                
		\draw (a0) -- (b2);
		\draw (a1) -- (b3);				
		\draw (a2) -- (b0);				
		\draw (a3) -- (b1);				
              \end{tikzpicture}\;\;\raisebox{0.555cm}{\text{or}}\;\;
		\begin{tikzpicture}[scale=0.666, baseline=0.25cm*0.666]
		\node [scale=0.4, draw=black, fill=gray] (x1-1) at (0,0) {};
		\node [scale=0.4, draw=black, fill=gray] (x2-1) at (1,0) {};
		\node [scale=0.4, draw=black, fill=gray] (y1-1) at (3,0) {};
		\node [scale=0.4, draw=black, fill=gray] (y2-1) at (4,0) {};
		\node [scale=0.4, draw=black, fill=gray] (y4-1) at (5,0) {};
		\node [scale=0.4, draw=black, fill=gray] (y5-1) at (6,0) {};
		\node [scale=0.4, draw=black, fill=gray] (z1-1) at (8,0) {};
		\node [scale=0.4, draw=black, fill=gray] (z2-1) at (9,0) {};
		\begin{scope}[yshift=2.5cm]			
		\node [scale=0.4, draw=black, fill=gray] (x1-2) at (0,0) {};
		\node [scale=0.4, draw=black, fill=gray] (x2-2) at (1,0) {};
		\node [scale=0.4, draw=black, fill=gray] (y1-2) at (3,0) {};
		\node [scale=0.4, draw=black, fill=gray] (y2-2) at (4,0) {};
		\node [scale=0.4, draw=black, fill=gray] (y4-2) at (5,0) {};
		\node [scale=0.4, draw=black, fill=gray] (y5-2) at (6,0) {};
		\node [scale=0.4, draw=black, fill=gray] (z1-2) at (8,0) {};
		\node [scale=0.4, draw=black, fill=gray] (z2-2) at (9,0) {};
		\end{scope}						
		\path (x2-1) edge[draw=none] node {$\ldots$} (y1-1);
		\path (y5-1) edge[draw=none] node {$\ldots$} (z1-1);
		\path (x2-2) edge[draw=none] node {$\ldots$} (y1-2);
		\path (y5-2) edge[draw=none] node {$\ldots$} (z1-2);
		\draw (x1-1) -- (y4-2);
		\draw (x2-1) -- (y5-2);
		\draw (x1-2) -- (y4-1);
		\draw (x2-2) -- (y5-1);
		\draw (z1-1) -- (y1-2);
		\draw (z2-1) -- (y2-2);
		\draw (z1-2) -- (y1-1);
		\draw (z2-2) -- (y2-1);
			\end{tikzpicture}
		\end{gather*}
		capture the true spirit of half-liberation in the unitary case. Rather, it is the bracket structure

                
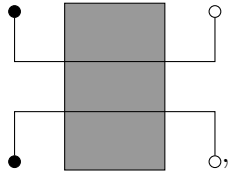
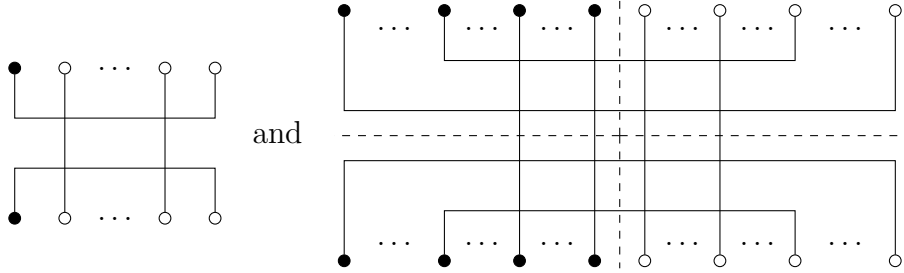
\begin{gather*}
                	\begin{tikzpicture}[scale=0.666,baseline=0]
	\node [scale=0.4, fill=black, draw=black,circle] (x1) at (0,0) {};
	\node [scale=0.4, fill=white, draw=black,circle] (x2) at (4,0) {};		
	\node [scale=0.4, fill=black, draw=black,circle] (y1) at (0,3) {};
	\node [scale=0.4, fill=white, draw=black,circle] (y2) at (4,3) {};				
	\draw [fill=lightgray] (1,-0.166) rectangle (3,3.166);
	\draw (x1) -- ++ (0,1) -| (x2);
	\draw (y1) -- ++ (0,-1) -| (y2);
	\end{tikzpicture},  
      \end{gather*}
      more precisely,
		\begin{gather*}
		\begin{tikzpicture}[scale=0.666,baseline=1.5cm*0.666]
		\node[circle, scale=0.4, draw=black, fill=black] (x1-1) at (0,0) {};
		\node[circle, scale=0.4, draw=black, fill=white] (y1-1) at (1,0) {};
		\node[circle, scale=0.4, draw=black, fill=white] (y2-1) at (3,0) {};
		\node[circle, scale=0.4, draw=black, fill=white] (x2-1) at (4,0) {};			
		\node[circle, scale=0.4, draw=black, fill=black] (x1-2) at (0,3) {};
		\node[circle, scale=0.4, draw=black, fill=white] (y1-2) at (1,3) {};
		\node[circle, scale=0.4, draw=black, fill=white] (y2-2) at (3,3) {};
		\node[circle, scale=0.4, draw=black, fill=white] (x2-2) at (4,3) {};	
		\path (y1-1) edge[draw=none] node {$\ldots$} (y2-1);										
		\path (y1-2) edge[draw=none] node {$\ldots$} (y2-2);														
		\draw (x1-1) -- ++(0,1)-| (x2-1);
		\draw (x1-2) -- ++(0,-1)-| (x2-2);
		\draw (y1-1) -- (y1-2);				
		\draw (y2-1) -- (y2-2);								
              \end{tikzpicture}
              \quad \raisebox{0cm}{\text{and}} \quad
                                    \begin{tikzpicture}[scale=0.666,baseline=2.35cm*0.666]
      \node[circle, scale=0.4, draw=black, fill=black] (x1) at (0,0) {};
      \node[circle, scale=0.4, draw=black, fill=white] (x2) at (11,0) {};
      \node[circle, scale=0.4, draw=black, fill=black] (x3) at (0,5) {};
      \node[circle, scale=0.4, draw=black, fill=white] (x4) at (11,5) {};
      \node[circle, scale=0.4, draw=black, fill=black] (y1) at (5,0) {};
      \node[circle, scale=0.4, draw=black, fill=white] (y2) at (6,0) {};
      \node[circle, scale=0.4, draw=black, fill=black] (y3) at (5,5) {};
      \node[circle, scale=0.4, draw=black, fill=white] (y4) at (6,5) {};
      \node[circle, scale=0.4, draw=black, fill=black] (z1) at (2,0) {};
      \node[circle, scale=0.4, draw=black, fill=white] (z2) at (9,0) {};
      \node[circle, scale=0.4, draw=black, fill=black] (z3) at (2,5) {};
      \node[circle, scale=0.4, draw=black, fill=white] (z4) at (9,5) {};
      \node[circle, scale=0.4, draw=black, fill=black] (w1) at (3.5,0) {};
      \node[circle, scale=0.4, draw=black, fill=white] (w2) at (7.5,0) {};
      \node[circle, scale=0.4, draw=black, fill=black] (w3) at (3.5,5) {};
      \node[circle, scale=0.4, draw=black, fill=white] (w4) at (7.5,5) {};                  
      \draw (x1) -- ++(0,2) -| (x2);
      \draw (x3) -- ++(0,-2) -| (x4);
      \draw (y1) to (y3);
      \draw (y2) to (y4);
      \draw (z1) -- ++(0,1) -| (z2);
      \draw (z3) -- ++(0,-1) -| (z4);
      \draw (w1) to (w3);
      \draw (w2) to (w4);
      \draw[dashed] (5.5,2.5) -- (5.5,-0.15);
      \draw[dashed] (5.5,2.5) -- (5.5,5.2);
      \draw[dashed] (5.5,2.5) -- (-0.2,2.5);
      \draw[dashed] (5.5,2.5) -- (11.2,2.5);      
      \path (x1)  -- node [pos=0.5, above=2pt] {$\ldots$} (z1);
      \path (z1)  -- node [pos=0.5, above=2pt] {$\ldots$} (w1);
      \path (w1)  -- node [pos=0.5, above=2pt] {$\ldots$} (y1);
      \path (x3)  -- node [pos=0.5, below=2pt] {$\ldots$} (z3);
      \path (z3)  -- node [pos=0.5, below=2pt] {$\ldots$} (w3);
      \path (w3)  -- node [pos=0.5, below=2pt] {$\ldots$} (y3);
      \path (x2)  -- node [pos=0.5, above=2pt] {$\ldots$} (z2);
      \path (z2)  -- node [pos=0.5, above=2pt] {$\ldots$} (w2);
      \path (w2)  -- node [pos=0.5, above=2pt] {$\ldots$} (y2);
      \path (x4)  -- node [pos=0.5, below=2pt] {$\ldots$} (z4);
      \path (z4)  -- node [pos=0.5, below=2pt] {$\ldots$} (w4);
      \path (w4)  -- node [pos=0.5, below=2pt] {$\ldots$} (y4);
    \end{tikzpicture}
		\end{gather*}
		that one should consider. So, heuristically, the generator
		
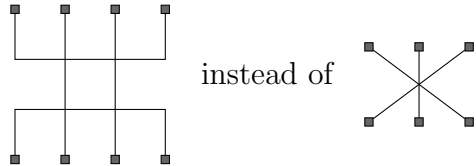
\begin{gather*}
		\begin{tikzpicture}[scale=0.666, baseline=(current  bounding  box.center)]
		\node[scale=0.4, draw=black, fill=gray] (x1-1) at (0,0) {};
		\node[scale=0.4, draw=black, fill=gray] (y1-1) at (1,0) {};
		\node[scale=0.4, draw=black, fill=gray] (y2-1) at (2,0) {};
		\node[scale=0.4, draw=black, fill=gray] (x2-1) at (3,0) {};			
		\node[scale=0.4, draw=black, fill=gray] (x1-2) at (0,3) {};
		\node[scale=0.4, draw=black, fill=gray] (y1-2) at (1,3) {};
		\node[scale=0.4, draw=black, fill=gray] (y2-2) at (2,3) {};
		\node[scale=0.4, draw=black, fill=gray] (x2-2) at (3,3) {};	
		\draw (x1-1) -- ++(0,1)-| (x2-1);
		\draw (x1-2) -- ++(0,-1)-| (x2-2);
		\draw (y1-1) -- (y1-2);				
		\draw (y2-1) -- (y2-2);								
		\end{tikzpicture}
		\quad\text{instead of}\quad
		\begin{tikzpicture}[scale=0.666, baseline=(current  bounding  box.center)]
		\node[scale=0.4, draw=black, fill=gray] (x1-1) at (0,0) {};
		\node[scale=0.4, draw=black, fill=gray] (x2-1) at (1,0) {};
		\node[scale=0.4, draw=black, fill=gray] (x3-1) at (2,0) {};
		\node[scale=0.4, draw=black, fill=gray] (x1-2) at (0,1.5) {};			
		\node[scale=0.4, draw=black, fill=gray] (x2-2) at (1,1.5) {};
		\node[scale=0.4, draw=black, fill=gray] (x3-2) at (2,1.5) {};
		\draw (x1-1) -- (x3-2);				
		\draw (x2-1) -- (x2-2);				
		\draw (x3-1) -- (x1-2);								
		\end{tikzpicture}
		\end{gather*}
		is the right one to import from the orthogonal case and an appropriate way to view half-liberation as presented in this article.

                \subsection{$C^*$-Algebraic Relations}
                \label{subsection:C-star-relations}
                It is straightforward to identify the $C^*$-algebras of the quantum groups associated with the categories $(\mc S_w)_{w\in \nnint}$ and $(\mc I_D)_{D\in \mc D}$, see \cite{TaWe15b} and \cite{We17b,We17c}. For the convenience of the reader, we list the $C^*$-algebraic relations corresponding to the generators of the categories $(\mc S_w)_{w\in \nnint}$ and $(\mc I_D)_{D\in \mc D}$:
                \par
                \begin{enumerate}[wide]
                \item 

                The partition $\PartHalfLibWBW$ induces the following relations: For all $a,b,c\in \{u_{i,j}\}_{i,j=1}^n$
                \begin{align*}
                  ab^*c=cb^*a.
                \end{align*}
              \item
                For $w\in \pint$, the following relations are imposed on the algebra of the quantum group with category $\mc S_w$ by the generator $\PartBracketBWwW$ or rather its rotation $\PartPermWWwW$ (see Section~\ref{subsection:new-perspective}): For all $a,b_1,\ldots,b_{w-1},c\in \{u_{i,j}\}_{i,j=1}^n$
                \begin{align}
                  \label{eq:relations-1}
                  ab_1\ldots b_{w-1}c=cb_1\ldots b_{w-1}a.
                \end{align}
              \item For the quantum group whose intertwiner spaces are generated by the category $\mc S_0=\mc I_{\emptyset}$ with the generators $\{\mr{Br}_\bullet(\{v\})\mid v\in \pint\}$ and $\PartHalfLibWBW$ (or rather their combination $\PartPermWBvWvW$) we can express the fundamental relations as follows: For all $a,b,c_1,\ldots,c_v,d_1,\ldots,d_v\in \{u_{i,j}\}_{i,j=1}^n$ 
                \begin{align*}
                  ab^*c_1^*\ldots c_v^*d_1\ldots d_v=c_1^\ast\ldots c_v^\ast d_1\ldots d_vab^*.
                \end{align*}
                in addition to the relations induced by  $\PartHalfLibWBW$.
\item For general subsemigroups $D$ of $(\nnint,+)$ the relations of the quantum group with associated category $\mc I_D$ cannot be expressed as compactly. \par
  We adopt the convention $-0\neq 0$ and \[-v<-(v-1)<\ldots<-1<-0<0<1<\ldots<v-1<v\] and define $I_v\eqpd \{-v,\ldots,-1,-0,0,1,\ldots,v\}$ for all $v\in \pint$. Moreover, for all $m\in \pint$ abbreviate $\nnint(m)\eqpd \{0,\ldots,m\}$ and $\pint(m)\eqpd \{1,\ldots,m\}$. Then, the relations induced by $\mr{Br}_\bullet(\nnint(m)\backslash D)$ for a submonoid $D$ of $(\nnint,+)$, i.e.\ $0\in D$, and $v\in \pint$ with $v\notin D$ are the following: For all $\alpha,\beta:I_v\to \pint(n)$
  \begin{align}
    \label{eq:relations-2}
    \sum_{\gamma: I_v \to  \pint(n)}\left(\prod_{\substack{j\in \nnint(v)\\j\notin D}}\delta_{\gamma_{-j},\gamma_{j}}\delta_{\beta_{-j},\beta_j}\right)\left(\prod_{\substack{i\in \nnint(v)\\i\in D}}\delta_{\gamma_{-i},\beta_{-i}}\delta_{\gamma_{i},\beta_i}\right)\, u_{\gamma_{-v},\alpha_{-v}}^*\ldots u_{\gamma_{-0},\alpha_{-0}}^*u_{\gamma_{0},\alpha_{0}}\ldots u_{\gamma_{v},\alpha_{v}}\\
\notag    =\sum_{\gamma':I_v \to \pint(n) }\left(\prod_{\substack{j\in \nnint(v)\\j\notin D}}\delta_{\gamma'_{-j},\gamma'_{j}}\delta_{\alpha_{-j},\alpha_j}\right)\left(\prod_{\substack{i\in \nnint(v)\\i\in D}}\delta_{\gamma'_{-i},\alpha_{-i}}\delta_{\gamma'_{i},\alpha_i}\right)\, u_{\beta_{-v},\gamma'_{-v}}^*\ldots u_{\beta_{-0},\gamma'_{-0}}^*u_{\beta_{0},\gamma'_{0}}\ldots u_{\beta_{v},\gamma'_{v}}.
  \end{align}
  For example, the partition $\mr{Br}_\bullet(\{1,2,5\})$ 
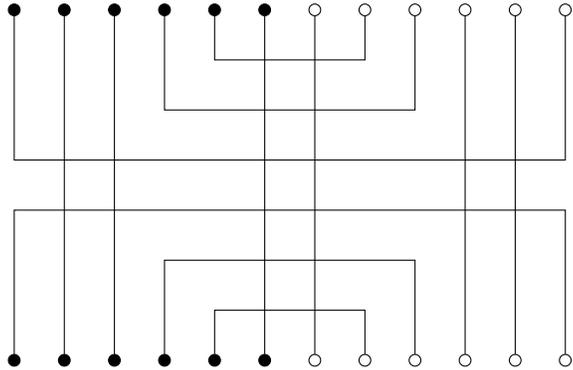
\begin{align*}
  \begin{tikzpicture}[scale=0.666]
    \node [circle, scale=0.4, draw=black, fill=black] (x1) at (0,0) {};
    \node [circle, scale=0.4, draw=black, fill=black] (x2) at (1,0) {};
    \node [circle, scale=0.4, draw=black, fill=black] (x3) at (2,0) {};
    \node [circle, scale=0.4, draw=black, fill=black] (x4) at (3,0) {};
    \node [circle, scale=0.4, draw=black, fill=black] (x5) at (4,0) {};
    \node [circle, scale=0.4, draw=black, fill=black] (x6) at (5,0) {};
    \node [circle, scale=0.4, draw=black, fill=white] (x7) at (6,0) {};
    \node [circle, scale=0.4, draw=black, fill=white] (x8) at (7,0) {};
    \node [circle, scale=0.4, draw=black, fill=white] (x9) at (8,0) {};
    \node [circle, scale=0.4, draw=black, fill=white] (x10) at (9,0) {};
    \node [circle, scale=0.4, draw=black, fill=white] (x11) at (10,0) {};
    \node [circle, scale=0.4, draw=black, fill=white] (x12) at (11,0) {};
    \node [circle, scale=0.4, draw=black, fill=black] (y1) at (0,7) {};
    \node [circle, scale=0.4, draw=black, fill=black] (y2) at (1,7) {};
    \node [circle, scale=0.4, draw=black, fill=black] (y3) at (2,7) {};
    \node [circle, scale=0.4, draw=black, fill=black] (y4) at (3,7) {};
    \node [circle, scale=0.4, draw=black, fill=black] (y5) at (4,7) {};
    \node [circle, scale=0.4, draw=black, fill=black] (y6) at (5,7) {};
    \node [circle, scale=0.4, draw=black, fill=white] (y7) at (6,7) {};
    \node [circle, scale=0.4, draw=black, fill=white] (y8) at (7,7) {};
    \node [circle, scale=0.4, draw=black, fill=white] (y9) at (8,7) {};
    \node [circle, scale=0.4, draw=black, fill=white] (y10) at (9,7) {};
    \node [circle, scale=0.4, draw=black, fill=white] (y11) at (10,7) {};
    \node [circle, scale=0.4, draw=black, fill=white] (y12) at (11,7) {};            
    \draw (x1) -- ++(0,3) -| (x12);
    \draw (y1) -- ++(0,-3) -| (y12);
    \draw (x2) -- (y2);
    \draw (x3) -- (y3);
    \draw (x6) -- (y6);
    \draw (x7) -- (y7);
    \draw (x10) -- (y10);
    \draw (x11) -- (y11);
    \draw (x4) -- ++(0,2) -| (x9);
    \draw (x5) -- ++(0,1) -| (x8);
    \draw (y4) -- ++(0,-2) -| (y9);
    \draw (y5) -- ++(0,-1) -| (y8);    
  \end{tikzpicture}.
\end{align*}
induces the following relations: For all $a_1,a_2,a_3,b_1,b_2,b_3\in \{u_{i,j}\}_{i,j=1}^n$ and all $i,j,i',j':\pint(3)\to \pint(n)$
\begin{align*}
  &\delta_{i_1',j_1'}\delta_{i_2',j_2'}\delta_{i_3',j_3'}\sum_{k_1,k_2,k_3=1}^nu_{k_1,i_1}^\ast a_1^\ast a_2^\ast u_{k_2,i_2}^\ast u_{k_3,i_3}^\ast a_3^\ast b_3u_{k_3,j_3}u_{k_2,j_2}b_2b_1u_{k_1,j_1}\\
  &=\delta_{i_1,j_1}\delta_{i_2,j_2}\delta_{i_3,j_3}\sum_{k_1',k_2',k_3'=1}^nu_{i_1',k_1'}^\ast a_1^\ast a_2^\ast u_{i_2',k_2'}^\ast u_{i_3',k_3'}^\ast a_3^\ast b_3u_{j_3',k_3'}u_{j_2',k_2'}b_2b_1u_{j_1',k_1'}.
\end{align*}
                 \end{enumerate}
                \par
                Put $u\eqpd (u_{i,j})_{i,j=1}^n$ and $\overline u\eqpd (u_{i,j}^\ast)_{i,j=1}^n$. Note that, although we hence know that $\mc S_w$ corresponds to the quantum group with $C^*$-algebra
                \begin{align*}
                  C^*(\{u_{ij}\}_{i,j=1}^n\mid u,\overline u \text{ unitary}, \forall a,b_1,\ldots,b_{w-1},c\in \{u_{i,j}\}_{i,j=1}^n: \text{relations \eqref{eq:relations-1} hold} ),
                \end{align*}
                while $\mc I_D$ corresponds to
                \begin{gather*}
                  C^*(\{u_{ij}\}_{i,j=1}^n\mid u,\overline u \text{ unitary}, \forall v\in \pint\backslash D:\forall \alpha,\beta: I_v\to \pint(n): \text{relations \eqref{eq:relations-2} hold}),
                \end{gather*}
                we know basically nothing about these quantum groups. In particular, it would be enlightening to construct these quantum groups from known ones, if possible, in the sense of \cite{TaWe15b}.
                
\subsection{Further Questions}
\begin{enumerate}[wide, labelwidth= !]
  \item With all easy quantum groups $G$ classified in the regions $U_n\subseteq G\subseteq U_n^+$ (present article and \cite{MaWe18a}) and $O_n\subseteq G\subseteq O_n^+$ (\cite{BaSp09}) as well as $O_n\subseteq G\subseteq U_n$ (\cite{TaWe15a}) and $O_n^+\subseteq G\subseteq U_n^+$ (\cite{TaWe15a}), a natural question is to find all unitary easy quantum groups $G$ with $O_n\subseteq G\subseteq U_n^+$. This is an ongoing project of the authors.
\item In general, the complete classification of categories of two-colored partitions is still an open question (see also \cite{TaWe15a}; see \cite{RaWe13} for the complete classification in the orthogonal case). The question of all such categories with $\PartIdenLoBB\otimes \PartIdenLoWW\in \mathcal C$ has recently been settled by Gromada  in \cite{Gr18}.
    \item In the orthogonal case, $O_n^\ast$ is not only the only easy quantum group $G$ such that $O_n\subseteq G\subseteq O_n^+$, but also the only compact quantum group in the region $O_n\subseteq G\subseteq O_n^*$ (\cite{BaBiCoCu11}). One wonders whether there exist  further compact quantum groups $G$ with $U_n\subseteq G\subseteq U_n^+$ besides the unitary easy quantum groups. 
	\item Our results might yield clues to advancing with the classification of bi-easy (\enquote{busy}) geometries begun in \cite{Ba17a}. Also, the study of affine homogeneous spaces of the free complex sphere (see \cite{Ba17b}) might benefit from the results in the present article.
	\item In \cite{Fr14a}, \cite{Fr14b} and \cite{Fr17}, Freslon investigates a different notion of colored partitions. And some of the results of \cite{MaWe18a} and this article, especially \cite[Lemma.~6.7]{MaWe18a}, generalize to his setting. One could try to apply the same methods to classify all the categories of this kind.
\end{enumerate}

\printbibliography

\end{document}